%% file: main.tex
\documentclass[11pt]{article}
\usepackage{amssymb}
\usepackage{amsmath}
\usepackage{tikz}
\usetikzlibrary{decorations.pathreplacing}
\usetikzlibrary{positioning}
\usetikzlibrary{arrows}
\usepackage[absolute,overlay]{textpos}
\usepackage{mathtools}
\usepackage{mathrsfs}
\usepackage{graphicx}
\usepackage[utf8]{inputenc}
\usepackage{setspace}
\usepackage{abstract}
\usepackage[english]{babel}
\usepackage{amsthm}
\usepackage{keytheorems}
\usepackage[center]{titlesec}
\usepackage{mathtools}
\usepackage{float}
\usepackage{fancyhdr}
\usepackage{soul}
\usepackage[a4paper, left=1.9cm, right=1.9cm, top=1.9cm, bottom=1.9cm,footskip=0.5in]{geometry} %
\usepackage{enumitem}
\usepackage{tikz-cd}
\usetikzlibrary{fadings}
\usetikzlibrary{patterns}
\usetikzlibrary{shadows.blur}
\usetikzlibrary{shapes}
\usepackage{pst-solides3d}

\usepackage{comment}
\usepackage{stmaryrd}

\usepackage{quiver} 

\usepackage[sc]{mathpazo}

\usepackage{proof}

\usepackage[colorlinks=true,allcolors=teal]{hyperref}

\newcommand{\myuline}[1]{%
  \mspace{2mu}%
  \underline{\smash{\mspace{-2mu}#1\mspace{-2mu}}}%
  \mspace{2mu}%
}

\DeclareMathAlphabet{\mathpzc}{OT1}{pzc}{m}{it}

\usepackage[style=alphabetic,doi=false,isbn=false,backend=bibtex]{biblatex}

\AtEveryBibitem{%
  \ifentrytype{misc}
  {%
  }{%
     \clearfield{url}%
    \clearfield{urldate}%
    \clearfield{urlyear}%

  }%
  \clearlist{language}%
  \clearfield{month}%
  \clearfield{note}
  \clearfield{urldate}
  \clearfield{urlyear}%
}



\AtBeginEnvironment{quote}{\small}

\newcommand{\cC}{\mathcal{C}}
\newcommand{\cD}{\mathcal{D}}

\newcommand{\cO}{\mathcal{O}}

\newcommand{\cCi}{\mathscr{C}}

\newcommand{\Set}{\mathbf{Set}}
\newcommand{\Cat}{\mathbf{Cat}}

\newcommand{\T}{\mathbb{T}}

\newcommand{\uHom}{\myuline{\mathsf{Hom}}}
\newcommand{\uMap}{\myuline{\mathpzc{Map}}}
\newcommand{\Map}{\mathpzc{Map}}
\newcommand{\Cocone}{\myuline{\mathsf{Cocone}}}

\newtheorem*{theorem*}{Theorem}

\newtheorem{theorem}{Theorem}[section]
\newtheorem{corollary}[theorem]{Corollary}
\newtheorem{lemma}[theorem]{Lemma}
\newtheorem{proposition}[theorem]{Proposition}

\theoremstyle{definition}
\newtheorem{construction}[theorem]{Construction}

\newtheorem{remark}[theorem]{Remark}
\newtheorem{definition}[theorem]{Definition}

        \titleclass{\subsubsubsection}{straight}[\subsection]
        
        \newcounter{subsubsubsection}[subsubsection]
        \renewcommand\thesubsubsubsection{\thesubsubsection.\arabic{subsubsubsection}}

        \titleformat{\subsubsubsection}
          {\normalfont\normalsize\bfseries}{\thesubsubsubsection}{1em}{}
        \titlespacing*{\subsubsubsection}
        {0pt}{3.25ex plus 1ex minus .2ex}{1.5ex plus .2ex}
        
        \makeatletter
        \renewcommand\paragraph{\@startsection{paragraph}{5}{\z@}%
          {3.25ex \@plus1ex \@minus.2ex}%
          {-1em}%
          {\normalfont\normalsize\bfseries}}
        \renewcommand\subparagraph{\@startsection{subparagraph}{6}{\parindent}%
          {3.25ex \@plus1ex \@minus .2ex}%
          {-1em}%
          {\normalfont\normalsize\bfseries}}
        \def\toclevel@subsubsubsection{4}
        \def\toclevel@paragraph{5}
        \def\toclevel@paragraph{6}
        \def\l@subsubsubsection{\@dottedtocline{4}{7em}{4em}}
        \def\l@paragraph{\@dottedtocline{5}{10em}{5em}}
        \def\l@subparagraph{\@dottedtocline{6}{14em}{6em}}
        \makeatother
        
\setlist[description]{leftmargin=\parindent,labelindent=\parindent,rightmargin=\parindent}

\title{\vspace{-0.5cm} Elementary $\infty$-toposes from type theory}
\author{Maximilian Petrowitsch \\ \small Western University \\ \small \url{mpetrowi@uwo.ca}}
\author{Daniël Apol\\ \small University of Gothenburg \\ \small and Chalmers University of Technology \\ \small \url{daniel.apol@chalmers.se }\and Maximilian Petrowitsch\\\small Western University\\ \small \url{mpetrowi@uwo.ca}}
\date{\today}

\newif\ifworkinprogress
\workinprogresstrue %

\tikzcdset{scale cd/.style={every label/.append style={scale=#1},
    cells={nodes={scale=#1}}}}

\begin{document}

\maketitle

\vspace{-1cm}

\begin{abstract}
\input{abstract}
\end{abstract}

\input{sect0-intro}

\tableofcontents

\input{sect1-preliminaries}
\input{sect2-univalent-universes-in-TT}

\input{sect3-univalent-tribes}
\input{sect4-univalent-infinity-categories}
\input{sect5-elementary-infty-toposes}

\printbibliography

\end{document}

%% file: abstract.tex
We prove that every categorical model of dependent type theory with dependent sums and products, intensional identity types and univalent universes presents via its $\infty$-localisation an elementary $\infty$-topos, that is, a finitely complete, locally cartesian closed $\infty$-category with enough univalent universal morphisms. We also show that elementary $\infty$-toposes have small subobject classifiers and that the $\infty$-localisation admits finite colimits if the type theory has 0-types and pushout types. To achieve this, we extend Joyal's theory of tribes by introducing the notion of a univalent tribe and a univalent fibration in a tribe.

%% file: sect0-intro.tex
\section{Introduction}

Since its conception, it has been speculated (see e.g. \cite{shulman_internal_2012}) that Homotopy Type Theory (HoTT) \cite{the_univalent_foundations_program_homotopy_2013} is the internal language of particular higher categories also called \textit{elementary $\infty$-toposes}. A precise formulation of this statement was given by Kapulkin and Lumsdaine in \autocite[Conj. 3.7]{kapulkin_homotopy_2018} and is known as the \textit{internal language conjecture}. Here, HoTT is understood as Martin-Löf dependent type theory with dependent sums, dependent products, intensional identity types and univalent universes.

So far, several important steps have been made towards establishing such a connection between higher categories and type theory. Kapulkin and Szumiło \cite{kapulkin_internal_2019} showed that dependent type theory with intensional identity types is the internal language of finitely complete $\infty$-categories. Assuming in addition dependent products, Kapulkin \cite{kapulkin_locally_2017} proved that every model of such a type theory presents a locally cartesian closed $\infty$-category. Shulman \cite{shulman_all_2019} showed that HoTT can be interpreted as internal language into any Grothendieck $\infty$-topos. The relation between $\infty$-categories and type theories has been further explored by Nguyen and Uemura \cite{Nguyen_uemura_2025}, where they established a correspondence between certain $\infty$-categories and $\infty$-type theories, a notion that generalises ordinary dependent type theories.

Proving a correspondence between HoTT and elementary $\infty$-toposes is currently an open problem. Elementary $\infty$-toposes are supposed to generalise both Grothendieck $\infty$-toposes as studied by Lurie \cite{lurie_higher_2009} and ordinary 1-toposes, whose internal language is a version of intuitionistic higher-order logic \cite{mac_lane_sheaves_1994}. There is currently no generally agreed upon definition of an elementary $\infty$-topos, even though there has been increasing interest in the topic in the last decade. Two important proposals for a definition were made by Shulman in \cite{shulman_towards_2018} and by Rasekh in \cite{rasekh_theory_2022}. Rasekh also proved from his definition some expected topos-theoretic properties such as descent and locality. Moreover, Gepner and Kock \cite{gepner_univalence_2017} proved several important results about univalent morphisms in presentable locally cartesian closed $\infty$-categories, some of which can also be applied in the non-presentable case of elementary $\infty$-toposes. 

We define in Definition \ref{def: elementary infinity topos} an elementary $\infty$-topos as a finitely complete, locally cartesian closed $\infty$-category with enough univalent morphisms. The main result of this paper is the following:

\begin{theorem*}[Theorem \ref{thm: categorical models present infty toposes}]
   Let $\cC$ be a categorical model of a dependent type theory with $\Sigma, \Pi$ and $\mathtt{Id}$-types, satisfying the $\Pi$-$\eta$ rule and function extensionality with enough univalent universes which are closed under all type constructors. Then, its $\infty$-localisation $\cC_{\infty}$ is an elementary $\infty$-topos.
\end{theorem*}

This result is a contribution to the internal language conjecture, which, in its precise statement, asserts that there is a diagram of functors
\[\begin{tikzcd}
	{\mathbf{CxlCat}_{\mathtt{HoTT}}} && {\mathbf{ElTop}_{\infty}} \\
	{\mathbf{CxlCat}_{\mathtt{Id},1,\Sigma,\Pi}} && {\mathbf{LCCC}_{\infty}} \\
	{\mathbf{CxlCat}_{\mathtt{Id},1,\Sigma}} && {\mathbf{Lex}_{\infty}}
	\arrow["{Cl_{\infty}^{\mathtt{HoTT}}}"', dashed, from=1-1, to=1-3]
	\arrow[from=1-1, to=2-1]
	\arrow[hook, from=1-3, to=2-3]
	\arrow["{Cl_{\infty}^{\mathtt{Id},1,\Sigma,\Pi}}"', from=2-1, to=2-3]
	\arrow[from=2-1, to=3-1]
	\arrow[hook, from=2-3, to=3-3]
	\arrow["{Cl_{\infty}^{\mathtt{Id},1,\Sigma}\sim}"', from=3-1, to=3-3]
\end{tikzcd}\]
in which the horizontal functors are Dwyer-Kan (DK) equivalences. The categories on the left are the categories of models of the respective dependent type theories. The categories on the right are the categories of finitely complete $\infty$-categories, locally cartesian closed $\infty$-categories and the category of elementary $\infty$-toposes.

The horizontal functors take a categorical model, viewed as a category with weak equivalences, where the weak equivalences are obtained from the structure induced by the identity types, and turn it into an $\infty$-category inverting all weak equivalences. We will call this the $\infty$-\textit{localisation}.

The horizontal bottom functor was shown to be a DK equivalence by Kapulkin and Szumiło in \autocite{kapulkin_internal_2019}. Kapulkin \autocite{kapulkin_locally_2017} showed moreover that the horizontal functor in the middle can be obtained as the $\infty$-localisation, but it is currently an open problem to show that it is a DK equivalence. Regarding the top horizontal functor, it is currently even unknown whether the $\infty$-localisation takes values in $\mathbf{ElTop}_{\infty}$, and it is the purpose of this paper to prove that this is indeed the case.

To obtain the desired horizontal functor on the top, we will use a construction of the $\infty$-localisation that is due to Cisinksi and has been applied to models of type theory for example in \cite{Nguyen_uemura_2025}. It is based on his notion of an $\infty$-category of fibrant objects and its localisation via Cisinski's right calculus of fractions, and is developed in \autocite[§7]{cisinski_higher_2019}. There are equivalent alternatives for the construction of the $\infty$-localisation (See e.g. \cite[\nopp 1.6]{Barwick2016-pr}). For example, one can take the hammock localisation followed by taking a fibrant replacement and then take the homotopy coherent nerve. Alternatively, one could use Szumiło's functor $N_f$, which turns a fibration category into a quasicategory of frames \cite[§3.1]{szumilo_two_2014}.

The advantage of Cisinski's approach is that one gets directly a localisation functor and, thanks to Cisinski's work in \autocite[§7]{cisinski_higher_2019}, this $\infty$-localisation commutes with most of the relevant structure in the original category. This allows deriving the result straightforwardly once everything is set up correctly. To do that, we will use the fact that every categorical model $\cC$ of HoTT has the structure of a tribe in the sense of Joyal as defined in \cite{joyal_notes_2017} and every tribe is a fibration category in the sense of Ken Brown as defined in \cite{brown_abstract_1973}. Thus, the nerve $N(\cC)$ is canonically an $\infty$-category of fibrant objects in the sense of \autocite{cisinski_higher_2019}. Cisinski showed that for a fibration category $\cC$ the localisation $\cC_\infty$ of $N(\cC)$ at the weak equivalences is in particular finitely complete and locally cartesian closed, provided $\cC$ carries the necessary structure, and it induces a nicely behaving localisation functor $\gamma \colon N(\cC) \to \cC_\infty$ that commutes with finite limits and the right adjoint to pullback.

Our approach is then to show that a categorical model of HoTT carries the structure of what we will call a \textit{univalent tribe}. That is, a tribe with sufficiently many univalent fibrations, meaning that every fibration is a homotopy unique pullback of a univalent fibration. For this, we will represent the set of equivalences between pullbacks of some fibration by an object of equivalences in the tribe and define univalence as a condition on that object. This notion of univalent tribe is an extension of Joyal's theory of tribes, which we build on and which is developed extensively in \cite{joyal_notes_2017}.

We then show that the localisation $\gamma$ preserves the properties of this object of equivalences, sending it to an object representing the actual space of equivalences in $\cC_\infty$. This allows us to derive that $\gamma$ sends univalent fibrations in $\cC$ to univalent morphisms in $\cC_\infty$ and to deduce our main result, namely that $\cC_\infty$ is an elementary $\infty$-topos. Thus, we obtain the top functor in the diagram as the following composite:

$$\mathbf{CxlCat}_{\mathtt{HoTT}} \to \mathbf{UnivTrb} \to \mathbf{ElTop}_{\infty}.$$

Our definition of an elementary $\infty$-topos is a variation of both proposed definitions in \autocite{shulman_towards_2018} and \autocite{rasekh_theory_2022}, where in addition the existence of finite colimits and the existence of a subobject classifier are assumed. 

Regarding finite colimits, it should be noted that there are examples, such as the $\infty$-category of $\pi$-finite spaces as studied in \cite{anel_category_2025}, which satisfy the axioms of our definition of elementary $\infty$-topos, but do not admit all finite colimits. Since such categories carry structure that is interesting enough to consider interpreting type theory in them, we want to include them in our definition of elementary $\infty$-topos. Moreover, we prove in Theorem \ref{thm: pushout types give finite colimits} that the assumption of pushout types (\cite[§6.8]{the_univalent_foundations_program_homotopy_2013}) and 0-types induces finite colimits in the $\infty$-localisation. Thus, the axiom is implied if we assume the existence of the appropriate higher inductive types in the type theory.

Regarding subobject classifiers, we will prove as Theorem \ref{theorem: every elementary infinity topos has subobject classifiers} that finite completeness, local cartesian closedness and univalent morphisms imply the existence of small subobject classifiers each of which classifies the monomorphisms of some univalent morphism. In \autocite{shulman_towards_2018} and \autocite{rasekh_theory_2022} the existence of a single `global' subobject classifier that classifies \textit{all} monomorphisms is assumed. Concerning the internal language of an $\infty$-topos satisfying this condition, this would require an additional assumption in the type theory. Subobject classifiers in type theory are usually as large as the ambient univalent universe (See \autocite[§10.1.4]{the_univalent_foundations_program_homotopy_2013}). This means that a global subobject classifier would require a universe of all universes, and it would be as large as that universe, inducing size issues. To fix this, one can assume the axiom of \textit{propositional resizing} (\autocite[Axiom 3.5.5]{the_univalent_foundations_program_homotopy_2013}). We will prove that, assuming this axiom, one obtains global subobject classifiers in the $\infty$-localisation of a categorical model of HoTT. We will however for the sake of generality not assume it in the general case. Moreover, Uemura shows in \cite{Uemura_2019_cubical_ass} that the axiom of propositional resizing is independent by constructing a model of univalence in which it fails, hence whose presentation does not admit a subobject classifier. We maintain again that such a model interpreting type theory is interesting enough to call its presentation an elementary $\infty$-topos.

The setup introduced in this paper extends also straightforwardly to type theories which admit $W$-types and graph colimits. We study the $\infty$-localisation of the models of such type theory in forthcoming work.

\vspace{0.5cm}

\noindent \textbf{Outline}

\vspace{0.25cm}

In section \ref{sect: preliminaries}, we will briefly review the necessary background about type theory and its models, tribes and about $\infty$-categories and their localisation. In section \ref{sect: univalence in TT}, we define univalent universes in type theory and the notion of a categorical model of univalence that we will use. In section \ref{sect: univalent tribs}, we introduce the notions of a univalent fibration and a homotopy subobject classifier in a tribe and prove that the latter always exists in a univalent tribe. We also show that every categorical model of univalence is a univalent tribe. In section \ref{sect: univalence in infinity-categories}, we prove that every univalent fibration becomes a univalent morphism in the $\infty$-localisation of a tribe. Moreover, we show that every univalent morphism in a locally cartesian closed and finitely complete $\infty$-category induces the existence of a subobject classifier. Finally, in section \ref{sect: ele infty toposes} we put all pieces together to derive our two main results: every univalent tribe, and hence model of HoTT, presents an elementary $\infty$-topos and every elementary $\infty$-topos has subobject classifiers.

\newpage

The following table gives a road map of where in the paper each of these steps is established:

\[\begin{tikzcd}[row sep=small, column sep={0cm, 0pt}, every label/.append style={outer sep=0.2cm}]
     & \mathbf{CxlCat}_{\mathtt{HoTT}} && \rightarrow &&&\textbf{UnivTrb} &&&& \rightarrow &&&&& \mathbf{ElTop}_{\infty} \\
    \textbf{Object of Equivalences} \quad &  \text{Def. } \ref{def: linv in contextual category}  && \xRightarrow{\text{Lem. \ref{lemma: equivalence context is object of equivalences}}} &&& \text{Constr. } \ref{constr: object of equiv} &&&& \xRightarrow{\text{Prop. \ref{prop: object of equivalences in loclaisation}}} &&&&& \text{Constr. } \ref{construction: object of equivalences infinity cat} \\
    \textbf{Univalence} \quad &  \text{Def. } \ref{def: univalent universe in contextual category}  && \xRightarrow{\text{Prop. \ref{prop: univalent universe is univalent fibration}}} &&& \text{Def. } \ref{def:univalent fibration-categorical def} &&&& \xRightarrow{\text{Thm. \ref{theorem: univalence in a tribe implies univalence in infinity-cat}}} &&&&& \text{Def. } \ref{def: univalence in infinity category1} \\
    \textbf{Subobject Classifier} \quad &  \text{Def. } \ref{def: Prop in contextual category}  && \xRightarrow{\text{Sect. }\ref{sect: homotopy subobject classifiers tribe} \hspace{0.20cm}} &&& \text{Def. } \ref{def: homotopy subobject classifier tribe} &&&& \xRightarrow{\text{Prop. \ref{prop: infty localisation of tribe with subobject classifier shas subobject classifiers}}} &&&&& \text{Def. } \ref{def: subobject classifier infinity-categories} \\
    \textbf{Colimits} \quad &  \text{Def. } \ref{def: cocone in contextual category}  && \xRightarrow{\text{Lemma. }\ref{lemma: categorical model has all internal homotopy pushouts}\hspace{0.20cm}} &&& \text{Def. } \ref{def:internal homotopy pushout new} &&&& \xRightarrow{\text{Prop. \ref{thm: localisation of pi-tribe with internal homotopy initial object and internal homotopy pushouts is cocomplete}}} &&&&& \text{Sect. } \ref{sect: finite colimits} 
\end{tikzcd}\]

\vspace{0.5cm}

\noindent \textbf{Acknowledgements}

 \vspace{0.1cm}

We are grateful to Chris Kapulkin for helpful comments on earlier versions of this draft and for many helpful discussions. We also want to thank Reid Barton, Tim Campion, Jonas Höfer, and Christian Sattler for many helpful conversations. Parts of this work are part of Maximilian Petrowitsch's PhD thesis supervised by Chris Kapulkin.

%% file: sect1-preliminaries.tex
\section{Preliminaries}
\label{sect: preliminaries}

We will first review some preliminary definitions and results about dependent type theory and its models, about tribes, and about $\infty$-categories and their localisations.

    We will work with Martin-Löf dependent type theory with intensional identity types as presented in \cite{martin-lof_intuitionistic_1984}. For a concise presentation of the syntax and rules that we will use, we refer to Appendix A in \cite{kapulkin_simplicial_2021}. From there will assume the basic structural rules, the rules for $\mathtt{Id}$, $\Sigma$ and $\Pi$-types (A.2.) as well as the $\Pi$-$\eta$ rule and the function extensionality rules (A.4.). 

    For the semantics of dependent type theory, we will use the framework of \textit{contextual categories} as introduced by Cartmell in his PhD thesis and published as \autocite{cartmell_generalised_1986}. We will refer to these contextual categories as \textit{categorical models of type theory}. A contextual category may be equipped with particular structures to model particular type constructors, such as such a $\Sigma$, $\Pi$ and $\mathtt{Id}$-structures. The definitions of these can be found in Appendix B of \cite{kapulkin_simplicial_2021}. There is a functor $\mathcal{C} \colon \mathbf{TT}_{\T} \to \mathbf{CxlCat}_{\T}$ from the category of type theories satisfying the rules of $\T$ to the category of those contextual categories that admit the respective structure to model the type constructors of $\T$. This functor sends a type theory to its \textit{syntactic category} which is an initial object in  $\mathbf{CxlCat}_{\T}$ (\cite{de_boer_proof_2020}, \cite{streicher_semantics_1991}).

    Tribes, introduced by Joyal in \cite{joyal_notes_2017}, are a particular weak model of dependent type theory where substitution is only modelled up to isomorphism or homotopy rather than equality.

    \begin{definition}[{\autocite[Def. 3.1.6]{joyal_notes_2017}}]
    Let $\cC$ be a category with a subcategory whose morphisms are called \textit{fibrations}. A morphism in $\cC$ is \textit{anodyne} if it has the left lifting property with respect to every fibration. We denote fibrations as $\twoheadrightarrow$ and anodyne maps as $\tilde{\rightarrowtail}$.
    
    A \textit{tribe} is a category $\cC$ with a subcategory of fibrations satisfying 
    \begin{enumerate}
        \item $\cC$ has a terminal object and for every object $X$, the unique morphism $X \to 1$ is a fibration (i.e. all objects are \textit{fibrant}).
        \item Pullbacks along fibrations exist in $\cC$ and fibrations are stable under pullback.
        \item Every morphism factors as an anodyne morphism followed by a fibration.
        \item Anodyne morphisms are stable under pullbacks along fibrations.
    \end{enumerate}
    \end{definition}
    
    Every categorical model dependent type theory with $\Sigma$ and $\mathtt{Id}$-structure is a tribe \cite[Prop. 9.7]{kapulkin_internal_2019} where the fibrations are the maps isomorphic to a composite of dependent projections. The main idea is that $\mathtt{Id}$-types correspond to path objects and terms of identity types to homotopies. We refer to §3 and the Appendix of \cite{joyal_notes_2017} for the definition of path objects, homotopies, homotopy equivalences, homotopy pullbacks, homotopy monomorphisms and the homotopy category of a tribe which we will denote as $Ho \cC$. We call a fibration that is also a homotopy equivalence a \textit{trivial fibration}. We will denote as $\cC/\!\!/X$ the tribe given by the full subcategory of $\cC/X$ spanned by fibrations and refer to it as the \textit{fibrant slice}. For type theories with extensional $\Pi$-types, the corresponding concept is that of a $\pi$-tribe. 

    \begin{definition}[Cf. {\cite[Def. 3.8.1]{joyal_notes_2017}}]
    \label{def_dep_prod}
    
        A tribe $\cC$ is a \textit{$\pi$-tribe} if for all fibrations $p \colon X \to Y$ the pullback functor $p^* \colon \cC/Y \to \cC /X$ has a partially defined right adjoint $\Pi_p$ that preserves homotopy equivalences and is defined between the fibrant slices $\Pi_p \colon \cC /\!\!/X \to \cC/\!\!/Y$.
    \end{definition}

    This implies that a $\pi$-tribe is cartesian closed, and we will write $\uHom(A,B)$ to denote the internal hom-object for objects $A$ and $B$. Every categorical model of dependent type theory with $\mathtt{Id}, \Sigma$ and $\Pi$-types satisfying function extensionality is a $\pi$-tribe by \autocite[Prop. 5.4]{kapulkin_locally_2017}. In fact, the consequence that $\Pi_p$ preserves trivial fibrations is the categorical expression of function extensionality (Cf. \cite[Lemma 5.9]{shulman_univalence_2015}). We have moreover the following important properties of $\pi$-tribes:

\begin{lemma}[Cf. {\autocite[Lem. 4.4]{kapulkin_locally_2017}}]
    \label{lemma: different characterisations of function extensionality}
    Let $\cC$ be a $\pi$-tribe and $p \colon A \to B$ a fibration. Then, 
    \begin{enumerate}
        \item  $\Pi_p \colon \cC/\!\!/A \to \cC/\!\!/B$ preserves path object factorisations in the underlying fibration categories.
        \item  for any morphisms $h,h'$ in $\cC/\!\!/A$, if $h \sim h'$ in $\cC/\!\!/A$, then $\Pi_p(h) \sim \Pi_p(h')$ in $\cC/\!\!/B$.
        \item the homotopy relation is compatible with adjoint transposition along the partial adjunction $p^* \dashv \Pi_p$. That is, if $\overline{h}$ and $\overline{h'}$ are adjoint transposes of $h$ and $h'$, then $h \sim h'$ if and only if  $\overline{h} \sim \overline{h'}$.\qed
    \end{enumerate}
\end{lemma}

 We will use the fact that every tribe is a fibration category in the sense of Ken Brown \cite{brown_abstract_1973} where the weak equivalences are the homotopy equivalences.

\begin{definition}[\autocite{brown_abstract_1973}]

    A \textit{fibration category} is a category $\cC$ with a subcategory of \textit{fibrations} and a subcategory of \textit{weak equivalences} denoted $\xrightarrow{\sim}$, where a map is an \textit{trivial fibration} if it is a fibration and a weak equivalence, such that
    \begin{enumerate}
        \item $\cC$ has a terminal object and for all objects $X \in \cC$ the unique morphism $X \to 1$ is a fibration.
        \item pullbacks along fibrations exist and (trivial) fibrations are stable under pullback.
        \item any morphism can be factored as a weak equivalence followed by a fibration.
        \item the class of weak equivalences satisfies the 2-out-of-3 property.
    \end{enumerate}
    \end{definition}

    A functor between fibration categories is \textit{exact} if it preserves fibrations, trivial fibrations, pullbacks along fibrations, and terminal objects. For the definition of path objects and homotopies in a fibration category see \cite[§I.1]{brown_abstract_1973}. Since all objects in a tribe are cofibrant, every path object in its underlying fibration category is via a trivial fibration equivalent to one in a tribe, and, by choosing a section, the two notions of homotopy coincide.
    
    We will pick quasi-categories as our model of $(\infty,1)$-categories, though most arguments can be spelt out in a model-independent fashion. We will refer to them simply as $\infty$-categories. Classical expositions of quasi-categories can be found in Joyal's work \cite{joyal_theory_2008} or in Lurie's \textit{Higher Topos Theory} \cite{lurie_higher_2009}. Most of what we need is however contained in Land's introductory book \cite{land_introduction_2021}. We will use in particular the straightening-unstraightening correspondence \autocite[§2.2.1]{lurie_higher_2009} and the Yoneda lemma for $\infty$-categories \autocite[Prop. 4.2.10]{land_introduction_2021}. For any $\infty$-category $\cCi$ we denote as $h\cCi$ its homotopy category. We call a morphism in an $\infty$-category an \textit{equivalence} if it becomes an isomorphism in $h \cC$. We say that a cone $\Tilde{F} \colon  \Delta^0 \star I \to \cCi$ over some diagram $F\colon I \to \cCi$ is a \textit{limit cone} if the map $\cCi/ \Tilde{F} \to \cCi/F$ is a trivial Kan fibration. A morphism $f \colon a \to b$ in an $\infty$-category  is a \textit{monomorphism} or \textit{(-1)-truncated} if the following square is a pullback

    \[\begin{tikzcd}
    	a & a \\
    	a & b
    	\arrow["id", from=1-1, to=1-2]
    	\arrow["id"', from=1-1, to=2-1]
    	\arrow["f", from=1-2, to=2-2]
    	\arrow["f"', from=2-1, to=2-2]
    \end{tikzcd}\]

     To avoid confusion, we will call a morphism in in the $\infty$-category of spaces $\mathbf{Spc}$ always `(-1)-truncated' if it is a monomorphism in this sense and we will refer to it as `monomorphism' if it is an actual strict monomorphism of simplicial sets.

    \begin{lemma}[{\autocite[Rmk. 5.5.6.10]{lurie_higher_2009}}]
    
        \label{lemma: different characterisations of monomorphism in infty cat}
        Let $f \colon a \to b$ be a morphism in an $\infty$-category. Then,
        \begin{enumerate}
            \item $f$ is a monomorphism if and only if for every morphism $ g \colon x \to b$, the mapping space $\mathpzc{Map}_{\cCi/b}(g,f)$ is either empty or contractible.
            \item $f$ is an equivalence if and only if for every morphism $ g \colon x \to b$, the mapping space $\mathpzc{Map}_{\cCi/b}(g,f)$ is contractible.
            \qed
        \end{enumerate} 
    \end{lemma}

    Cisinski generalised the notion of a fibration category to $\infty$-categories.

    \begin{definition}[{\autocite[Def. 7.4.6, 7.4.12 and 7.5.7]{cisinski_higher_2019}}]

     An \textit{$\infty$-category of fibrant objects} is a triple $(\cCi, W, Fib)$ where $\cCi$ is an $\infty$-category, $W\subseteq \cCi$ is a subcategory whose morphisms are called \textit{weak equivalences} and $Fib \subseteq \cCi$ is a subobject whose morphisms are called \textit{fibrations} such that

     \begin{enumerate}
         \item $Fib$ contains all identities and is closed under composition.
         \item $W$ satisfies the 2-out-of-3 property.
         \item $\cCi$ has a terminal object $1$ and for any object $x \in \cCi$ any map $x \to 1$ is in $Fib$.
         \item $\cCi$ admits pullbacks along fibrations and (trivial) fibrations are stable under pullback.
         \item for any morphism $f \colon x \to y$ there exist a weak equivalence $w \colon x \to x'$ in $W$ and a fibration $p \colon x' \to y$ such that $f$ is a composition of $p$ and $w$.
     \end{enumerate}
     
\end{definition}

    In particular, the nerve $N(\cC)$ of any fibration category $\cC$, hence also any tribe, is canonically an $\infty$-category of fibrant objects. The $\infty$-localisation $\cC_{\infty}$ of $N(\cC)$ at the set of weak equivalences always exists and is by \autocite[Prop. 7.5.6]{cisinski_higher_2019} finitely complete, giving us a localisation functor $\gamma \colon N(\cC) \to \cC_{\infty}$ which preserves terminal objects and pullbacks along fibrations. We will also write $\cC$ instead of $N(\cC)$. By \autocite[Rmk. 7.1.4]{cisinski_higher_2019}, we can (and will always) choose $\gamma$ and $\cC_{\infty}$ such that $\gamma$ is the identity on objects. Given an object $A$ in $\cC$, we will therefore write $A$ instead of $\gamma(A)$ whenever it is clear from the context that we mean the induced object in the localisation $\cC_{\infty}$. In order to perform concrete calculations in $\cC_{\infty}$ we will use Cisinski's $\infty$-categorical right calculus of fractions developed in \cite[§7.2]{cisinski_higher_2019} applied to the class of trivial fibrations of $\cC$. This implies by the formula \cite[\;7.2.10.4]{cisinski_higher_2019} that any morphism in $\cC_{\infty}$ can be written as $\gamma(p)\gamma(s)^{-1}$ where $s$ is a trivial fibration. The following result about adjunctions and localisations is crucial.

    \begin{proposition}[{\autocites[Prop. 7.1.14]{cisinski_higher_2019}[Prop. 5.1.15]{land_introduction_2021} }]
    \label{prop:adjoint functors derive adjoint functors unit/counit}
    Suppose $F  \colon \cC \rightleftarrows \cD \colon G$ is an adjunction of exact functors between fibration categories with unit $\eta$ and counit $\epsilon$. Then, there is an adjunction $F_{\infty}  \colon \cC_{\infty} \rightleftarrows \cD_{\infty}  \colon G_{\infty}$ between their $\infty$-localisations with unit  $\eta_{\infty}$ and counit $\epsilon_{\infty}$ such that the diagrams
    
    \[\begin{tikzcd}
    	\cC & \cD \\
    	{\cC_{\infty}} & {\cD_{\infty}}
    	\arrow[""{name=0, anchor=center, inner sep=0}, "F", shift left=2, from=1-1, to=1-2]
    	\arrow["\gamma"', from=1-1, to=2-1]
    	\arrow[""{name=1, anchor=center, inner sep=0}, "G", shift left=2, from=1-2, to=1-1]
    	\arrow["\gamma", from=1-2, to=2-2]
    	\arrow[""{name=2, anchor=center, inner sep=0}, "{F_{\infty}}", from=2-1, to=2-2]
    	\arrow[""{name=3, anchor=center, inner sep=0}, "{G_{\infty}}", shift left=3, from=2-2, to=2-1]
    	\arrow["\dashv"{anchor=center, rotate=-90}, draw=none, from=0, to=1]
    	\arrow["\dashv"{anchor=center, rotate=-90}, draw=none, from=2, to=3]
    \end{tikzcd}\]
    
    \[\begin{tikzcd}
    	{\cC \times \Delta^1} & \cC & {\cD \times \Delta^1} & \cD \\
    	{\cC_{\infty} \times \Delta^1} & {\cC_{\infty}} & {\cD_{\infty} \times \Delta^1} & {\cD_{\infty}}
    	\arrow["\eta", from=1-1, to=1-2]
    	\arrow["{\gamma \times id}"', from=1-1, to=2-1]
    	\arrow["\gamma", from=1-2, to=2-2]
    	\arrow["\epsilon", from=1-3, to=1-4]
    	\arrow["{\gamma \times id}"', from=1-3, to=2-3]
    	\arrow["\gamma", from=1-4, to=2-4]
    	\arrow["{\eta_{\infty}}"', from=2-1, to=2-2]
    	\arrow["{\epsilon_{\infty}}"', from=2-3, to=2-4]
    \end{tikzcd}\]
    commute up to homotopy.
    \qed\end{proposition}

    This proposition implies in particular that the $\infty$-localisation commutes up to homotopy with adjoint transposition meaning that

    \[\begin{tikzcd}
    	{\mathsf{Hom}_{\cC}(FX,Y)} & {\mathsf{Hom}_{\cC}(X,GY)} \\
    	{\mathpzc{Map}_{\cC_{\infty}}(F_{\infty}X,Y)} & {\mathpzc{Map}_{\cC_{\infty}}(X,G_{\infty}Y)}
    	\arrow["\cong", from=1-1, to=1-2]
    	\arrow["\gamma"', from=1-1, to=2-1]
    	\arrow["\gamma", from=1-2, to=2-2]
    	\arrow["\simeq"', from=2-1, to=2-2]
    \end{tikzcd}\]
    commutes up to homotopy, a fact that we will assume when using Proposition \ref{prop:adjoint functors derive adjoint functors unit/counit}.

    Moreover, it follows from \autocite[Prop. 7.6.16]{cisinski_higher_2019} that if $\cC$ is a $\pi$-tribe then $\cC_{\infty}$ is locally cartesian closed. In the particular case in which we have the partial adjunction $p^* \dashv \Pi_{p}$ in a $\pi$-tribe $\cC$ for some fibration $p \colon X \to Y$, we get by the fact that $(\cC/\!\!/X)_{\infty} \simeq \cC_{\infty}/X$ proved in \cite[Cor. 7.6.13]{cisinski_higher_2019} that the induced adjoint functors from Proposition \ref{prop:adjoint functors derive adjoint functors unit/counit} are equivalent to an adjoint pair $\gamma(p)^* : \cC_{\infty}/Y \rightleftarrows \cC_{\infty}/X \colon \Pi_{\gamma(p)}$ making the diagrams above commute up to homotopy. We will use this fact implicitly throughout. We will write $\myuline{\mathpzc{Map}}(a,b)$ to denote the internal mapping space object for objects $a$ and $b$.

%% file: sect2-univalent-universes-in-TT.tex
\section{Univalence in Type Theory}
\label{sect: univalence in TT}

We will start by defining the notion of a univalent universe in dependent type theory and in contextual categories.

\subsection{Univalent Universes}
\begin{definition}
\normalfont
     A \textit{universe} $U$ in a type theory $\T$ is given by the following rules:
    $$
        \infer[(U_F)]{\vdash U \text{ type}}{} \qquad \infer[(U_I)]{X: U \vdash \mathtt{El}(X) \text{ type}}{}
    $$
    We will also assume that $U$ is closed under type constructors. The rules for that can be found in \autocite[Appendix A.3]{kapulkin_simplicial_2021}.    

\end{definition}

\begin{definition}[{\cite[§6.2]{rijke_introduction_2022}}]
    A type theory $\T$ has \textit{enough universes} if for any finite list of types
    $$\Gamma_0 \vdash A_0 \text{ type} \qquad ... \qquad \Gamma_n \vdash A_n  \text{ type}$$
    there is a universe $U$ closed under the type constructors of $\T$ together with a term 
    $$\Gamma_i \vdash \widehat{A_i} : U$$
     for all $0 \leq i \leq n$ such that
    $$\Gamma_i \vdash \mathtt{El}(\widehat{A_i}) \equiv A_i \text{ type}.$$
    We define $U_0$ to be the universe obtained from the empty list of types. For any universe $U_n$, we define the successor $U_{n+1}$ to be the universe obtained from the list
    \begin{align*}
    & \vdash U \text{ type} \\
    X: U & \vdash \mathtt{El}(X) \text{ type}.
    \end{align*}

\end{definition}

\begin{definition}
    \normalfont
    \label{def: type of eauivalences}
    Let $\T$ be a type theory with $\mathtt{Id}$ and $\Sigma$ and $\Pi$-types, let $A,B$ be types of $\T$ and let $f \colon  A \to B$ be a function. We define a type of equivalences $A \simeq B$ as follows. Define
    $$\mathtt{LInv}(f)  :\equiv \sum_{g \colon  B \to A}\mathtt{Id}_{A \to A}(gf,1_A)$$
    $$\mathtt{RInv}(f)  :\equiv \sum_{g \colon  B \to A}\mathtt{Id}_{B \to B}(fg,1_B).$$
    where $fg$ and $gf$ are obtained as composites using the rules for non-dependent function types. Concretely, they are defined using the composition function 
    $$\mathtt{comp} \colon (B \to C) \to ((A \to B) \to (A \to C))$$
    given by $\mathtt{comp}:= \lambda g.\lambda f. \lambda x. g(f(x))$. Next, define
    $$\mathtt{isEquiv}(f) :\equiv \mathtt{LInv}(f) \times \mathtt{RInv}(f)$$
    where we say that $f$ is an equivalence if $\mathtt{isEquiv}(f)$ is inhabited. Thus, we define the type of equivalences of functions $A \to B$ as
    $$(A\simeq B) :\equiv \sum_{f \colon  A \to B} \mathtt{isEquiv}(f).$$

\end{definition}

Given a universe $U$ and $A: U$ and $B: U$, we can form the type $\mathtt{Id}_{U}(A,B)$ of proofs that types $A$ and $B$ are equal in the universe $U$.  Using the identity elimination rule, we obtain a term
$$A: U, B: U, p: \mathtt{Id}_U(A,B) \vdash \mathtt{idToEquiv}(A,B,p) : \mathtt{El}(A) \simeq \mathtt{El}(B).$$

\begin{definition}
    \normalfont
    \label{def: univalence in type theory}
    We say that $U$ is \textit{univalent} if for any $A : U$ and $B : U$ the function

    $$\lambda p.\mathtt{idToEquiv}(A,B,p) : \mathtt{Id}_U(A,B) \to \mathtt{El}(A) \simeq \mathtt{El}(B)$$
    is an equivalence. That is, if there is a term,

        $$A: U, B: U \vdash \mathtt{uvt}(A,B) : \mathtt{isEquiv}(\lambda p.\mathtt{idToEquiv}(A,B,p)).$$
\end{definition}

\begin{definition}
    \label{def: type theory Prop}
    Given a type $A :U$ in a universe $U$, we define the type of witnesses that $A$ is a proposition as
    $$\mathtt{isProp}(A) :\equiv \prod_{x,y : \mathtt{El}(A)}\mathtt{Id}_{\mathtt{El}(A)}(x,y)$$
    and we define $$\mathtt{Prop}_U : \equiv \sum_{A: U}\mathtt{isProp}(A).$$
    
\end{definition}

Observe that for any successor universe $U_{n+1}$ there is a natural inclusion $U_{n} \to U_{n+1}$ which induces a natural function $\mathtt{Prop}_{U_n} \to \mathtt{Prop}_{U_{n+1}} $.

\begin{definition}
    A type theory $\T$ satisfies \textit{propositional resizing} if every function $\mathtt{Prop}_{U_n} \to \mathtt{Prop}_{U_{n+1}} $ is an equivalence.
\end{definition}

\subsection{Models of Univalence}

We will now define what it means for a categorical model $\cC$ of type theory to model the univalence axiom. In particular if $\T$ is a univalent type theory in the above sense then $\cC(\T)$ will be such a model of the univalence axiom. Our definition will be equivalent to the one in \cite[Appendix B.3]{kapulkin_simplicial_2021} by function extensionality, but it will be simpler from a categorical point of view, avoiding pointwise identifications using $\Pi$-types. 

Fix a categorical model $\cC$ that carries an $\mathtt{Id}$, $\Sigma$ and $\Pi$-structure satisfying the $\eta$-rule. If $\Gamma.A$ and $\Gamma.B$ are two context extensions of $\Gamma$, we will define the ordinary function context as $\Gamma.[A,B] := \Gamma.\Pi(A,p_A^*B)$ using the $\Pi$-structure. By the $\Pi$-structure, for any context extension $\Gamma.A$ of $\Gamma$ there is a canonical section $\lambda((id_A,id_A)) \colon \Gamma \to \Gamma.[A,A]$.

\begin{definition}
    \label{def: linv in contextual category}
    
    Let $\Gamma.A$ and $\Gamma.B$ be two context extensions of $\Gamma$. Define $\Gamma.[A,B].\mathsf{RInv}_{A,B}$ to be the object
    $$\Gamma.[A,B].\Sigma(p_{[A,B]}^*[B,A], (\mathsf{comp}_{B,A,B},\lambda((id_B,id_B)) p_{[A,B]} p_{p_{[A,B]}^*[B,A]}
    )^*\mathsf{Id}_{[B,B]})$$
    where $\mathsf{comp}_{B,A,B} \colon \Gamma.[A,B].p_{[A,B]}^*[B,A] \to \Gamma.[B,B]$ is defined as 
    $$\mathsf{comp}_{B,A,B} := \lambda((id_{\Gamma.[A,B].p_{[A,B]}^*[B,A].(p_{[B,A]}q_{p_{[A,B]},[B,A]})^*B},c))$$
    and $c:= q_{p_A,B} \mathsf{app}_{A, p_A^*B} (id_{\Gamma.[A,B]} \times_{\Gamma} (q_{p_B,A} \mathsf{app}_{B,p_B^*A}))$ is defined as the composite
    $$\Gamma.[A,B].p_{[A,B]}^*[B,A].(p_{[B,A]}q_{p_{[A,B]},[B,A]})^*B \xrightarrow{id_{\Gamma.[A,B]} \times_{\Gamma} (q_{p_B,A} \mathsf{app}_{B,p_B^*A})}  \Gamma.[A,B].p_{[A,B]}^*A \xrightarrow{\mathsf{app}_{A,p_A^*B}} \Gamma.A.p_A^*B \xrightarrow{q_{p_A,B}} \Gamma.B.$$
    Define $\Gamma.[A,B].\mathsf{LInv}_{A,B}$ to be the object
    $$\Gamma.[A,B].\Sigma(p_{[A,B]}^*[B,A], (\mathsf{comp}_{A,B,A},\lambda((id_A,id_A)) p_{[A,B]} p_{p_{[A,B]}^*[B,A]})^*\mathsf{Id}_{[A,A]})$$ 
    where $\mathsf{comp}_{A,B,A} \colon \Gamma.[A,B].p_{[A,B]}^*[B,A] \to \Gamma.[A,A]$ is defined as 
    $$\mathsf{comp}_{A,B,A} := \lambda((id_{\Gamma.[B,A].p_{[B,A]}^*[A,B].(p_{[A,B]}q_{p_{[B,A]},[A,B]})^*A},c'))\mathsf{exch}_{[A,B],[B,A]}$$
    and $c' := q_{p_B,A} \mathsf{app}_{B, p_B^*A} (id_{\Gamma.[B,A]} \times_{\Gamma} (q_{p_A,B} \mathsf{app}_{A,p_A^*B})) $ is defined as the composite
    $$\Gamma.[B,A].p_{[B,A]}^*[A,B].(p_{[A,B]}q_{p_{[B,A]},[A,B]})^*A \xrightarrow{id_{\Gamma.[B,A]} \times_{\Gamma} (q_{p_A,B} \mathsf{app}_{A,p_A^*B})}  \Gamma.[B,A].p_{[B,A]}^*B \xrightarrow{\mathsf{app}_{B,p_B^*A}} \Gamma.B.p_B^*A \xrightarrow{q_{p_B,A}} \Gamma.A.$$
    Then, define
    $$\Gamma.[A,B].\mathsf{isEquiv}_{A,B} := \Gamma.[A,B].\mathsf{LInv}_{A,B} \times \mathsf{RInv}_{A,B}$$
    $$\Gamma.\mathsf{Equiv}(A,B) := \Gamma.\Sigma([A,B],\mathsf{isEquiv}_{A,B}).$$

\end{definition}

\begin{definition}
    For any context extension $\Gamma.A$ of $\Gamma$ the map $\lambda(id_A,id_A)\colon \Gamma \to \Gamma.A$ induces a canonical section $\mathsf{id}_A^{\mathsf{isEquiv}} \colon  \Gamma \to \Gamma.[A,A].\mathsf{isEquiv}_{A,A}$. This again induces a canonical section $\mathsf{id}_A^{\mathsf{Equiv}} \colon \Gamma \to \Gamma.\mathsf{Equiv}(A,A)$.

\end{definition}

\begin{definition}
    Given a context extension $\Gamma.A.B$, using the $\mathtt{Id}$-structure we define a map
    $$\mathsf{idToEquiv}_{A,B} \colon   \Gamma.A.p_A^*A.\mathsf{Id}_A \to \Gamma.A.p_A^*A.\mathsf{Id}_A.p_{\mathsf{Id}_A}^*\mathsf{Equiv}(p_{p_A^*A}^*B, q_{p_A,A}^*B)$$
    as
    $$ \mathsf{idToEquiv}_{A,B} := J_{p_{Id_A}^*\mathsf{Equiv}(p_{p_A^*A}^*B, q_{p_A,A}^*B), q_{refl_A,p_{Id_A}^* \mathsf{Equiv}(p_{p_A^*A}^*B, q_{p_A,A}^*B)} \mathsf{id}_B^{\mathsf{Equiv}} }.$$
    Then the $\Pi$-structure gives us a section 
    $$\lambda(\mathsf{idToEquiv}_{A,B}) \colon \Gamma.A.p_A^*A \to \Gamma.A.p_A^*A.[\mathsf{Id}_A, \mathsf{Equiv}(p_{p_A^*A}^*B, q_{p_A,A}^*B)].$$
    We define univalence as the context representing that this map is itself an equivalence
    
    $$\mathsf{isUvt}(A,B) := \Gamma.A.p_A^*A.\lambda(\mathsf{idToEquiv}_{A,B})^*\mathsf{isEquiv}_{\mathsf{Id}_A,\mathsf{Equiv}(p_{p_A^*A}^*B, q_{p_A,A}^*B)}.$$

\end{definition}

\begin{definition}
    A \textit{universe} in a categorical model $\cC$ is a distinguished object $U.\mathsf{El} \in Ob_2\cC$. We say that the universe is closed under type constructors if it satisfies the conditions of \cite[Appendix B.2.1]{kapulkin_simplicial_2021}. 
\end{definition}

\begin{definition}
\label{def: univalent universe in contextual category}
    Given a universe $U.\mathsf{El}$ in a categorical model $\cC$ we say that $U.\mathsf{El}$ \textit{satisfies the univalence axiom} if $\cC$ is equipped with a section $\mathsf{uvt}_{U,\mathsf{El}} \colon U.p_U^*U \to \mathsf{isUvt}(U,\mathsf{El})$.
\end{definition}

\begin{definition}
    Given a universe $U.\mathsf{El}$ in a categorical model $\cC$ there is a canonical map
    $$\mathsf{id}^{\delta}  \colon  U \to U.p_U^*U.\mathsf{Equiv}(p_{p_U^*U}^*\mathsf{El}, q_{p_U,U}^*\mathsf{El})$$
    induced by $\lambda(id_{U.\mathsf{El}},id_{U.\mathsf{El}}) \colon  U \to U.[\mathsf{El}, \mathsf{El}]$. 
\end{definition}

\begin{definition}
    A categorical model $\cC$ has \textit{enough universes} if for every finite set of context context extensions $\Gamma_0.A_0,...,\Gamma_n.A_n$ there is a universe $U.\mathsf{El}$ closed under the type constructors that are modelled by $\cC$ together with for all $0 \leq i \leq n$ a map $\chi_{p_{A_i}} \colon  \Gamma_i \to U$ such that $p_{A_i} = \chi_{p_{A_i}}^*p_{El}$, i.e. $p_{A_i}$ is the canonical pullback 

    \[\begin{tikzcd}
    	{\Gamma_i.A_i} & {U.El} \\
    	{\Gamma_i} & U
    	\arrow["{q_{\chi_{p_{A_i}},U.El}}", from=1-1, to=1-2]
    	\arrow["{p_{A_i}}"', from=1-1, to=2-1]
    	\arrow["\lrcorner"{anchor=center, pos=0.125}, draw=none, from=1-1, to=2-2]
    	\arrow["{p_{El}}", from=1-2, to=2-2]
    	\arrow["{\chi_{p_{A_i}}}"', from=2-1, to=2-2]
    \end{tikzcd}\]
We define $U_0.\mathsf{El}_0$ to be the universe obtained from the empty set of context extensions and given $U_n.\mathsf{El}_n$ we define $U_{n+1}.\mathsf{El}_{n+1}$ to be the universe obtained from the set consisting of the dependent projections $U_n \to 1$ and $U_n.\mathsf{El}_n \to U_n$. This induces canonical maps $\chi_{p_{\mathsf{El}_n}} \colon U_n \to U_{n+1}$.
\end{definition}

\begin{definition}
    \label{def: Prop in contextual category}
    Let $U.\mathsf{El}$ be a universe in a categorical model $\cC$. We define a context
    $$U.\mathsf{isProp} : \equiv U.\mathsf{El}.\Pi(U.\mathsf{El.p_{\mathsf{El}}^*\mathsf{El}}, U.\mathsf{El.p_{\mathsf{El}}^*\mathsf{El}}.\mathtt{Id}_{\mathsf{El}})$$
    which comes with a dependent projection $U.\mathsf{isProp} \to U$. Then, we define

    $$\mathsf{Prop}_U : \equiv \Sigma(U, U.\mathsf{isProp}).$$
\end{definition}

\begin{definition}
    \label{def: propositional resizing in contextual category}
    A categorical model $\cC$ satisfies \textit{propositional resizing}, if for any two  universes with $\chi_{p_{El_n}} \colon U_{n} \to U_{n+1}$ and for the map
    $$U_n.\mathsf{isProp} \to U_{n+1}.\mathsf{isProp}$$
    which is induced by pullback and which we call $i$ there is a section
    $$1 \to \lambda((id_{p_{U_n.\mathsf{isProp}}},i)))^*(\mathsf{isEquiv}_{U_n.\mathsf{isProp},U_{n+1}.\mathsf{isProp}}).$$
\end{definition}

It is clear from the definition that if a universe $U$ in a type theory $\T$ satisfies the univalence axiom in the sense of Definition \ref{def: univalence in type theory}, then in its syntactic category the interpretation of the universe $\cC(\T)$, $[ X: U \vdash \mathtt{El}(X) \text{ type} ]$ satisfies the univalence axiom in the sense of Definition \ref{def: univalent universe in contextual category}. In particular, if $\T$ has enough univalent universes, then $\cC(\T)$ has enough univalent universes and if $\T$ satisfies propositional resizing then so does  $\cC(\T)$.

\subsection{Intermezzo on Pushout Types}

\begin{definition}[{\cite[6.8]{the_univalent_foundations_program_homotopy_2013}}]
    Given a span of functions $f \colon C \to A$ and $g \colon C \to B$ a \textit{pushout type} is the higher inductive type $A \amalg_C B$ presented by constructors

    \begin{align*}
        \mathtt{inl} & : C \to A \amalg_C B \\
        \mathtt{inr} & : C \to B \amalg_C B \\
        \text{for each $c : C$, a term } \mathtt{glue}(c) & : \mathtt{Id}_{A \amalg_C B}(\mathtt{inl}(c), \mathtt{inr}(c))
    \end{align*}
\end{definition}

\begin{definition}[{\cite[Def. 6.8.1]{the_univalent_foundations_program_homotopy_2013}}]
     Given a span of functions $f \colon C \to A$ and $g \colon C \to B$ and a type $D$ define the type of cocones under this span with apex $D$ as 
     
     $$\mathtt{Cocone}(f,g, D):\equiv \sum_{i : A \to D}\sum_{j : B \to D} \mathtt{Id}_{C \to D}(if, jg).$$ 
\end{definition}

\begin{lemma}[{\cite[Lem. 6.8.2]{the_univalent_foundations_program_homotopy_2013}}]
    There is a canonical equivalence

    $$(A \amalg_C B \to D)\simeq \mathtt{Cocone}(f,g, D)$$
    given by sending $t$ to $(t \mathtt{inl}, t \mathtt{inr}, \mathtt{ap}_t \mathtt{glue}).$
\end{lemma}

\begin{definition}
    \label{def: cocone in contextual category}
    Let $\Gamma.A$, $\Gamma.B$, $\Gamma.C$ and $\Gamma.D$ be context extensions in a categorical model $\cC$ of dependent type theory, and let $f \colon \Gamma \to \Gamma.[C,A]$ and $g \colon \Gamma \to \Gamma.[C,B]$ be sections. Define $\Gamma.\mathsf{Cocone}(f,g, D)$ to be the object

    $$\Gamma.\mathsf{Cocone}(f,g, D) : \equiv \Gamma.[A,D].\Sigma(p_{[A,D]}^*[B,D], k^*\mathsf{Id}_{[C,D]})$$
    where $k$ is defined as the map
    $$\Gamma.[A,D].p^*_{[A,D]}[B,D] \xrightarrow{\mathsf{comp}_{C,A,D} (id_{\Gamma.[A,D]} \times_{\Gamma} f p_{[B,D]})   \times_{\Gamma} \mathsf{exch}_{[A,D], [B,D]} \mathsf{comp}_{C,B,D} (id_{\Gamma.[B,D]} \times_{\Gamma} g p_{[A,D]} )}   \Gamma.[C,D].p^*_{[C,D]}[C,D].$$

    Given an extension $\Gamma.Q$ and a section $(i_A, i_B, H) \colon \Gamma \to \Gamma.\mathsf{Cocone}(f,g, Q)$ we say that $\Gamma.Q$ is a \textit{homotopy pushout} of $f$ and $g$ if for any context extension $\Gamma.D$ and the induced canonical map
    
    $$h \colon \Gamma.[Q, D] \to \Gamma.p_{[Q, D]}^*\mathsf{Cocone}(f,g, D)$$ there is a section 
    $\Gamma \to \lambda(h)^*\mathsf{isEquiv}_{[Q,D], \mathsf{Cocone}(f,g,D)}.$ We say that a categorical model \textit{has all homotopy pushouts} if there is a homotopy pushout for any pair of sections $f \colon \Gamma \to \Gamma.[C,A]$ and $g \colon \Gamma \to \Gamma.[C,B]$.

\end{definition}

%% file: sect3-univalent-tribes.tex
\section{Univalent Tribes}
\label{sect: univalent tribs}

Our goal in this section is to define the notion of a univalent fibration in a tribe, this will be the analogue of a univalent universe in type theory. For this, we will adapt the definition we gave for contextual categories to work in arbitrary tribes. 

Before we can start, we will prove the following lemma about homotopy monic fibrations that we will use on several occasions throughout this section. We can think of homotopy monic fibrations as those dependent types that are propositions.

    \begin{lemma}
    \label{lemma: different characterisations of homotopy monic fibration}
    
        Let $f \colon X \to Y$ be a fibration in a tribe $\cC$. The following are equivalent:
    
        \begin{enumerate}
            \item $f$ is a homotopy monomorphism.
            \item For any fibered path object $P_f$ of $f$ in $\cC/\!\!/Y$, the fibration $(\partial_0,\partial_1)$ is trivial.
            \item For any fibered path object $P_f$ of $f$, the fibration $(\partial_0,\partial_1)$ has a section.
            \item For any morphisms $g,h \colon A \to X$ such that $fg$ and $fh$ are fibrations, $fg = fh$ implies $g \sim h$.
    
        \end{enumerate}
    \end{lemma}
    
    \begin{proof}
    $(1)\Leftrightarrow(2)$ follows from 2-out-of-3 in the diagram 
    \[\begin{tikzcd}
    	& {P_f} \\
    	X && {X \times_Y X}
    	\arrow["{(\partial_0,\partial_1)}", from=1-2, to=2-3]
    	\arrow["{\sim i}", from=2-1, to=1-2]
    	\arrow["{(id,id)}"', from=2-1, to=2-3]
    \end{tikzcd}\]
    defining the path object factorisation, since $f$ is homotopy monic if and only if $(id,id)$ is a homotopy equivalence. $(2) \Rightarrow (3)$ follows, since in a tribe trivial fibrations have sections by \cite[Cor. 3.4.7.]{joyal_notes_2017}. For $(3) \Rightarrow (1)$ choose a homotopy inverse $h$ of $i \colon  X \to P_f$ then any section $s$ of $(\partial_0,\partial_1)$ gives a right homotopy inverse $hs$ of $(id,id)$
    $$(id,id) hs = (\partial_0,\partial_1)ihs \sim (\partial_0,\partial_1)s \sim id$$
    and the pullback projections are always left inverses of $(id,id)$.

    For $(3) \Rightarrow (4)$, any section $s$ of $(\partial_0,\partial_1)$ is a homotopy $\pi_1 \sim \pi_2$ between the pullback projections $\pi_i \colon X \times_Y X \to X$:
    \[\begin{tikzcd}
    	& {P_f} \\
    	{X \times_Y X} && {X \times_Y X}
    	\arrow["{(\partial_0,\partial_1)}", from=1-2, to=2-3]
    	\arrow["s", from=2-1, to=1-2]
    	\arrow["{(\pi_1,\pi_2) = id}"', from=2-1, to=2-3]
    \end{tikzcd}\]
    Suppose now we are given maps $g_1,g_2 \colon Z \to X$ such that $h= fg_1 = fg_2$. Then, by the universal property of the pullback, there is a unique induced map $(g_1,g_2)\colon Z \to X \times_Y X$ such that 
    $$g_1 = \pi_1(g_1,g_2) \sim \pi_2(g_1,g_2) = g_2$$
    as required.
    
    For $(4) \Rightarrow (3)$ assume that $f$ satisfies $(4)$. Then, since $f\pi_1 = f\pi_2$ is some fibration, we get by (4) that $\pi_1 \sim \pi_2$. The homotopy witnessing that $\pi_1 \sim \pi_2$ gives a section of $(\partial_0,\partial_1)$.
    \end{proof}

\subsection{The Object of Equivalences}

In order to define the notion of a univalent fibration in a $\pi$-tribe, we will define an object of equivalences $\myuline{\mathsf{Eq}}_B(p_1,p_2)$ for fibrations $p_1 \colon  E_1 \to B$ and $p_2 \colon  E_2 \to B$, analogously to the definition of the context of equivalences $\Gamma.\mathsf{Equiv}(A,B)$ for contextual categories. We will first consider the case $B=1$ and then use the fact that $\cC/1 \cong \cC$ and that the fibrant slice of a tribe is again a tribe to apply the construction to $\cC/\!\!/B$ in order to obtain $\myuline{\mathsf{Eq}}_B(p_1,p_2)$. 

For this, we need to find an analogous definition for the map $\mathsf{comp}$ that works in any $\pi$-tribe. Here we use the fact that a $\pi$-tribe is cartesian closed, hence admits a self-enrichment via the usual internal composition maps. For objects $A, B$ and $C$, this morphism
$c_{A,B,C}\colon \uHom(B,C)  \times \uHom(A,B) \to \uHom(A,C)$ as is defined as the adjoint transpose of 
$$ \uHom(B,C) \times \uHom(A,B) \times A \xrightarrow{id \times \epsilon_{A,B} } \uHom(B,C) \times B \xrightarrow{\epsilon_{B,C}} C.$$
using the counit of the product-hom adjunction. This morphism makes the following diagram commute
\[\begin{tikzcd}
	{\mathsf{Hom}(X, \uHom(B,C) \times \uHom(A,B))} & { \mathsf{Hom}_{\cC/X}(X \times B, X \times C) \times \mathsf{Hom}_{\cC/X}(X \times A, X \times B)} \\
	{\mathsf{Hom}(X, \uHom(A,C))} & {\mathsf{Hom}_{\cC/X}(X \times A, X \times C)}
	\arrow["\cong", from=1-1, to=1-2]
	\arrow["{{(c_{A,B,C})}_*}", from=1-1, to=2-1]
	\arrow["{C_{A,B,C}}"', from=1-2, to=2-2]
	\arrow["\cong"', from=2-1, to=2-2]
\end{tikzcd}\]
where the map $C_{A,B,C}$ is the composition function. Thus, we have in particular that given maps $\overline{f} \colon X \to  \uHom(A,B) $ and $\overline{g} \colon X \to  \uHom(B,C)$ with transposes $f\colon X \times A \to X \times B$ and  $g\colon X \times B \to X \times C$ over $X$, the transpose of the composite $c_{A,B,C} (\overline{g}, \overline{f})$ is the composite $gf \colon X \times A \to X \times C$ over $X$. That is,
$$\overline{c_{A,B,C} (\overline{g}, \overline{f})} = gf.$$
For any object $A$ and any object $Z$ we will write $const_A$ to denote the unique map $Z \to 1 \to \uHom(A,A)$ that is induced by the transpose of the identity on $A$.

\begin{construction}
\label{constr: object of equiv}
Consider objects $A$ and $B$ in a $\pi$-tribe as well as path object factorisations $A \to P_A \to A \times A$ and $B \to P_B \to B \times B$. We define representations for left and right homotopy invertible morphisms as the pullbacks
\[\begin{tikzcd}[every label/.append style={outer sep=0.2cm}]
	{\underline{\mathsf{RInv}}(A,B)} & {\underline{\mathsf{Hom}}(B,P_B)} \\
	{\underline{\mathsf{Hom}}(A,B) \times \underline{\mathsf{Hom}}(B,A)} & {\underline{\mathsf{Hom}}(B,B) \times \underline{\mathsf{Hom}}(B,B)}
	\arrow[from=1-1, to=1-2]
	\arrow[from=1-1, to=2-1]
	\arrow[two heads, from=1-1, to=2-1]
	\arrow["{({\partial_0}_*,{\partial_1}_*)}", two heads, from=1-2, to=2-2]
	\arrow[""{name=0, anchor=center, inner sep=0}, "{(c_{B,A,B}, const_B)}"', from=2-1, to=2-2]
	\arrow["\lrcorner"{anchor=center, pos=0.125}, draw=none, from=1-1, to=0]
\end{tikzcd}\]
\[\begin{tikzcd}[every label/.append style={outer sep=0.2cm}]
	{\underline{\mathsf{LInv}}(A,B)} & {\underline{\mathsf{Hom}}(A,P_A)} \\
	{\underline{\mathsf{Hom}}(B,A) \times \underline{\mathsf{Hom}}(A,B)} & {\underline{\mathsf{Hom}}(A,A) \times \underline{\mathsf{Hom}}(A,A)}
	\arrow[from=1-1, to=1-2]
	\arrow[from=1-1, to=2-1]
	\arrow[two heads, from=1-1, to=2-1]
	\arrow["{({\partial_0}_*,{\partial_1}_*)}", two heads, from=1-2, to=2-2]
	\arrow[""{name=0, anchor=center, inner sep=0}, "{(c_{A,B,A},const_A)}"', from=2-1, to=2-2]
	\arrow["\lrcorner"{anchor=center, pos=0.125}, draw=none, from=1-1, to=0]
\end{tikzcd}\]
The map $({\partial_0}_*,{\partial_1}_*) \cong (\partial_0,\partial_1)_*$ is a fibration because it is the value of the internal hom functor on the fibration $(\partial_0,\partial_1)$ which preserves fibrations by the $\pi$-tribe axioms. Composing $\myuline{\mathsf{RInv}}(A,B)$ and $\myuline{\mathsf{LInv}}(A,B)$ with the obvious projections (which are fibrations by all objects in a tribe being fibrant) and taking a pullback
\[\begin{tikzcd}
	{\myuline{\mathsf{Eq}}(A,B)} & {\myuline{\mathsf{LInv}}(A,B)} \\
	{\myuline{\mathsf{RInv}}(A,B)} & {\uHom(A,B)}
	\arrow[two heads, from=2-1, to=2-2]
	\arrow[two heads, from=1-2, to=2-2]
	\arrow[two heads, from=1-1, to=2-1]
	\arrow[two heads, from=1-1, to=1-2]
	\arrow["\lrcorner"{anchor=center, pos=0.125}, draw=none, from=1-1, to=2-2]
\end{tikzcd}\]
we obtain our object of equivalences $\myuline{\mathsf{Eq}}(A,B)$.

\end{construction}

\begin{lemma}
\label{lemma: transpose of morphism into object of equivalences is equivalence}
    A map $\overline{e}\colon X \to \myuline{\mathsf{Eq}}(A,B)$ corresponds to maps $e \colon X \times A \to X \times B$, $g,h\colon X \times B \to X \times A$ , $H \colon X \times A \to X \times P_A$ and $K \colon X \times B \to X \times P_B$ all over $X$ such that 
    $$H \colon ge \sim id$$
    $$K \colon eh \sim id.$$
    Hence, $e$ is a homotopy equivalence.
\end{lemma}

\begin{proof}
    The maps $e,g,h,H$ and $K$ are given by the universal property of the pullbacks and transposition. Since pullback preserves path objects $X \times P_A$ is a path object for $X \times A$ over $X$ and similar $X \times P_B$ is a path object for $X \times B$ over $X$. Both path objects are equipped with projections $\partial_i' := X \times \partial_i$. The transpose of ${\partial_i}_*\overline{H}$ is $\partial_i' H$ and similar for $K$. 
    
    By definition of the composition map, the transpose of the composite $c_{A,B,A}(\overline{g},\overline{e})$ is $ge$ and the transpose of the composite  $c_{B,A,B}(\overline{e},\overline{h})$ is $eh$, hence $\partial_0' H = ge$ and $\partial_0'K = eh$. The transpose of the constant identity $X \to 1 \to \uHom(A,A)$ is the identity $X \times A \to X \times A$ and similar for $B$, thus $\partial_1'H = id$ and $\partial_0'K = id$ which proves the claim.
\end{proof}

In particular, points $1 \to \myuline{\mathsf{Eq}}(A,B)$ correspond to homotopy equivalences $A \to B$. We therefore have 

\begin{proposition}
\label{prop-Eq(-xA,B)-is-representable}
    The functor $\mathsf{Eq}_X(- \times A, - \times  B)\colon \cC^{op} \to \Set$ sending $X \in \cC$ to the set $\mathsf{Eq}_X(X \times A,X \times B)$ of homotopy equivalences over $X$ together with their homotopy inverses and witnessing homotopies for fixed path objects for $A$ and $B$ is representable in a tribe such that we have 
    $$\mathsf{Hom}(X, \myuline{\mathsf{Eq}}(A,B)) \cong \mathsf{Eq}_{X}(X \times A, X \times B ).$$
\end{proposition}

\begin{proof}
    Immediate from Lemma \ref{lemma: transpose of morphism into object of equivalences is equivalence}.
\end{proof}

    We get the local version of the statement by the fact that fibrant slices are again tribes and applying the statement to the tribe $\cC/\!\!/B$.
    
\begin{corollary}
\label{cor: representation of functor object of equivalences in tribe}
    Let $p_1\colon E_1 \to B$, $p_1\colon E_2 \to B$ be fibrations in a tribe, then the functor $\cC/B^{op} \to \Set$ sending $f\colon X \to B$  to $\mathsf{Eq}_X(f^*p_1,f^*p_2)$ is representable by an object $\myuline{\mathsf{Eq}}_B(p_1,p_2)$ such that
    
    \[\pushQED{\qed} 
    \mathsf{Hom}_{\cC/B}(f, \myuline{\mathsf{Eq}}_B(p_1,p_2)) \cong \mathsf{Eq}_X(f^*p_1,f^*p_2).\qedhere
    \popQED\]
    \end{corollary}

Next, we show that this construction does not depend on the choice of path ojects.

\begin{lemma}
    \label{lemma: object of equivalences is independent of choice of of path object}
    Let $\cC$ be a $\pi$-tribe and $A$ and $B$ be objects in $\cC$. Then if $P_A, P_A'$ and $P_B, P_B'$ are path objects of $A$ and $B$ respectively, any two objects of equivalences $\myuline{\mathsf{Eq}}(A,B)$ and $\myuline{\mathsf{Eq}}(A,B)'$ obtained from applying Construction \ref{constr: object of equiv} to $P_A, P_B$ and $P_A', P_B'$ respectively are homotopy equivalent. That is, the object of equivalences is unique up to homotopy equivalence.
\end{lemma}

\begin{proof}

    By lifting, $P_A \simeq P_A'$ and $P_B \simeq P_B'$. By the $\pi$-tribe axioms, $\uHom(A,P_A) \simeq \uHom(A,P_A')$ and similarly $\uHom(B,P_B) \simeq \uHom(B,P_B')$. Since for any morphism the induced pullback functor between fibrant slices is always exact, we see that pulling back along $(c_{B,A,B},const_B)$ in 
\[\begin{tikzcd}
	{\myuline{\mathsf{RInv}}(A,B)} & {\myuline{\mathsf{RInv}}(A,B)'} & {\myuline{\mathsf{Hom}}(B,P_B)} & {\myuline{\mathsf{Hom}}(B,P_B')} \\
	{\myuline{\mathsf{Hom}}(A,B) \times \myuline{\mathsf{Hom}}(B,A)} &&& {\myuline{\mathsf{Hom}}(B,B) \times \myuline{\mathsf{Hom}}(B,B)}
	\arrow["\sim", from=1-1, to=1-2]
	\arrow[two heads, from=1-1, to=2-1]
	\arrow[two heads, from=1-2, to=2-1]
	\arrow["\sim", from=1-3, to=1-4]
	\arrow[two heads, from=1-3, to=2-4]
	\arrow[two heads, from=1-4, to=2-4]
	\arrow["{(c_{B,A,B},const_B)}"', from=2-1, to=2-4]
\end{tikzcd}\]
induces a homotopy equivalence $\myuline{\mathsf{RInv}}(A,B) \simeq \myuline{\mathsf{RInv}}(A,B)'$ and similarly we get $\myuline{\mathsf{LInv}}(A,B) \simeq \myuline{\mathsf{LInv}}(A,B)'$. Repeating the same argument, we get $\myuline{\mathsf{Eq}}(A,B) \simeq \myuline{\mathsf{Eq}}(A,B)'$.
\end{proof}

\begin{lemma}
\label{lemma: precomposition with homotopy equivalence is homotopy equivalence}
    Let $f \colon X \times A \to X \times B$ be a homotopy equivalence over $X$ in a $\pi$-tribe, $\overline{f}\colon X \to \uHom(A,B)$ be its adjoint transpose and $C$ be any object. Then, the composite
    $$X \times \uHom(C,A) \xrightarrow{(\pi_X, c_{C,A,B}(\overline{f} \times id_{\uHom(C,A)}))} X \times \uHom(C,B)$$
    postcomposition map denoted as $f_*$ is a homotopy equivalence over $X$.
\end{lemma}

\begin{proof}
    We may assume $f$ has a two-sided homotopy inverse $g \colon  X \times B \to X \times A$ over $X$ which has as transpose a map $\overline{g} \colon  X \to \uHom(B,A)$. This allows us to define a precomposition map $g_*$ as the composite
    $$X \times \uHom(C,B) \xrightarrow{(\pi_X,  c_{C,B,A} ( \overline{g} \times id_{\uHom(C,B)}))} X \times \uHom(C,A).$$
    We claim $f_*$ and $g_*$ are homotopy inverses of each other. To show that $g_*f_*$ is homotopic to the identity $id_{X \times \uHom(C,A)} = (\pi_X, \pi_{ \uHom(C,A)})$, it suffices by \cite[Lem. 3.3.14]{joyal_notes_2017} to show that the second component of the composite $g_*f_*$
    $$X \times \uHom(C,A) \xrightarrow{(\pi_X,c_{C,B,A}(\overline{g}\pi_X, c_{C,A,B}  (\overline{f} \times id_{\uHom(C,A)} ))} X \times \uHom(C,A)$$
    is homotopic to the projection $ \pi_{ \uHom(C,A)} \colon  X \times \uHom(C,A) \to \uHom(C,A)$. By Lemma \ref{lemma: different characterisations of function extensionality}, it suffices to show that the transpose of this map is homotopic to the transpose of $\pi_{ \uHom(C,A)}$. We compute
    \begin{align*}
         \overline{c_{C,B,A}(\overline{g}\pi_X, c_{C,A,B}  (\overline{f} \times id_{\uHom(C,A)} ))} = \; & ( g \times id_{\uHom(C,A)})\overline{c_{C,A,B} (\overline{f}\pi_X,\pi_{\uHom(C,A)})} \\[1em]
        =  \; & ( g \times id_{\uHom(C,A)}) ( f \times id_{\uHom(C,A)}) \overline{\pi_{\uHom(C,A)}} \\[1em]
        = \; &  ( gf \times id_{\uHom(C,A)}) \overline{\pi_{\uHom(C,A)}}.
    \end{align*}
     But $gf \sim id$. So this map is homotopic to $\overline{\pi_{\uHom(C,A)}}$ which is the transpose of the projection as desired. The case of $f_*g_*$ is analogous.
\end{proof}

\begin{corollary}
    Dually, if $f$ is homotopy equivalence, then the pre-composition map $f^*$ given by $$(\pi_X,c_{A,B,C} (id_{\uHom(B,A)} \times \overline{f}))$$ is a homotopy equivalence.
\qed\end{corollary}

The following proposition is the categorical counterpart to the type-theoretic fact that assuming function extensionality the type of witnesses that a function is an equivalence is a proposition, i.e. $\mathtt{isProp}(\mathtt{isEquiv}(f))$ is always inhabited for any function $f$.

\begin{proposition}
\label{prop: projection object of equivalences is homotopy monomorphism}

     Let $\cC$ be a $\pi$-tribe and let $A$ and $B$ be objects in $\cC$. Then, the projection $\pi_1\colon \myuline{\mathsf{Eq}}(A,B) \to \uHom(A,B)$ is a homotopy monomorphism.
\end{proposition}

\begin{proof}
    By Lemma \ref{lemma: different characterisations of homotopy monic fibration} (4), it suffices to show that for any morphisms $\overline{H}, \overline{K} \colon X \to \myuline{\mathsf{Eq}}(A,B)$ such that $\pi_1\overline{H}=\pi_1\overline{K} = \overline{f}$ for some map $\overline{f} \colon X \to \uHom(A,B)$ we have $\overline{H} \sim \overline{K}$. Define $\myuline{\mathsf{RInv}}_f$ as the pullback

    \[\begin{tikzcd}[every label/.append style={outer sep=0.2cm}]
	{\myuline{\mathsf{RInv}}_f} & {\myuline{\mathsf{RInv}}(A,B)} & {\myuline{\mathsf{Hom}}(B,P_B)} \\
	{X \times  \myuline{\mathsf{Hom}}(B,A)} & {\myuline{\mathsf{Hom}}(A,B) \times \myuline{\mathsf{Hom}}(B,A)} & {\myuline{\mathsf{Hom}}(B,B) \times \myuline{\mathsf{Hom}}(B,B)} \\
	& {} \\
	&& {}
	\arrow[from=1-1, to=1-2]
	\arrow[from=1-1, to=2-1]
	\arrow["\lrcorner"{anchor=center, pos=0.125}, draw=none, from=1-1, to=3-2]
	\arrow[from=1-2, to=1-3]
	\arrow[from=1-2, to=2-2]
	\arrow[two heads, from=1-2, to=2-2]
	\arrow["\lrcorner"{anchor=center, pos=0.125}, draw=none, from=1-2, to=4-3]
	\arrow["{({\partial_0}_*,{\partial_1}_*)}", two heads, from=1-3, to=2-3]
	\arrow["{\overline{f} \times id_{\myuline{\mathsf{Hom}}(B,A)}}"', from=2-1, to=2-2]
	\arrow["{(c_{B,A,B},const_B)}"', from=2-2, to=2-3]
    \end{tikzcd}\]
    
    \vspace{-1.5cm}
    
    \noindent Observe that the composite of the first component of the morphism on the bottom is the second component of the internal post-composition map $f_{*}$. Thus, pulling back this diagram along $X \to 1$ we obtain a diagram of pullbacks
    \[\begin{tikzcd}[scale cd=0.8, every label/.append style={outer sep=0.2cm}]
	{\myuline{\mathsf{RInv}}_f} & {X \times \bullet} & {X \times \myuline{\mathsf{Hom}}(B,P_B) } \\
	{(X \times \myuline{\mathsf{Hom}}(B,A)) \times_X (X \times 1)} & {(X \times \myuline{\mathsf{Hom}}(B,A)) \times_X (X \times \myuline{\mathsf{Hom}}(B,B))} & {(X \times \myuline{\mathsf{Hom}}(B,B)) \times_X (X \times  \myuline{\mathsf{Hom}}(B,B))} \\
	{X \times 1} & {X \times \myuline{\mathsf{Hom}}(B,B)} & {} \\
	& {} & {} \\
	& {}
	\arrow[from=1-1, to=1-2]
	\arrow[from=1-1, to=2-1]
	\arrow["\lrcorner"{anchor=center, pos=0.125}, draw=none, from=1-1, to=4-2]
	\arrow[from=1-2, to=1-3]
	\arrow[from=1-2, to=2-2]
	\arrow["{(p_1,p_2)}", two heads, from=1-2, to=2-2]
	\arrow["\lrcorner"{anchor=center, pos=0.125}, draw=none, from=1-2, to=4-3]
	\arrow[two heads, from=1-3, to=2-3]
	\arrow["{id \times_X  (id_X \times\overline{id_B})}"', from=2-1, to=2-2]
	\arrow["{\pi_2}", from=2-1, to=3-1]
	\arrow["\lrcorner"{anchor=center, pos=0.125}, draw=none, from=2-1, to=5-2]
	\arrow["{f_* \times_X id}"', from=2-2, to=2-3]
	\arrow["{\pi_2}", from=2-2, to=3-2]
	\arrow["{id_X \times \overline{id_B}}"', from=3-1, to=3-2]
    \end{tikzcd}\]
    
    \vspace{-1.5cm}
    
    \noindent which is a diagram in $\cC/\!\!/X$. By Lemma \ref{lemma: transpose of morphism into object of equivalences is equivalence}, the adjoint transpose of the morphism $\overline{f}$ is a homotopy equivalence $f \colon  X \times A \to X \times B$ over $X$. By Lemma \ref{lemma: precomposition with homotopy equivalence is homotopy equivalence}, the post-composition map $f_*$ is therefore a homotopy equivalence. By Lemma \ref{lemma: different characterisations of function extensionality}, we see that $\uHom(B,P_B)$ is in fact a path object for the object $\uHom(B,B)$ in $\cC \simeq \cC/\!\!/1$. Since the pullback functor $\cC/\!\!/1 \to \cC/\!\!/X$ is exact, it preserves path objects and so $X \times \uHom(B,P_B)$ is a path object for $X \times \uHom(B,B)$ over $X$ in $\cC/\!\!/X$. Hence, the object $X \times \bullet$ is in fact a mapping path object for $f_*$. Then, $p_1$ is the pullback of a trivial fibration hence a homotopy equivalence, thus the induced map $X \times \uHom(B,A) \to X \times \bullet$ is a homotopy equivalence as well by 2-out-of-3. Therefore $p_2$ is a trivial fibration by 2-out-of-3. Then, $\myuline{\mathsf{RInv}}_f \to X \times 1 \cong X$ is the pullback of the trivial fibration $p_2$ along the constant identity $const_B \colon  X \to \uHom(B,B)$, hence a trivial fibration.

Dually, one shows for $\myuline{\mathsf{LInv}}(A,B)$ and pre-composition $f^*$ that the dually constructed fiber $\myuline{\mathsf{LInv}}_f \to X$ is a trivial fibration. Then, considering $\myuline{\mathsf{Eq}}_f := \myuline{\mathsf{LInv}}_f \times_X \myuline{\mathsf{RInv}}_f$, we get that $\myuline{\mathsf{Eq}}_f \to X$ is a trivial fibration. Then, the maps $H$ and $K$ induce by the universal property of the pullback sections $H'$ and $K'$ of this trivial fibration as in the diagram
\[\begin{tikzcd}
	X \\
	& {\myuline{\mathsf{Eq}}_f} & {\myuline{\mathsf{Eq}}(A,B)} \\
	& X & {\myuline{\mathsf{Hom}}(A,B)}
	\arrow["{H',K'}", dashed, from=1-1, to=2-2]
	\arrow["{H,K}", curve={height=-12pt}, from=1-1, to=2-3]
	\arrow["{id_X}"', curve={height=12pt}, equals, from=1-1, to=3-2]
	\arrow[from=2-2, to=2-3]
	\arrow["\sim"', two heads, from=2-2, to=3-2]
	\arrow["\lrcorner"{anchor=center, pos=0.125}, draw=none, from=2-2, to=3-3]
	\arrow[from=2-3, to=3-3]
	\arrow["{\overline{f}}"', from=3-2, to=3-3]
\end{tikzcd}\]
Since any two sections of a trivial fibration are homotopic we get that $H' \sim K'$, hence $H \sim K$.
 \end{proof}

\subsection{Univalent Fibrations and Univalent Tribes}

Using the construction of the object of equivalences from above, we can define the notion of univalent fibration in a $\pi$-tribe. Let $p \colon E \to B$ be a fibration and let  $ \pi_i \colon B\times B \to B$ be the product projections. We can then form the internal hom
$$(\partial_0, \partial_1) \colon \uHom_{B \times B}(\pi_1^*p,\pi_2^*p) \to B \times B$$
which is a fibration that represents a functor $(\cC/B\times B)^{op} \to \Set$ which sends a pair $(f,g) \colon X \to B \times B$ to the set $\mathsf{Hom}_{\cC/X}(f^*p,g^*p)$.

Taking $(f,g)$ to be the diagonal $(id,id) \colon B \to B \times B$, we get by taking the identity on $E$ a canonical map $p \to p$. By the universal property of the internal hom, this corresponds to a morphism

\[\begin{tikzcd}
	B && {\uHom_{B \times B}(\pi_1^*p,\pi_2^*p)} \\
	& {B \times B}
	\arrow["{\delta_E}", from=1-1, to=1-3]
	\arrow["{(id,id)}"', from=1-1, to=2-2]
	\arrow["{(\partial_0,\partial_1)}", two heads, from=1-3, to=2-2]
\end{tikzcd}\]

   Given a fibration $p\colon E \to B$ we write $\myuline{\mathsf{Eq}}_{B}(E):= \myuline{\mathsf{Eq}}_{B \times B}(\pi_1^*p,\pi_2^*p)$. Then, we have by Corollary \ref{cor: representation of functor object of equivalences in tribe} for any $(f,g) \colon X \to B\times B$
   $$\mathsf{Hom}_{\cC/B\times B}((f,g), \myuline{\mathsf{Eq}}_{B}(E)) \cong \mathsf{Eq}_{X}(f^*p,g^*p).$$
The map $\delta_E\colon B \to \uHom_{B \times B}(\pi_1^*p,\pi_2^*p)$ factors through $\myuline{\mathsf{Eq}}_{B}(E)$ (by taking the identity as inverses and constant homotopies) hence we get our definition of univalence:

\begin{definition}
\label{def:univalent fibration-categorical def}

A fibration $p\colon E \to B$ in a $\pi$-tribe is \textit{univalent} if the morphism $\delta_E \colon B \to \myuline{\mathsf{Eq}}_{B}(E)$ is a homotopy equivalence.

\end{definition}

By Lemma \ref{lemma: object of equivalences is independent of choice of of path object}, this definition is independent of the choice of object of equivalences, i.e. of the choice of the underlying path objects.

\begin{definition}
\label{def: local class tribe}
    
    Let $\cC$ be a $\pi$-tribe. A class of fibrations $S$ in $\cC$ is a \textit{local class} if it is closed under pullback. A local class is \textit{closed} if it is closed under composition and if for all fibrations $f,g \in S$, $\Pi_f(g) \in S$. A local class is $p$-\textit{bounded} if its elements are all obtained as pullbacks of some morphism $p$.
\end{definition}

\begin{definition}
\label{def: local class of pullbacks tribe}
    Given a fibration $p\colon E \to B$ in a tribe, we denote as $S_p$ the class of morphisms obtained as pullbacks of $p$. This is a $p$-bounded local class.
\end{definition}

\begin{definition}
\label{def: enough univalent fibs tribe}
    A $\pi$-tribe has \textit{enough univalent fibrations} if for every fibration $f\colon A \to X$, there is a univalent fibration $p$ such that $f \in S_p$ and $S_p$ is closed. We will call such a tribe also a \textit{univalent tribe}.
\end{definition}

Observe that if  a fibration $p\colon E \to B$ is univalent, then there is a factorisation

\[\begin{tikzcd}
	B && {\myuline{\mathsf{Eq}}_B(E)} \\
	& {B \times B}
	\arrow["{\delta_E \sim}", from=1-1, to=1-3]
	\arrow["{(id,id)}"', from=1-1, to=2-2]
	\arrow["{(\partial_0,\partial_1)}", two heads, from=1-3, to=2-2]
\end{tikzcd}\]
of the diagonal of $B$ into a homotopy equivalence, followed by a fibration. This is a path object factorisation in a fibration category but not yet a path object in a tribe since the first map need not be anodyne. However, since all objects in a tribe are cofibrant, by factoring the homotopy equivalence $\delta_E$, any map into $\myuline{\mathsf{Eq}}_B(E)$ can be lifted along a trivial fibration to a path object in a tribe, thus $\myuline{\mathsf{Eq}}_B(E)$ can be upgraded and used as if it were a full path object. This means that morphisms into $\myuline{\mathsf{Eq}}_B(E)$ are in fact homotopies.

This shows that our categorical definition of univalence incorporates the type theoretic notion of univalence as `being equivalent is equivalent to being homotopic'. The following proposition establishes an important property of univalent fibrations, namely that they classify fibrations via pullback along a homotopy unique morphism.

\begin{proposition}
\label{prop: pullback of univalent fibration is homotopy unique}
    Let $p \colon E \to B$ be a univalent fibration in a $\pi$-tribe and let $f \colon A \to X$ and  $f' \colon A' \to X$ be any fibrations in $\cC/ X$ such that there are pullbacks 
\[\begin{tikzcd}
	A & E & {A'} & E \\
	X & B & X & B
	\arrow[from=1-1, to=1-2]
	\arrow["f"', two heads, from=1-1, to=2-1]
	\arrow["\lrcorner"{anchor=center, pos=0.125}, draw=none, from=1-1, to=2-2]
	\arrow["p", two heads, from=1-2, to=2-2]
	\arrow[from=1-3, to=1-4]
	\arrow["{f'}"', two heads, from=1-3, to=2-3]
	\arrow["\lrcorner"{anchor=center, pos=0.125}, draw=none, from=1-3, to=2-4]
	\arrow["p", two heads, from=1-4, to=2-4]
	\arrow["\psi"', from=2-1, to=2-2]
	\arrow["{\psi'}"', from=2-3, to=2-4]
\end{tikzcd}\]
Then, $A$ and $A'$ are homotopy equivalent over $X$ via some homotopy equivalence $k \colon f \to f'$ in $\cC/X$ if and only if $\psi \sim \psi'$. 

\end{proposition}

\begin{proof}
    We have $f = \psi^*p $ and  $f'= \psi'^*p$. $k \colon \psi^*p \to \psi'^*p$ is a homotopy equivalence if and only if it defines an element of $\mathsf{Eq}_X(\psi^*p, \psi'^*p)$. By the universal property of $\myuline{\mathsf{Eq}}_B(E)$, this corresponds to the existence of a morphism $H \colon (\psi,\psi') \to (\partial_0,\partial_1)$ over $B \times B$ such that 
\[\begin{tikzcd}
	X && {\myuline{\mathsf{Eq}}_B(E)} \\
	& {B \times B}
	\arrow["H", from=1-1, to=1-3]
	\arrow["{(\psi,\psi')}"', from=1-1, to=2-2]
	\arrow["{(\partial_0,\partial_1)}", two heads, from=1-3, to=2-2]
\end{tikzcd}\]
commutes. But $\myuline{\mathsf{Eq}}_B(E)$ can be upgraded to a path object $\myuline{\mathsf{Eq}}_B(E)'$ for $B$ and $H$ can be upgraded to a homotopy $\tilde{H} \colon X \to \myuline{\mathsf{Eq}}_B(E)'$ such that $\overline{H} \colon \psi \sim \psi'$. Conversely, any such homotopy can be factored through $\myuline{\mathsf{Eq}}_B(E)$ inducing by the universal property of $\myuline{\mathsf{Eq}}_B(E)$ a homotopy equivalence $\psi^*p =f \simeq \psi'^*p =f'$ over $X$.
\end{proof}

\begin{remark}
\label{remark: univalent fibration is small object classifier}
    
    In the light of Proposition \ref{prop: pullback of univalent fibration is homotopy unique} we can think of univalent fibrations $p$ as `small object classifiers' for the class of fibrations $S_p$ obtained by pulling back $p$,  in the sense that for any fibration $f \colon A \to X$ in $S_p$ there is a homotopy unique `name' $\chi_f \colon X \to B$ such that $f$ is (up to homotopy equivalence) the pullback $\chi_f^*p$. Thus, we have via pullback a bijection
    $$\{\text{homotopy classes of morphism } X \to B\} \cong \{\text{homotopy equiv. classes of fibrations } A \to X \in S_p\}.$$
    This is a direct generalisation of subobject classifiers $\Omega$ for ordinary categories, which classify all monomorphisms $A \to X$ via a unique characteristic morphism  $X \to \Omega$.

\end{remark}

The preceding remark establishes a `homotopy universal property' of univalent fibrations $p \colon E \to B$ in a $\pi$-tribe $\cC$. As one expects, this homotopy universal property becomes an actual universal property in the $\infty$-localisation of the $\pi$-tribe, that is when inverting all homotopy equivalences. This is in fact one of our main results below when considering the localisation of a $\pi$-tribe $\cC$ as an $\infty$-category.

 However, we already get for the homotopy category $Ho\cC$ the expected `truncated' universal property that does not distinguish when two fibrations in $S_p$ are homotopy equivalent in more than one way.

\begin{proposition}
    Let $p \colon E \to B$ be a univalent fibration in a $\pi$-tribe $\cC$ and let $Ho\cC$ be the homotopy category of $\cC$. The functor $E_p \colon (Ho\cC)^{op} \to \Set$ sending an object $X$ to the set 
    $$E_p(X) =\{ \text{homotopy equiv. classes of fibrations } \bullet \to X \in S_p\}$$
    and a homotopy class $[h\colon X \to Y]$ to a function $E_p(h) \colon E_p(Y) \to E_p(X)$ sending a homotopy equivalence class $[f\colon A \to Y]$ to $[h^*f\colon h^*A \to X]$ is well-defined and represented by $B$ with the universal element $[p] \in E_p(B)$ such that

    $$\mathsf{Hom}_{Ho\cC}(X,B) \cong E_p(X).$$
\end{proposition}

\begin{proof}

    The bijection and well-definedness of the maps involved follow directly from Proposition \ref{prop: pullback of univalent fibration is homotopy unique} and the exactness of $h^*$.
\end{proof}

We will see below that it is necessary and possible to pass to $\infty$-categories in order to give $E_p(X)$ the full structure of a space without modding out (i.e. truncating) by some equivalence relation.

\begin{lemma}
    
    \label{lemma: pullback of object of equivalences is object of equivalences}
    Let $p \colon E \to B$ be a fibration and let $f \colon Y \to X$ be a fibration such that we have a pullback
    \[\begin{tikzcd}
    	X & E \\
    	Y & B
    	\arrow["h", two heads, from=1-1, to=1-2]
    	\arrow["q"', two heads, from=1-1, to=2-1]
    	\arrow["\lrcorner"{anchor=center, pos=0.125}, draw=none, from=1-1, to=2-2]
    	\arrow["p", two heads, from=1-2, to=2-2]
    	\arrow["f"', two heads, from=2-1, to=2-2]
    \end{tikzcd}\]
    Then, the morphism $E_f \to Y \times Y$ given by the pullback
    
    \[\begin{tikzcd}
    	{E_f} & {\myuline{\mathsf{Eq}}_B(E)} \\
    	{Y \times Y} & {B\times B}
    	\arrow["k", two heads, from=1-1, to=1-2]
    	\arrow["{(\partial_0',\partial_1')}"', two heads, from=1-1, to=2-1]
    	\arrow["\lrcorner"{anchor=center, pos=0.125}, draw=none, from=1-1, to=2-2]
    	\arrow["{(\partial_0,\partial_1)}", two heads, from=1-2, to=2-2]
    	\arrow["{f \times f}"', two heads, from=2-1, to=2-2]
    \end{tikzcd}\]
    is an object of equivalences for the fibration $q$.
\end{lemma}

\begin{proof}
    Let $(\psi_1,\psi_2) \colon A \to Y \times Y$ be any object in $\cC/Y\times Y$. By the adjunction $\Sigma_{f\times f} \dashv (f\times f)^*$ between the post-composition and pullback functors and by the universal property of $\myuline{\mathsf{Eq}}_B(E)$ we have
    \begin{align*}
            \mathsf{Hom}_{\cC/Y \times Y}((\psi_1,\psi_2), E_f) \cong \; & \mathsf{Hom}_{\cC/B \times B}( (f\psi_1,f\psi_2), \myuline{\mathsf{Eq}}_B(E)) \\
            \cong \; & \mathsf{Eq}_{A \times A}(\psi_1^*(f^*p),\psi_2^*(f^*p)) \\
            = \; & \mathsf{Eq}_{A \times A}(\psi_1^*q,\psi_2^*q)
    \end{align*}
    which shows that $E_f$ is an object of equivalences of $q$.
\end{proof}

The following proposition is the 1-categorical version of the analogue statement for univalent morphisms in $\infty$-categories, as proved in \autocite[Cor. 3.10]{gepner_univalence_2017} and \autocite[Prop. 2.5]{rasekh_univalence_2021}.

\begin{proposition}
    
    \label{prop: univalence is preserved and reflected by pullback along homotopy monomorphisms}
    Let $p \colon E \to B$ a univalent fibration and let $f \colon Y \to B$ be a fibration such that
    \[\begin{tikzcd}
    	X & E \\
    	Y & B
    	\arrow["h", two heads, from=1-1, to=1-2]
    	\arrow["q"', two heads, from=1-1, to=2-1]
    	\arrow["\lrcorner"{anchor=center, pos=0.125}, draw=none, from=1-1, to=2-2]
    	\arrow["p", two heads, from=1-2, to=2-2]
    	\arrow["f"', two heads, from=2-1, to=2-2]
    \end{tikzcd}\]
    is a pullback. Then, $q$ is univalent if and only if $f$ is a homotopy monomorphism.
\end{proposition}

\begin{proof}

   By Lemma \ref{lemma: pullback of object of equivalences is object of equivalences}, pullback along $f \times f$ gives an object of equivalences $E_f \to Y \times Y$ for $q$. Let $\delta_Y \colon Y \to E_f$ be the transpose of the identity on $q$ defined analogously to $\delta_E$. Both $\delta_E f$ and $k \delta_Y$ represent the identity equivalence on $q$. Thus, this induces by the Yoneda lemma a commutative diagram
\[\begin{tikzcd}
	Y & B \\
	{E_f} & {\underline{\mathsf{Eq}}_B(E)} \\
	{Y \times Y} & {B\times B}
	\arrow["f", two heads, from=1-1, to=1-2]
	\arrow["{\delta_Y}"', from=1-1, to=2-1]
	\arrow["{(id,id)}"', curve={height=24pt}, from=1-1, to=3-1]
	\arrow["{\delta_E \sim}", from=1-2, to=2-2]
	\arrow["{(id,id)}", curve={height=-30pt}, from=1-2, to=3-2]
	\arrow["k", two heads, from=2-1, to=2-2]
	\arrow["{(\partial_0',\partial_1')}", two heads, from=2-1, to=3-1]
	\arrow["\lrcorner"{anchor=center, pos=0.125}, draw=none, from=2-1, to=3-2]
	\arrow["{(\partial_0,\partial_1)}"', two heads, from=2-2, to=3-2]
	\arrow["{f \times f}"', two heads, from=3-1, to=3-2]
\end{tikzcd}\]
The map $\delta_E$ is a homotopy equivalence by univalence of $p$. Now, by the isomorphism $Y \times_B Y \cong (Y \times Y) \times_{(B \times B)} B$, $f$ is a homotopy monomorphism if and only if the outer rectangle of this diagram is a homotopy pullback. The lower square is a homotopy pullback, since it is a pullback of fibrations. By the pasting lemma for homotopy pullbacks \autocite[Lem. 6.2.4]{joyal_notes_2017}, therefore the upper square is a homotopy pullback if and only if $f$ is a homotopy monomorphism. But, by \autocite[Lem. 6.2.3]{joyal_notes_2017} the upper square is a homotopy pullback if and only if $\delta_X$ is a homotopy equivalence. Thus, $q$ is univalent if and only if $f$ is a homotopy monomorphism.
\end{proof}

\subsection{Homotopy Subobject Classifiers}
\label{sect: homotopy subobject classifiers tribe}

We will need to show later on that our $\infty$-category presented by the type theory via the $\infty$-localisation of its models has subobject classifiers. We define here the analogous notion of such classifiers in a tribe, that is, we define classifiers for homotopy monomorphisms or homotopy subobjects. The construction uses the categorical interpretation of the definition of the type $\mathtt{Prop}$ that we saw in Definition \ref{def: type theory Prop}. We will in particular define an object $\Omega_p$ for every univalent fibration $p$ whose construction is extending the definition of $U.\mathsf{isProp}$ that we saw in Definition \ref{def: Prop in contextual category} from contextual categories to $\pi$-tribes.

We start with establishing an alternative characterisation of those homotopy monic fibrations that are classified by some univalent fibration. The intuition is that a non-empty fiber of such a morphism is contractible if and only if it is contractible as fiber of the domain of the univalent fibration. 

Let $\cC$ be a tribe. Consider a fibration $f$ classified by a fibration $p$ as in the pullback
\[\begin{tikzcd}
    X & E \\
    Y & B
    \arrow["{q_f}", from=1-1, to=1-2]
    \arrow["f"', two heads, from=1-1, to=2-1]
    \arrow["\lrcorner"{anchor=center, pos=0.125}, draw=none, from=1-1, to=2-2]
    \arrow["p", two heads, from=1-2, to=2-2]
    \arrow["{\chi_f}"', from=2-1, to=2-2]
\end{tikzcd}\]
We will write $p \times p$ for the product of $p$ with itself in the fibrant slice $\cC/\!\!/B$, i.e. $p\times p$ is the fibration $E \times_B E \to B$ in the pullback square

    \[\begin{tikzcd}
    	{E\times_BE} & E \\
    	E & B
    	\arrow[two heads, from=1-1, to=1-2]
    	\arrow[two heads, from=1-1, to=2-1]
    	\arrow["{p\times p}"{description}, two heads, from=1-1, to=2-2]
    	\arrow["p", two heads, from=1-2, to=2-2]
    	\arrow["p"', two heads, from=2-1, to=2-2]
    \end{tikzcd}\]
Given any fibered path object factorisation $E \to P_p \to E \times_B E$ for $p$, pullback along $\chi_f$ will induce a fibered path object factorisation $X \to P_f \to X \times_Y X$ for $f$ as in the diagram below since the pullback functor between slices is a morphism tribes, hence preserves fibrations, anodyne maps and homotopy equivalences. Moreover, pullback preserves limits so that $\chi_f^*(p\times p) = \chi^*p \times \chi^*p = f \times f$. We denote as $q_f \times_{\chi_f} q_f$ the resulting morphism $X\times_Y X \to E \times_B E$. Indeed, this is the canonical morphism between the pullbacks. We can summarize this situation in the following diagram
\[\begin{tikzcd}
	X & E \\
	{P_f} & {P_p} \\
	{X \times_YX} & {E\times_BE} \\
	Y & B
	\arrow["{q_f}", from=1-1, to=1-2]
	\arrow["\sim"', tail, from=1-1, to=2-1]
	\arrow["\lrcorner"{anchor=center, pos=0.125}, draw=none, from=1-1, to=2-2]
	\arrow["\sim", tail, from=1-2, to=2-2]
	\arrow["r", from=2-1, to=2-2]
	\arrow["{(\partial_0',\partial_1')}"', two heads, from=2-1, to=3-1]
	\arrow["\lrcorner"{anchor=center, pos=0.125}, draw=none, from=2-1, to=3-2]
	\arrow["{(\partial_0,\partial_1)}", two heads, from=2-2, to=3-2]
	\arrow["{q_f\times_{\chi_f}q_f}", from=3-1, to=3-2]
	\arrow["{f\times f}"', two heads, from=3-1, to=4-1]
	\arrow["\lrcorner"{anchor=center, pos=0.125}, draw=none, from=3-1, to=4-2]
	\arrow["{p \times p}", two heads, from=3-2, to=4-2]
	\arrow["{\chi_f}"', from=4-1, to=4-2]
\end{tikzcd}\]

\begin{lemma}

    \label{lemma: characterisation homotopy monomorphisms classified by univalent fibration}
    A fibration $f \colon X \to Y$ classified by some univalent fibration $p \colon E \to B$ in a $\pi$-tribe via a diagram
    \[\begin{tikzcd}
    	X & E \\
    	Y & B
    	\arrow["{q_f}", from=1-1, to=1-2]
    	\arrow["f"', two heads, from=1-1, to=2-1]
    	\arrow["\lrcorner"{anchor=center, pos=0.125}, draw=none, from=1-1, to=2-2]
    	\arrow["p", two heads, from=1-2, to=2-2]
    	\arrow["{\chi_f}"', from=2-1, to=2-2]
    \end{tikzcd}\]
    is a homotopy monomorphism if and only if there is a homotopy
    \[\begin{tikzcd}
    	{X \times_Y X} && {P_p} \\
    	& {E\times_B E}
    	\arrow["H", from=1-1, to=1-3]
    	\arrow["{q_f\times_{\chi_f}q_f}"', from=1-1, to=2-2]
    	\arrow["{(\partial_0,\partial_1)}", two heads, from=1-3, to=2-2]
    \end{tikzcd}\]
where $P_p$ is a fibered path object of $p$ in the tribe $\cC/\!\!/B$.
    
\end{lemma}

\begin{proof}
 Let $(\partial_0',\partial_1') \colon P_f \to X \times_Y X$ denote the pullback of $(\partial_0,\partial_1)$ along $q_f\times_{\chi_f}q_f$. Now, sections of this pullback are in bijection with morphisms $H \colon q_f\times_{\chi_f}q_f \to (\partial_0,\partial_1)$ over $E \times_B E$.  By Lemma \ref{lemma: different characterisations of homotopy monic fibration}, $f$ is a homotopy monomorphism if and only if the fibration $(\partial_0',\partial_1')$ has such a section. Thus, such a morphism $H$ exists if and only if $f$ is a homotopy monic fibration.
\end{proof}

\begin{definition}
\label{def: homotopy subobject classifier tribe}
    Let $S$ be a local class of fibrations. A \textit{homotopy subobject classifier} for $S$ is a homotopy monic fibration $\top \colon \mathsf{El}(\Omega_S) \to \Omega_S$ such that $\top$ is univalent and for every homotopy monic fibration $f \colon A \to X$ in $S$ there is a pullback
    \[\begin{tikzcd}
    	A & {\mathsf{El}(\Omega_S)} \\
    	X & {\Omega_S}
    	\arrow[from=1-1, to=1-2]
    	\arrow["f"', two heads, from=1-1, to=2-1]
    	\arrow["\lrcorner"{anchor=center, pos=0.125}, draw=none, from=1-1, to=2-2]
    	\arrow["\top", two heads, from=1-2, to=2-2]
    	\arrow[from=2-1, to=2-2]
    \end{tikzcd}\]
\end{definition}
    
\begin{remark}
    In an ordinary category, a subobject classifier $1 \to \Omega$ is a monomorphism such that every monomorphism $A \to X$ is (up to isomorphism) a pullback of the subobject classifier along a unique map $X \to \Omega$  inducing by pullback a representation
    $$\mathsf{Hom}(X, \Omega) \cong \mathsf{Sub}(X)$$
    of the contravariant $\Set$-valued  subobject functor $\mathsf{Sub}(-)$. For a homotopy subobject classifier, we want that every homotopy monomorphism is (up to homotopy equivalence) the pullback along a homotopy unique map $X \to \Omega$. This uniqueness requirement is encoded in our definition by univalence (Def. \ref{def:univalent fibration-categorical def}), namely by demanding that the inclusion of the identity $\delta_{\mathsf{El}(\Omega_S)}$ is a homotopy equivalence. This turns the object of equivalences $\myuline{\mathsf{Eq}}_{\Omega_S}(\mathsf{El}(\Omega_S))$ into a path object and maps into it into homotopies. Thus, we get by Remark \ref{remark: univalent fibration is small object classifier} and Proposition \ref{prop: pullback of univalent fibration is homotopy unique} via pullback a bijection

    \vspace{-0.5cm}

    $$\{\text{htpy. classes of morphism } X \to \Omega_S\} \cong \{\text{htpy. equiv. classes of htpy. monic fibrations } A \to X \in S\}.$$
    
\end{remark}

    We will now show that from a given univalent fibration $p \colon E \to B$ in a $\pi$-tribe $\cC$, we can always construct a homotopy subobject classifier  $\top \colon \mathsf{El}(\Omega_p) \to \Omega_p$ for the local class $S_p$ induced by $p$. The construction is analogous to the construction of an object of propositions for some universe in type theory that we saw in Definition \ref{def: type theory Prop}.

\begin{construction}

    \label{constr: homotopy subobject classifier from univalent fibration}

    Let $p \colon  E \to B$ be a univalent fibration in a $\pi$-tribe $\cC$. Let $p \times p$ be the product of $p$ with itself in $\cC/\!\!/B$ given by the pullback $E \times_B E \to B$. By partial local cartesian closedness there is a functor 
    $$\Pi_{p \times p} \colon \cC/\!\!/E \times_B E \to \cC/\!\!/B$$
    that is partial right adjoint to pullback 
    $$(p\times p)^* \colon \cC/B \to \cC/E \times_B E.$$
    Let 
    $$E \xrightarrow{\sim} P_p \xrightarrow{(\partial_0,\partial_1)} E \times_B E$$
    be a fibered path object factorisation of $p$ in $\cC/\!\!/B$. Define $pr_p \colon  \Omega_p \to B$ to be the fibration
    $$\Pi_{p\times p}(\partial_0,\partial_1)\colon \Pi_{p \times p}(P_p) \to B$$
    such that we have morphisms
    \[\begin{tikzcd}
    	{P_p} & {\Omega_p:=\Pi_{p \times p}(P_p)} \\
    	{E \times_B E} & B
    	\arrow["{(\partial_0,\partial_1)}"', two heads, from=1-1, to=2-1]
    	\arrow["{Pr_p := \Pi_{p\times p}(\partial_0,\partial_1)}", two heads, from=1-2, to=2-2]
    	\arrow["{p\times p}"', two heads, from=2-1, to=2-2]
    \end{tikzcd}\]
    Define $\top \colon \mathsf{El}(\Omega_p) \to \Omega_p$ to be the pullback

    \[\begin{tikzcd}
    	{\mathsf{El}(\Omega_p)} & E \\
    	{\Omega_p} & B
    	\arrow[two heads, from=1-1, to=1-2]
    	\arrow["\top"', two heads, from=1-1, to=2-1]
    	\arrow["\lrcorner"{anchor=center, pos=0.125}, draw=none, from=1-1, to=2-2]
    	\arrow["p", two heads, from=1-2, to=2-2]
    	\arrow["{Pr_p}"', two heads, from=2-1, to=2-2]
    \end{tikzcd}\]

\end{construction}

We will show that $\top$ is in fact homotopy subobject classifier for $S_p$. Observe that $Pr_p \colon \Omega_p \to B$ has by adjointness the universal property that there is a correspondence of morphisms $\overline{H}$ and $H$
\[\begin{tikzcd}
	Y && {\Omega_p} && {X \times_Y X} && {P_p} \\
	& B &&&& {E\times_B E}
	\arrow["{\overline{H}}", from=1-1, to=1-3]
	\arrow["{\chi_f}"', from=1-1, to=2-2]
	\arrow["{Pr_p}", two heads, from=1-3, to=2-2]
	\arrow["H", from=1-5, to=1-7]
	\arrow["{q_f\times_{\chi_f}q_f}"', from=1-5, to=2-6]
	\arrow["{(\partial_0,\partial_1)}", two heads, from=1-7, to=2-6]
\end{tikzcd}\]
for any pullback

    \[\begin{tikzcd}
    	X & E \\
    	Y & B
    	\arrow["{q_f}", from=1-1, to=1-2]
    	\arrow["f"', two heads, from=1-1, to=2-1]
    	\arrow["\lrcorner"{anchor=center, pos=0.125}, draw=none, from=1-1, to=2-2]
    	\arrow["p", two heads, from=1-2, to=2-2]
    	\arrow["{\chi_f}"', from=2-1, to=2-2]
    \end{tikzcd}\]
and so in light of Lemma \ref{lemma: characterisation homotopy monomorphisms classified by univalent fibration} we get that the morphism $Pr_p$ is a fibration representing homotopy monic fibrations classified by $p$. We first show that $\top$ constructed this way is itself a homotopy monomorphism.

\begin{lemma}
\label{lemma: top is homotopy monomorphism}
    $\top \colon \mathsf{El}(\Omega_p) \to \Omega_p$ is a homotopy monomorphism.
\end{lemma}

\begin{proof}

This follows directly using Lemma \ref{lemma: characterisation homotopy monomorphisms classified by univalent fibration}: The identity fits into the following diagram on the left, inducing the desired homotopy $K$ by adjoint transposition on the right which witnesses that the map $\top$ is a homotopy monomorphism.
\[\begin{tikzcd}
	{\Omega_p} && {\Omega_p} && {\mathsf{El}(\Omega_p)\times_{\Omega_p} \mathsf{El}(\Omega_p)} && {P_p} \\
	& B &&&& {E \times_B E}
	\arrow["id", from=1-1, to=1-3]
	\arrow["{Pr_p}"', from=1-1, to=2-2]
	\arrow["{Pr_p}", from=1-3, to=2-2]
	\arrow["K", from=1-5, to=1-7]
	\arrow[from=1-5, to=2-6]
	\arrow["{(\partial_0,\partial_1)}", from=1-7, to=2-6]
\end{tikzcd}\]
\end{proof}

Next, we will show that $Pr_p \colon \Omega_p \to B$ is itself a homotopy monomorphism for any fibration $p$. This is the categorical version of the type-theoretic proof that the type of being a proposition is itself a proposition, i.e. $\mathtt{isProp}(\mathtt{isProp}(X))$ is always inhabited. 

\begin{proposition}

    \label{prop: homotopy subobject classifier is a proposition}
    If $p \colon X\to Y$ is a univalent fibration in a $\pi$-tribe $\cC$, then the morphism $Pr_p \colon \Omega_p \to B$ is itself a homotopy monomorphism.
\end{proposition}

\begin{proof}
    By Lemma \ref{lemma: different characterisations of homotopy monic fibration} it suffices to show that for any $\overline{H},\overline{K} \colon Y \to \Omega_p$ such that $Pr_p\overline{H} = Pr_p\overline{K} = \chi_f$ is some morphism $\chi_f\colon Y \to B$, we have $\overline{H}  \sim \overline{K} $. Denote as $f := \chi_f^*p$ given by
    \[\begin{tikzcd}
    	X & E \\
    	Y & B
    	\arrow["{q_f}", from=1-1, to=1-2]
    	\arrow["f"', two heads, from=1-1, to=2-1]
    	\arrow["\lrcorner"{anchor=center, pos=0.125}, draw=none, from=1-1, to=2-2]
    	\arrow["p", two heads, from=1-2, to=2-2]
    	\arrow["{\chi_f}"', from=2-1, to=2-2]
    \end{tikzcd}\]
    By the universal property of $\Omega_p$, $\overline{H}$ and $\overline{K}$ correspond to morphisms $H$ and $K$ from $q_f\times_{\chi_f}q_f$ to $ (\partial_0, \partial_1)$ over $E \times_B E$. These correspond bijectively to sections $h$ and $k$ of the pullback $(\partial_0',\partial_1') = (q_f\times_{\chi_f}q_f)^*(\partial_0, \partial_1)$ which is the induced fibration $P_f \to X \times_Y X$.

By Lemma \ref{lemma: different characterisations of homotopy monic fibration}, $(\partial_0',\partial_1')$ is therefore a trivial fibration. However, any two sections of a trivial fibration are homotopic, hence $h \simeq k$, from which follows that $H \sim K$ since $H = rh$ and $K=rk$ for $r$ being the induced map $P_f \to P_p$. Since the homotopy relation is stable under adjoint transposition along the partial dependent product adjunction by Lemma \ref{lemma: different characterisations of function extensionality}, we get $\overline{H} \sim \overline{K}$ and so $Pr_p$ is a homotopy monomorphism. 
\end{proof}

\begin{theorem}
\label{prop: constructed top is homotopy subobject classifier}
        If $p \colon E \to B$ is univalent in a $\pi$-tribe $\cC$, then $\top \colon \mathsf{El}(\Omega_p) \to \Omega_p$ is homotopy subobject classifier for $S_p$, i.e. for homotopy monic fibrations that are classified by the univalent fibration $p$.
\end{theorem}

\begin{proof}

By Lemma \ref{lemma: top is homotopy monomorphism}, $\top$ is a homotopy monic fibration. Let $f \colon X \to Y$ be any homotopy monic fibration that is classified by the univalent fibration $p$. Then there is a pullback

\[\begin{tikzcd}
    X & E \\
    Y & B
    \arrow["{q_f}", from=1-1, to=1-2]
    \arrow["f"', two heads, from=1-1, to=2-1]
    \arrow["\lrcorner"{anchor=center, pos=0.125}, draw=none, from=1-1, to=2-2]
    \arrow["p", two heads, from=1-2, to=2-2]
    \arrow["{\chi_f}"', from=2-1, to=2-2]
\end{tikzcd}\]
Since $f$ is a homotopy monic fibration, there is a homotopy by Lemma \ref{lemma: characterisation homotopy monomorphisms classified by univalent fibration}
\[\begin{tikzcd}
	{X \times_Y X} && {P_p} \\
	& {E\times_B E}
	\arrow["H", from=1-1, to=1-3]
	\arrow["{q_f\times_{\chi_f}q_f}"', from=1-1, to=2-2]
	\arrow["{(\partial_0,\partial_1)}", two heads, from=1-3, to=2-2]
\end{tikzcd}\]
By adjoint transposition, this corresponds to a morphism in the slice over $B$
\[\begin{tikzcd}
	Y && {\Pi_{p\times p}(P_p) = \Omega_p} \\
	& B
	\arrow["{\overline{H}}", from=1-1, to=1-3]
	\arrow["{\chi_f}"', from=1-1, to=2-2]
	\arrow["{Pr_p}", from=1-3, to=2-2]
\end{tikzcd}\]
Thus, we can fit this data into a commutative diagram
\[\begin{tikzcd}
	X & {\mathsf{El}(\Omega_p)} & E \\
	Y & {\Omega_p} & B
	\arrow["h"', dotted, from=1-1, to=1-2]
	\arrow["{q_f}", curve={height=-12pt}, from=1-1, to=1-3]
	\arrow["f"', two heads, from=1-1, to=2-1]
	\arrow[two heads, from=1-2, to=1-3]
	\arrow["\top"', two heads, from=1-2, to=2-2]
	\arrow["\lrcorner"{anchor=center, pos=0.125}, draw=none, from=1-2, to=2-3]
	\arrow["p", two heads, from=1-3, to=2-3]
	\arrow["{\overline{H}}", from=2-1, to=2-2]
	\arrow["{\chi_f}"', curve={height=18pt}, from=2-1, to=2-3]
	\arrow["{Pr_p}", two heads, from=2-2, to=2-3]
\end{tikzcd}\]
where the dashed arrow $h$ is the unique arrow induced by the universal property of the pullback. Since the right square and the outer rectangle of this diagram are a pullback, so is the left square by the pullback lemma.

It remains to show that $\top$ is univalent. By Proposition \ref{prop: homotopy subobject classifier is a proposition}, $Pr_p$ is a homotopy monomorphism. Since $p$ is univalent and $\top$ is the pullback of $p$ along $Pr_p$, Proposition \ref{prop: univalence is preserved and reflected by pullback along homotopy monomorphisms} implies that $\top$ is univalent. Hence, $\top$ is a homotopy subobject classifier.
\end{proof}

\subsection{Models of Univalence are Univalent Tribes}

Let $\cC$ be a categorical model of dependent type theory with a universe $U.\mathsf{El}$ satisfying the univalence axiom. We will show that $p_{El} \colon  U.\mathsf{El} \to U$ is a univalent fibration in the $\pi$-tribe structure on $\cC$.

\begin{lemma}
    \label{lemma: equivalence context is object of equivalences}
    If $\Gamma.A$ and $\Gamma.B$ are context extensions of $\Gamma$, then there is a canonical homotopy equivalence $\Gamma.\mathsf{Equiv(A,B)} \simeq \myuline{\mathsf{Eq}}(\Gamma.A,\Gamma.B)$ over $\Gamma$.
\end{lemma}
\begin{proof}
Suppose that $\myuline{\mathsf{Eq}}(\Gamma.A,\Gamma.B)$ was obtained from fibered path objects $P_{\Gamma.A}$ and $P_{\Gamma.B}$ in $\cC/\!\!/\Gamma$. The identity context $\Gamma.[B,B].p_{[B,B]}^*[B,B].Id_{\Gamma.[B,B]}$ is a fibered path object for $\Gamma.[B,B] \cong \uHom_{\cC/\!\!/\Gamma}(\Gamma.B,\Gamma.B)$. By Lemma \ref{lemma: different characterisations of function extensionality}, the map $\Gamma.[B,B] \to \uHom(\Gamma.B,P_{\Gamma.B})$ is a homotopy equivalence. Thus, we obtain lift in
\[\begin{tikzcd}
	{\Gamma.[B,B]} & {\uHom(\Gamma.B,P_{\Gamma.B})} \\
	{\Gamma.[B,B].p_{[B,B]}^*[B,B].Id_{\Gamma.[B,B]}} & {\Gamma.[B,B].p_{[B,B]}^*[B,B] = \Gamma.[B,B]\times_{\Gamma} \Gamma.[B,B]}
	\arrow["\sim", tail, from=1-1, to=1-2]
	\arrow["\sim"', tail, from=1-1, to=2-1]
	\arrow["{({\partial_0}_*,{\partial_1}_*)}", two heads, from=1-2, to=2-2]
	\arrow["\gamma"{description}, dashed, from=2-1, to=1-2]
	\arrow[two heads, from=2-1, to=2-2]
\end{tikzcd}\]
which is a homotopy equivalence by 2-out-of-3. By definition of the counit, the maps $\mathsf{comp}$ and $c_{\Gamma.B,\Gamma.A,\Gamma.B}$ coincide. Therefore, by the definition of $\Gamma.[A,B].\mathsf{RInv}_{A,B}$ and $\myuline{\mathsf{RInv}}(\Gamma.A,\Gamma.B)$ we have pullbacks
\[\begin{tikzcd}
	{\Gamma.[A,B].\mathsf{RInv}_{A,B}} & {\Gamma.[B,B].p_{[B,B]}^*[B,B].\mathsf{Id}_{\Gamma.[B,B]}} \\
	{\myuline{\mathsf{RInv}}(\Gamma.A,\Gamma.B)} & {\uHom(\Gamma.B,P_{\Gamma.B})} \\
	{\Gamma.[A,B].p_{[A,B]}^*[B,A]} & {\Gamma.[B,B].p_{[B,B]}^*[B,B]} \\
	& {}
	\arrow[from=1-1, to=1-2]
	\arrow["\sim"', from=1-1, to=2-1]
	\arrow["\lrcorner"{anchor=center, pos=0.125}, draw=none, from=1-1, to=4-2]
	\arrow["{\gamma \sim}", from=1-2, to=2-2]
	\arrow[from=2-1, to=2-2]
	\arrow[two heads, from=2-1, to=3-1]
	\arrow["\lrcorner"{anchor=center, pos=0.125}, draw=none, from=2-1, to=4-2]
	\arrow[two heads, from=2-2, to=3-2]
	\arrow["{(c_{\Gamma.B,\Gamma.A,\Gamma.B},const_{\Gamma.B})}"', from=3-1, to=3-2]
\end{tikzcd}\]

\vspace{-0.5cm}

\noindent where the left upper vertical map is a homotopy equivalence by the exactness of the pullback functor between fibrant slices. Similarly, using in addition the fact that $\mathsf{exch}_{[A,B],[B,A]}$ is an isomorphism, we obtain a homotopy equivalence $\Gamma.[A,B].\mathsf{LInv}_{A,B} \to \myuline{\mathsf{LInv}}(A,B)$. These induce an equivalence $\Gamma.[A,B].\mathsf{isEquiv_{A,B}} \simeq \myuline{\mathsf{Eq}}(\Gamma.A,\Gamma.B)$ over $[A,B]$. Using the fact that the $\Sigma$-structure can be given by composing dependent projections, we get an equivalence $\Gamma.\mathsf{Equiv}(A,B) \simeq \myuline{\mathsf{Eq}}(\Gamma.A,\Gamma.B)$ over $\Gamma$.
\end{proof}

\begin{proposition}
\label{prop: univalent universe is univalent fibration}
If $U.\mathsf{El}$ is a univalent universe in $\cC$, then $p_{El} \colon  U.\mathsf{El} \to U$ is a univalent fibration.
\end{proposition}
\begin{proof}

Suppose $U.\mathsf{El}$ is a univalent universe in $\cC$. Then, there is a map $\mathsf{uvt}_{U,\mathsf{El}} \colon U.p_U^*U \to \mathsf{isUvt}(U,\mathsf{El})$. Let
$$q := q_{p_{\mathsf{Id}_U}, \mathsf{Equiv(p_{p_U^*U}^*\mathsf{El}, q_{p_U,U}^*\mathsf{El})}}$$
be the pullback projection, such that we have a map $w_U$ defined as the composite
$$q\mathsf{idToEquiv} \colon  U.p_U^*U.\mathsf{Id}_U \to U.p_u^*U.\mathsf{Equiv}(p_{p_U^*U}^*\mathsf{El}, q_{p_U,U}^*\mathsf{El}).$$
Then, $\mathsf{uvt}_{U,\mathsf{El}}$ induces by definition of $\mathsf{isEquiv}$ and the $\Pi$-structure two maps
$$f,g \colon U.p_U^*U.\mathsf{Equiv}(p_{p_U^*U}^*\mathsf{El}, q_{p_U,U}^*\mathsf{El}) \to U.p_U^*U.\mathsf{Id}_U$$ 
such that the maps $(w_U f,id)$ and $(g w_U,id)$ factor each through some identity context, which we may assume is a path object. Thus, $w_U$ is a homotopy equivalence. By Lemma \ref{lemma: equivalence context is object of equivalences}, we have an equivalence $U.p_U^*U.\mathsf{Equiv}(p_{p_U^*U}^*\mathsf{El}, q_{p_U,U}^*\mathsf{El})  \simeq \myuline{\mathsf{Eq}}_U(U.\mathsf{El})  $ over $U.p_U^*U$. Then $w_U$ fits into the following commutative diagram
\[\begin{tikzcd}
	U & {U.p_U^*U.\mathsf{Equiv}(p_{p_U^*U}^*El, q_{p_U,U}^*El)} & {\myuline{\mathsf{Eq}}_U(U.\mathsf{El})} \\
	{U.p_U^*U.\mathsf{Id}_U} && {U.p_U^*U = U \times U}
	\arrow["{\mathsf{id}^{\delta}}"', from=1-1, to=1-2]
	\arrow["{\delta_{U.El}}", curve={height=-24pt}, from=1-1, to=1-3]
	\arrow["\sim"', tail, from=1-1, to=2-1]
	\arrow["\sim", from=1-2, to=1-3]
	\arrow[two heads, from=1-3, to=2-3]
	\arrow["{\sim w_U}"{description}, from=2-1, to=1-2]
	\arrow[two heads, from=2-1, to=2-3]
\end{tikzcd}\]
and, by 2-out-of-3, $\mathsf{id}^{\delta}$ is a homotopy equivalence, hence so is $\delta_{U.\mathsf{El}}$ by 2-out-of-3.  Thus, $p_{\mathsf{El}} \colon U.\mathsf{El} \to U$ is a univalent fibration.
\end{proof}

\begin{corollary}
\label{cor: categorical model has enough univalent universes}
    If a categorical model $\cC$ with $\Sigma$, $\Pi$ and $\mathtt{Id}$-structure has enough univalent universes, then $\cC$ has enough univalent fibrations.
\end{corollary}
\begin{proof}
    Since any fibration is isomorphic to a dependent projection, we can find for any fibration $f$ a univalent fibration $p_{\mathsf{El}} \colon U.\mathsf{El} \to U$ of which $f$ is a pullback. It remains to show that $S_{p_{\mathsf{El}}}$ is closed.
    
    Since the universe $U.\mathsf{El}$ is by assumption closed under the $\Sigma$-structures and any composite of two fibrations is isomorphic to a $\Sigma$-context, we see that $S_{p_{\mathsf{El}}}$ is closed under composition. The closure under $\Pi_p$ for any fibration $p \in S_{p_{\mathsf{El}}}$ follows directly from the closure of $U.\mathsf{El}$ under the $\Pi$-structure.
\end{proof}

\begin{lemma}
    \label{lemma: resizing in model gives almost homotopy unique pullbacks}
    If a categorical model $\cC$ with $\Sigma$, $\Pi$ and $\mathtt{Id}$-structures and enough univalent universes satisfies propositional resizing, then for any univalent  $p_{\mathsf{El}_n} \colon U_{n}.\mathsf{El}_n \to U_n$ with corresponding homotopy subobject classifier $\top \colon \mathsf{El}(\Omega_n) \to \Omega$ there is a homotopy commutative square
    \[\begin{tikzcd}
    	{\mathsf{El}(\Omega_n)} & {\mathsf{El}(\Omega_0)} \\
    	{\Omega_n} & {\Omega_0}
    	\arrow["\sim", from=1-1, to=1-2]
    	\arrow[two heads, from=1-1, to=2-1]
    	\arrow[two heads, from=1-2, to=2-2]
    	\arrow["\sim"', from=2-1, to=2-2]
    \end{tikzcd}\]
    in which the top and the bottom morphisms are homotopy equivalences.
\end{lemma}
\begin{proof}
    Unwinding the definition of propositional resizing in $\cC$, all maps $\Omega_{n} \to \Omega_{n+1}$ are homotopy equivalences, thus by stability of homotopy equivalences under pullback along fibrations and pullback pasting we get a pullback square
    \[\begin{tikzcd}
    	{\mathsf{El}(\Omega_n)} & {\mathsf{El}(\Omega_{n+1})} \\
    	{\Omega_{n}} & {\Omega_{n+1}}
    	\arrow["\sim", from=1-1, to=1-2]
    	\arrow[two heads, from=1-1, to=2-1]
    	\arrow["\lrcorner"{anchor=center, pos=0.125}, draw=none, from=1-1, to=2-2]
    	\arrow[two heads, from=1-2, to=2-2]
    	\arrow["\sim"', from=2-1, to=2-2]
    \end{tikzcd}\]
    in which the top and bottom horizontal maps are homotopy equivalences. Choosing homotopy inverses to both of them and composing the resulting homotopy commutative squares gives the result.
\end{proof}

\subsection{Internal Homotopy Initial Objects and Pushouts}
\begin{definition}\label{def:internal homotopy initial object}\normalfont
An \textit{internal homotopy initial object} in a $\pi$-tribe $\cC$ is an object $I$ in it such that $\uHom(I,X)$ is contractible for every $X\in\cC$. 
\end{definition}

\begin{construction}\label{constr:Cocone object simp}\normalfont
Given a span $A\xleftarrow{f}C\xrightarrow{g}B$ of morphisms in a $\pi$-tribe $\cC$ and another object $D$, define $\Cocone(f,g,D)$ via the pullback square
\[\begin{tikzcd}
	{\Cocone(f,g,D)} & {\uHom(C,P_D)} \\
	{\uHom(A,D) \times \uHom(B,D)} & {\uHom(C,D) \times \uHom(C,D)} \\
	& {}
	\arrow[from=1-1, to=1-2]
	\arrow[two heads, from=1-1, to=2-1]
	\arrow["\lrcorner"{anchor=center, pos=0.125}, draw=none, from=1-1, to=3-2]
	\arrow["{({\partial_0}_*,{\partial_1}_*)}", two heads, from=1-2, to=2-2]
	\arrow["{f^*\times g^*}"', from=2-1, to=2-2]
\end{tikzcd}\]
Given another object $E$, define a morphism $ap\colon\uHom(D,E)\to\uHom(P_D,P_E)$ as the adjoint transpose to the dashed lift in the diagram
\[\begin{tikzcd}
	{\uHom(D,E)\times D} & E & {P_E} \\
	{\uHom(D,E)\times P_D} & {(\uHom(D,E)\times D)\times (\uHom(D,E)\times D)} & {E\times E}
	\arrow["{\mathrm{ev}_{D,E}}", from=1-1, to=1-2]
	\arrow["\sim"', tail, from=1-1, to=2-1]
	\arrow["\sim", tail, from=1-2, to=1-3]
	\arrow[two heads, from=1-3, to=2-3]
	\arrow[dashed, from=2-1, to=1-3]
	\arrow[from=2-1, to=2-2]
	\arrow["{\mathrm{ev}_{D,E}\times \mathrm{ev}_{D,E}}"', from=2-2, to=2-3]
\end{tikzcd}\]
which exists since the left-hand morphism is anodyne as a product of anodyne maps. If $1\to\Cocone(f,g,Q)$ is a global section, consisting of the data of $i_A\colon A\to Q$, $i_B\colon B\to Q$ and $H\colon C\to P_Q$, then let $\uHom(Q,D) \to \Cocone(f,g,D)$ be the induced arrow in the commutative diagram
\[\begin{tikzcd}
	{\uHom(Q,D)} & {\uHom(P_Q,P_D)} & \\
	& {\Cocone(f,g,D)} & {\uHom(C,P_D)} \\
	& {\uHom(A,D) \times \uHom(B,D)} & {\uHom(C,D) \times \uHom(C,D)} \\
	&& {}
	\arrow["ap", from=1-1, to=1-2]
	\arrow[dashed, from=1-1, to=2-2]
	\arrow["{(i_B^*,i_C^*)}"', curve={height=18pt}, from=1-1, to=3-2]
	\arrow["{H^*}", curve={height=-6pt}, from=1-2, to=2-3]
	\arrow[from=2-2, to=2-3]
	\arrow[two heads, from=2-2, to=3-2]
	\arrow["\lrcorner"{anchor=center, pos=0.125}, draw=none, from=2-2, to=4-3]
	\arrow["{({\partial_0}_*,{\partial_1}_*)}", two heads, from=2-3, to=3-3]
	\arrow["{ f^* \times g^*}"', from=3-2, to=3-3]
\end{tikzcd}\]
\end{construction}

\begin{definition}\label{def:internal homotopy pushout new}\normalfont
Given a span $A\xleftarrow{f}C\xrightarrow{g}B$ in a $\pi$-tribe $\cC$, an \textit{internal homotopy pushout} of this data is an object $Q\in\cC$ together with a global section $(i_A,i_B,H)\colon 1\to\Cocone(f,g,Q)$ such that for every object $D\in\cC$ the map
$$\uHom(Q,D) \to\Cocone(f,g,D)$$
is a homotopy equivalence.
\end{definition}

\begin{lemma}
\label{lemma: categorical model has all internal homotopy pushouts}
    If $\cC$ is a categorical model of dependent type theory with 0-type structures (\cite[B.1.5.]{kapulkin_simplicial_2021}) and all homotopy pushouts, then it is in particular a $\pi$-tribe with an internal homotopy initial object and all internal homotopy pushouts.
\end{lemma}
\begin{proof}
    For homotopy initial objects, one considers the object $1.0$ given by the 0-type structure for the terminal object $1$. Concerning homotopy pushouts, by solving a lifting problem analogous to the one in the proof in \ref{lemma: equivalence context is object of equivalences} we see that the relevant map in question is a homotopy equivalence.
\end{proof}

%% file: sect4-univalent-infinity-categories.tex
\section{Univalence in $\infty$-Categories}
\label{sect: univalence in infinity-categories}

We will now define the notion of a univalent morphism in an $\infty$-category and prove that every univalent fibration in a $\pi$-tribe becomes a univalent morphism in the $\infty$-localisation. 

\subsection{The Right Fibration and the Object of Equivalences}

 Our goal in this section is to construct an object of equivalences $\myuline{\mathpzc{BiInv}}(a,b)$ in an arbitrary cartesian closed and finitely complete $\infty$-category. We show that it induces a representation of the right fibration of equivalences $Eq(a,b) \to \cCi$ that will be defined below, and in terms of which we will define univalence in an $\infty$-category. Hence, we show that  $Eq(a,b) \to \cCi$ is always representable in a cartesian closed and finitely complete $\infty$-category.

For any $\infty$-category $\cCi$ with objects $a$ and $b$ there is a right fibration $Map(a,b) \to \cCi$ defined via the pullback

    \[\begin{tikzcd}
    	{Map(a,b)} & {\cCi/b} \\
    	\cCi & \cCi
    	\arrow[from=1-1, to=1-2]
    	\arrow[from=1-1, to=2-1]
    	\arrow["\lrcorner"{anchor=center, pos=0.125}, draw=none, from=1-1, to=2-2]
    	\arrow[from=1-2, to=2-2]
    	\arrow["{- \times a}"', from=2-1, to=2-2]
    \end{tikzcd}\]
    of the representable right fibration $\cCi/b \to \cCi$ along the product functor $-\times a$.  Its fiber over an object $x$ is the space $\mathpzc{Map}_{\cCi}(x \times a, b) \simeq \mathpzc{Map}_{\cCi/x}(x \times a, x \times b)$. The right fibration $Map(a,b) \to \cCi$ is representable if and only if $\cCi$ is cartesian closed by \cite[Prop. 2.1]{gepner_univalence_2017} with the representing object being $\myuline{\mathpzc{Map}}(a,b)$.

\begin{definition}[{\autocite[§2.6]{gepner_univalence_2017}}]

\label{def: fibration of equivalences in infty cat}
    Let $\cCi$ be an $\infty$-category and $a$ and $b$ be objects in $\cCi$. Define
    $$\mathpzc{Eq}_{\cCi}(a,b):=\mathpzc{Map}_{\cCi^{\mathsf{gpd}}}(a,b)$$
    to be the space of maps between $a$ and $b$ in the maximal subgroupoid ${\cCi}^{\mathsf{gpd}}$ of $\cCi$. This space is a subobject $\mathpzc{Eq}_{\cCi}(a,b)\subseteq \mathpzc{Map}_{\cCi}(a,b)$ of the mapping space. It is equipped with a right fibration $Eq(a,b) \to \cCi$ whose fiber over $x \in \cCi$ is the $\infty$-groupoid $\mathpzc{Eq}_{\cCi/x}(x \times a, x \times b) \subseteq \mathpzc{Map}_{\cCi/x}(x \times a, x \times b)$. This is a subfibration of $Map(a,b) \to \cCi$ and the map $\mathpzc{Eq}_{\cCi/x}(x \times a, x \times b) \to \mathpzc{Map}_{\cCi/x}(x \times a, x \times b)$ is (-1)-truncated.
\end{definition}

Similarly, we have as a local version for any objects $p_1 \colon x \to b$ and $p_2 \colon y \to b$ in $\cCi/b$ a right fibration $Eq_{\cCi/b}(p_1,p_2) \to \cCi/b$ whose fiber over $f \colon a \to b$ is the $\infty$-groupoid $\mathpzc{Eq}_{\cCi/b}(f^*p_1, f^*p_2) \subseteq \mathpzc{Map}_{\cCi/b}(f^*p_1, f^*p_2)$.

Next, we need to define the $\infty$-categorical analog of the map $\mathsf{comp}$ in a contextual category and the map $c_{A,B,C}$ in a $\pi$-tribe. Given a cartesian closed $\infty$-category $\cC$ with counit maps $\epsilon_{a,b} \colon  \myuline{\mathpzc{Map}}(a,b) \times a \to b$, we define for any objects $a,b$ and $c$ the composition map
$$c_{a,b,c, \infty} \colon  \myuline{\mathpzc{Map}}(b,c)   \times \myuline{\mathpzc{Map}}(a,b) \to \myuline{\mathpzc{Map}}(a,c) $$
to be the adjoint transpose of a choice of composition
$$\myuline{\mathpzc{Map}}(b,c) \times \myuline{\mathpzc{Map}}(a,b) \times a  \xrightarrow{ id_{\myuline{\mathpzc{Map}}(b,c)} \times \epsilon_{a,b}}   \myuline{\mathpzc{Map}}(b,c)  \times b \xrightarrow{\epsilon_{b,c}} c.$$
We claim that this composition map induces a well-defined composition law on mapping spaces defined as the dashed arrow
\[\begin{tikzcd}
	{\mathpzc{\mathpzc{Map}}(x, \myuline{\mathpzc{Map}}(b,c) \times \myuline{\mathpzc{Map}}(a,b) )} & {\mathpzc{\mathpzc{Map}}_{\cCi/x}(x \times b, x \times c) \times \mathpzc{\mathpzc{Map}}_{\cCi/x}(x \times a, x \times b)} \\
	{\mathpzc{\mathpzc{Map}}(x, \myuline{\mathpzc{Map}}(a,c))} & {\mathpzc{\mathpzc{Map}}_{\cCi/x}(x \times a, x \times c) }
	\arrow["{(c_{a,b,c, \infty})_{*}}"', from=1-1, to=2-1]
	\arrow["\simeq"', from=1-2, to=1-1]
	\arrow["{C_{a,b,c,\infty}}", dashed, from=1-2, to=2-2]
	\arrow["\simeq"', from=2-1, to=2-2]
\end{tikzcd}\]
To see that, we pass to the homotopy category $h\cCi$. Observe that products in $h\cCi$ can be computed in $\cCi$. Moreover, the adjunction giving cartesian closedness on $\cC$ induces an adjunction on $h\cCi$ such that $\mathsf{Hom}_{h\cCi}(x \times y, z) \cong \mathsf{Hom}_{h\cCi}(x, \myuline{\mathpzc{Map}}(y, z))$ with counit $[\epsilon]$. This follows from the fact that $\cCi \mapsto h\cCi$ is a strict functor from the homotopy 2-category of quasicategories $h_2\mathbf{QCat}$ to $\Cat$, hence it preserves adjunctions. So in particular $[c_{a,b,c,\infty}]$ is the transpose of the morphism $[\epsilon]_{b,c} ([\epsilon]_{a,b} \times [id_{\myuline{\mathpzc{Map}}(b,c)}])$. Thus, given $f \in \mathpzc{Map}_{\cCi/x}(x \times a, x \times b)$ and $g \in \mathpzc{Map}_{\cCi/x}(x \times b, x \times c)$,  we get by the ordinary 1-category computation of the transpose of $[c_{a,b,c,\infty}](\overline{[g]},\overline{[f]})$ that
$$[C_{a,b,c, \infty}(g,f)]  = \overline{[c_{a,b,c,\infty}](\overline{[g]},\overline{[f]})}= [g][f] $$
in $h\cCi$ as desired. A further calculation in the homotopy category shows that this law is associative and unital up to homotopy.

\begin{construction}[Cf. Constr. \ref{constr: object of equiv}]
\label{construction: object of equivalences infinity cat}
    Let $a$ and $b$ be objects in a cartesian closed and finitely complete $\infty$-category $\cCi$. Define $\myuline{\mathpzc{RInv}}(a,b)$ and $\myuline{\mathpzc{LInv}}(a,b)$ via pullbacks
    \[\begin{tikzcd}[every label/.append style={outer sep=0.2cm}]
    	{\myuline{\mathpzc{RInv}}(a,b)} & {\myuline{\mathpzc{Map}}(b,b)} \\
    	{\myuline{\mathpzc{Map}}(a,b) \times \myuline{\mathpzc{Map}}(b,a)} & {\myuline{\mathpzc{Map}}(b,b) \times \myuline{\mathpzc{Map}}(b,b)} \\
    	\\
    	& {}
    	\arrow[from=1-1, to=1-2]
    	\arrow[from=1-1, to=2-1]
    	\arrow[from=1-1, to=2-1]
    	\arrow["\lrcorner"{anchor=center, pos=0.125}, draw=none, from=1-1, to=4-2]
    	\arrow["{(id,id)}"{pos=0.7}, from=1-2, to=2-2]
    	\arrow["{(c_{b,a,b,\infty},const_b)}"', from=2-1, to=2-2]
    \end{tikzcd}\]

    \vspace{-1.5cm}

    \[\begin{tikzcd}[every label/.append style={outer sep=0.2cm}]
    	{\myuline{\mathpzc{LInv}}(a,b)} & {\myuline{\mathpzc{Map}}(a,a)} \\
    	{\myuline{\mathpzc{Map}}(b,a) \times \myuline{\mathpzc{Map}}(a,b)} & {\myuline{\mathpzc{Map}}(a,a) \times \myuline{\mathpzc{Map}}(a,a)} \\
    	\\
    	& {}
    	\arrow[from=1-1, to=1-2]
    	\arrow[from=1-1, to=2-1]
    	\arrow[from=1-1, to=2-1]
    	\arrow["\lrcorner"{anchor=center, pos=0.125}, draw=none, from=1-1, to=4-2]
    	\arrow["{(id,id)}"{pos=0.6}, from=1-2, to=2-2]
    	\arrow["{(c_{a,b,a,\infty},const_a)}"', from=2-1, to=2-2]
    \end{tikzcd}\]
    
    \vspace{-1.5cm}
    
    \noindent and define $\myuline{\mathpzc{BiInv}}(a,b)$ as the pullback
    \[\begin{tikzcd}
    	{\myuline{\mathpzc{BiInv}}(a,b)} & {\myuline{\mathpzc{LInv}}(a,b)} \\
    	{\myuline{\mathpzc{RInv}}(a,b)} & {\myuline{\mathpzc{Map}}(a,b)} \\
    	& {}
    	\arrow[from=1-1, to=1-2]
    	\arrow[from=1-1, to=2-1]
    	\arrow["\lrcorner"{anchor=center, pos=0.125}, draw=none, from=1-1, to=3-2]
    	\arrow[from=1-2, to=2-2]
    	\arrow[from=2-1, to=2-2]
    \end{tikzcd}\]

\end{construction}

    \vspace{-1.0cm}
    By adjointness and the fact that the mapping space functor $\mathpzc{Map}_{\cCi}(x,-) \colon \cCi \to \mathbf{Spc}$ preserves limits, we have a pullback in $\mathbf{Spc}$
    \[\begin{tikzcd}[every label/.append style={outer sep=0.2cm}]
	{\mathpzc{Map}(x, \myuline{\mathpzc{RInv}}(a,b))} & {\mathpzc{Map}_{\cCi/x}(x \times b, x \times b)} \\
	{\mathpzc{Map}_{\cCi/x}(x \times a, x \times b) \times \mathpzc{Map}_{\cCi/x}(x \times b, x \times a)} & {\mathpzc{Map}_{\cCi/x}(x \times b, x \times b) \times \mathpzc{Map}_{\cCi/x}(x \times b, x \times b)} \\
	& {} \\
	& {}
	\arrow[from=1-1, to=1-2]
	\arrow[from=1-1, to=2-1]
	\arrow["\lrcorner"{anchor=center, pos=0.125}, draw=none, from=1-1, to=4-2]
	\arrow["{(id,id)}", from=1-2, to=2-2]
	\arrow["{(C_{b,a,b,\infty}, \mathpzc{const}_{x \times b})}"', from=2-1, to=2-2]
    \end{tikzcd}\]

    \vspace{-2cm}
    
    \noindent and we will denote this pullback as $\mathpzc{RInv}_{\cCi/x}(x \times a, x \times b)$. Observe that this space can be computed by taking the homotopy pullback in $\mathbf{sSet}$, assuming for example the Kan-Quillen model structure on simplicial sets. 
    
    To give an explicit construction of $\mathpzc{RInv}_{\cCi/x}(x \times a, x \times b)$, one can factor the diagonal of $\mathpzc{Map}_{\cCi/x}(x \times b, x \times b)$ through the path object $ \mathpzc{Map}_{\cCi/x}(x \times b, x \times b)^{\Delta^1}$ into the obvious trivial cofibration followed by the obvious Kan fibration. Then,  $\mathpzc{RInv}_{\cCi/x}(x \times a, x \times b)$ can be obtained by taking the strict pullback of that fibration along the map $(C_{b,a,b,\infty}, \mathpzc{const}_{x \times b})$. Thus, a point in this space consists of maps $f \colon x \times a \to x \times b$ and $g \colon x \times b \to x \times a$ over $x$ together with a homotopy $\Delta^1\to \mathpzc{Map}_{\cCi/x}(x \times b, x \times b)$ witnessing that $[C_{b,a,b,\infty}(f,g)] = [f][g] = [id_{x \times b}]$. In particular, the projection $\mathpzc{RInv}_{\cCi/x}(x \times a, x \times b) \to \mathpzc{Map}_{\cCi/x}(x \times a, x \times b)$ is then the composite of two Kan fibrations and therefore a Kan fibration.

    Dually, we define $\mathpzc{LInv}_{\cCi/x}(x \times a, x \times b)$ to be the space equivalent to the pullback $\mathpzc{Map}(x, \myuline{\mathpzc{LInv}}(a,b))$. Finally, we have a pullback
\[\begin{tikzcd}
	{\mathpzc{Map}(x, \myuline{\mathpzc{BiInv}}(a,b))} & {\mathpzc{LInv}_{\cCi/x}(x \times a, x \times b)} \\
	{\mathpzc{RInv}_{\cCi/x}(x \times a, x \times b)} & {\mathpzc{Map}_{\cCi/x}(x \times a, x \times b)} \\
	& {}
	\arrow[from=1-1, to=1-2]
	\arrow[from=1-1, to=2-1]
	\arrow["\lrcorner"{anchor=center, pos=0.125}, draw=none, from=1-1, to=3-2]
	\arrow[from=1-2, to=2-2]
	\arrow[from=2-1, to=2-2]
\end{tikzcd}\]

\vspace{-1cm}
    
    \noindent and we define $\mathpzc{BiInv}_{\cCi/x}(x \times a, x \times b)$ to be this space. This space can also be computed as a homotopy pullback. But since the maps involved are themselves Kan fibrations, we can just take the strict pullback. Thus, by construction, an object in $\mathpzc{BiInv}_{\cCi/x}(x \times a, x \times b)$ consists of a triple $(f,g,h)$ where $f \colon  x \times a \to x \times b$ and $g,h \colon   x \times b \to x \times a$ are maps over $x$ such that $[C_{a,b,a,\infty}(f,g)] = [f][g] = [id] $ and similarly $[C_{b,a,b,\infty}(h,f)] = [h][f] = [id]$. Hence, $f$ is an equivalence.

    Now, $\mathpzc{BiInv}_{\cCi/-}(- \times a, - \times b)$ is a functor $\cCi^{op} \to \mathbf{Spc}$. By the straightening-unstraightening equivalence this functor corresponds to a right fibration $BiInv(a,b) \to \cCi$. By construction as a limit of representable right fibrations, this right fibration is in fact representable, and it is represented by the object $\myuline{\mathpzc{BiInv}}(a,b)$.
    
    Our next goal is to show that this right fibration is equivalent to the right fibration of equivalences $Eq(a,b) \to \cC$ and hence to derive that the latter is representable as well.

    \begin{lemma}[Cf. Lem. \ref{lemma: precomposition with homotopy equivalence is homotopy equivalence}]
    \label{lemma: precomposition map in  infinity cat with eqivalence is equivalence}
            Let $f \colon x \times a \to x \times b$ be an equivalence over $x$ and $\overline{f}\colon x \to \myuline{\mathpzc{Map}}(a,b)$ be its adjoint transpose and $c$ be any object in $\cCi$. Then, the post-composition map
    $$x \times \myuline{\mathpzc{Map}}(c,a) \xrightarrow{(\pi_x, c_{c,a,b,\infty}(\overline{f} \times id_{\myuline{\mathpzc{Map}}(c,a)}))} x \times \myuline{\mathpzc{Map}}(c,b)$$
     denoted as $f_*$ is an equivalence.
    \end{lemma}

    \begin{proof}
        We may assume that $f$ has a two-sided inverse $g \colon  x \times b \to x \times a$ over $x$ which has as transpose a map $\overline{g} \colon  x \to \myuline{\mathpzc{Map}}(b,a)$. This allows us to define a precomposition map $g_*$ as the composite
    $$x \times \myuline{\mathpzc{Map}}(c,b) \xrightarrow{(\pi_x,  c_{c,b,a,\infty} ( \overline{g} \times id_{\myuline{\mathpzc{Map}}(c,b)}))} x \times \myuline{\mathpzc{Map}}(c,a).$$
    To show that any composition $g_*f_*$ is equivalent to the identity $id_{x \times \myuline{\mathpzc{Map}}(c,a)} = (\pi_x, \pi_{ \myuline{\mathpzc{Map}}(c,a)})$, it suffices to show that the second component of the composite $g_*f_*$
    $$x \times \myuline{\mathpzc{Map}}(c,a) \xrightarrow{(\pi_x,c_{c,b,a,\infty}(\overline{g}\pi_x, c_{c,a,b,\infty}  (\overline{f} \times id_{\myuline{\mathpzc{Map}}(c,a)} )))} x \times \myuline{\mathpzc{Map}}(c,a)$$
    is equivalent to the projection $ \pi_{ \myuline{\mathpzc{Map}}(c,a)} \colon  x \times \myuline{\mathpzc{Map}}(c,a) \to \myuline{\mathpzc{Map}}(c,a)$. For this it suffices to show that the transpose of this map is equivalent to the transpose of $\pi_{ \myuline{\mathpzc{Map}}(c,a)}$ which follows straighforwardly by a computation in the homotopy category $h\cCi$ analogous to the one in the proof of Lemma \ref{lemma: precomposition with homotopy equivalence is homotopy equivalence} using the definition of the composition maps $c_{-,-,-}$ and adjoint transposition. The case of $f_*g_*$ is analogous.
    \end{proof}

    The next proposition can be seen as the $\infty$-categorical version of the type theoretic proof of $\mathtt{isProp(isEquiv(f))}$.

    \begin{proposition}[Cf. Prop. \ref{prop: projection object of equivalences is homotopy monomorphism}]
    \label{proposition: projection from equivalence (biinv) to hom is monomorphism  infinity category}
            The projection map $\pi \colon  \myuline{\mathpzc{BiInv}}(a,b) \to \myuline{\mathpzc{Map}}(a,b)$ is a monomorphism.
    \end{proposition}

    \begin{proof}
        By Lemma \ref{lemma: different characterisations of monomorphism in infty cat}, it suffices to show that for any $\overline{f} \colon  x \to \myuline{\mathpzc{Map}}(a,b)$, the space $\mathpzc{Map}_{\cCi/\myuline{\mathpzc{Map}}(a,b)}(\overline{f}, \pi)$ is either empty or contractible. Suppose it is non-empty, then $\overline{f}$ is the transpose of an equivalence $f \colon  x \times a \to x \times b$ over $x$. Define $\myuline{\mathpzc{RInv}}_f$ as the pullback
    \[\begin{tikzcd}[every label/.append style={outer sep=0.2cm}]
	{\myuline{\mathpzc{RInv}}_f} & {\myuline{\mathpzc{RInv}}(a,b)} & {\myuline{\mathpzc{Map}}(b,b)} \\
	{x \times \myuline{\mathpzc{Map}}(b,a)} & {\myuline{\mathpzc{Map}}(a,b) \times \myuline{\mathpzc{Map}}(b,a)} & {\myuline{\mathpzc{Map}}(b,b) \times \myuline{\mathpzc{Map}}(b,b)} \\
	\\
	& {} & {}
	\arrow[from=1-1, to=1-2]
	\arrow[from=1-1, to=2-1]
	\arrow["\lrcorner"{anchor=center, pos=0.125}, draw=none, from=1-1, to=4-2]
	\arrow[from=1-2, to=1-3]
	\arrow[from=1-2, to=2-2]
	\arrow[from=1-2, to=2-2]
	\arrow["\lrcorner"{anchor=center, pos=0.125}, draw=none, from=1-2, to=4-3]
	\arrow["{(id,id)}", from=1-3, to=2-3]
	\arrow["{\overline{f}\times id_{\myuline{Map}(b,a)}}"', from=2-1, to=2-2]
	\arrow["{(c_{b,a,b,\infty},const_b)}"', from=2-2, to=2-3]
    \end{tikzcd}\]

    \vspace{-1cm}

    \noindent Pulling back this diagram along $x \to 1$ we obtain a diagram of pullbacks

\[\begin{tikzcd}[scale cd=0.9, every label/.append style={outer sep=0.2cm}]
	{\myuline{\mathpzc{RInv}}_f} & {x \times \bullet} & {x \times \myuline{\mathpzc{Map}}(b,b) } \\
	{(x \times \myuline{\mathpzc{Map}}(b,a)) \times_x (x \times 1)} & {(x \times \myuline{\mathpzc{Map}}(b,a)) \times_x (x \times \myuline{\mathpzc{Map}}(b,b))} & {(x \times \myuline{\mathpzc{Map}}(b,b)) \times_x (x \times \myuline{\mathpzc{Map}}(b,b))} \\
	{x \times 1} & {x \times\myuline{\mathpzc{Map}}(b,b)} \\
	\\
	& {} & {}
	\arrow[from=1-1, to=1-2]
	\arrow[from=1-1, to=2-1]
	\arrow["\lrcorner"{anchor=center, pos=0.125}, draw=none, from=1-1, to=5-2]
	\arrow[from=1-2, to=1-3]
	\arrow[from=1-2, to=2-2]
	\arrow["{(p_1,p_2)}", from=1-2, to=2-2]
	\arrow["\lrcorner"{anchor=center, pos=0.125}, draw=none, from=1-2, to=5-3]
	\arrow[from=1-3, to=2-3]
	\arrow["{id \times_x (id_x  \times\overline{id_b})}"', from=2-1, to=2-2]
	\arrow["{\pi_2}", from=2-1, to=3-1]
	\arrow["\lrcorner"{anchor=center, pos=0.125}, draw=none, from=2-1, to=5-2]
	\arrow["{f_* \times_x id}"', from=2-2, to=2-3]
	\arrow["{\pi_2}", from=2-2, to=3-2]
	\arrow["{id_x  \times\overline{id_b}}"', from=3-1, to=3-2]
    \end{tikzcd}\]

        \vspace{-1.5cm}

        \noindent in $\cCi/x$. Since $f$ is an equivalence so is the post-composition map $f_*$ by Lemma \ref{lemma: precomposition map in  infinity cat with eqivalence is equivalence}. Thus, $p_2$ is an equivalence. Therefore, the map
        $$\myuline{\mathpzc{RInv}}_f \to x \times 1 \simeq x$$
        is the pullback of an equivalence, hence an equivalence. So $\myuline{\mathpzc{RInv}}_f \simeq x$. Dually, one shows that $\myuline{\mathpzc{LInv}}_f \simeq x$. It follows that $\myuline{\mathpzc{BiInv}}_f := \myuline{\mathpzc{RInv}}_f \times_x \myuline{\mathpzc{LInv}}_f \simeq x$. In particular, we have a pullback diagram

        \[\begin{tikzcd}
            {\myuline{\mathpzc{BiInv}}_f} & {\myuline{\mathpzc{BiInv}}(a,b)} \\
            x & {\myuline{\mathpzc{Map}}(a,b)} \\
            & {}
            \arrow[from=1-1, to=1-2]
            \arrow["\simeq"', from=1-1, to=2-1]
            \arrow["\lrcorner"{anchor=center, pos=0.125}, draw=none, from=1-1, to=3-2]
            \arrow["\pi", from=1-2, to=2-2]
            \arrow["{\overline{f}}"', from=2-1, to=2-2]
        \end{tikzcd}\]

        \vspace{-1cm}

        Since the pullback functor $\overline{f}^*$ alsways has a left adjoint induced by composition with $\overline{f}$, there is an equivalence of spaces
        $$\mathpzc{Map}_{\cCi/\myuline{\mathpzc{Map}}(a,b)}(\overline{f}, \pi) \simeq \mathpzc{Map}_{\cCi/x}(id_x, \overline{f}^*\pi) $$
        between the space of morphisms from $\overline{f}$ to $\pi$ in the slice over $x$ and sections of pullback $\overline{f}^*\pi$. But $\overline{f}^*\pi$ is the equivalence $\myuline{\mathpzc{BiInv}}_f \to x$ so this space of sections is contractible by Lemma \ref{lemma: different characterisations of monomorphism in infty cat}, hence $\mathpzc{Map}_{\cCi/\myuline{\mathpzc{Map}}(a,b)}(\overline{f}, \pi)$ is contractible.
        \end{proof}

\begin{proposition}
    \label{prop: right fibration of equivalences is always representable}
    For any finitely complete and cartesian closed $\infty$-category $\cCi$,  $BiInv(a,b) \to \cCi$ and $Eq(a,b) \to \cCi$ are equivalent as right fibrations. Hence, the object of equivalences $\myuline{\mathpzc{BiInv}}(a,b)$ induces a representation of the right fibration $Eq(a,b) \to \cCi$. 
\end{proposition}

\begin{proof}
    We know that $Eq(a,b) \to \cCi$ is a subfibration of $Map(a,b) \to \cCi$ such that on fibers over objects $x$, the map $\mathpzc{Eq}_{\cC_{\infty}/x}(x \times a,x \times b) \to \mathpzc{Map}_{\cC_{\infty}/x}(x \times a,x \times b)$ is an inclusion, and it is in fact (-1)-truncated. We also know from Construction \ref{construction: object of equivalences infinity cat} that $\myuline{\mathpzc{BiInv}}(a,b)$ represents the right fibration $BiInv(a,b) \to \cCi$.
    
    To show that $BiInv(a,b) \to \cCi$ and $Eq(a,b) \to \cCi$ are equivalent, it suffices by the straightening-unstraightening correspondence to show that the corresponding functors $\cCi^{op} \to \mathbf{Spc}$ are naturally equivalent. For this we need to find first a natural transformation between them and then show that it is an equivalence. By the pointwise criterion for natural equivalences \cite[Cor. 2.2.2]{land_introduction_2021}, it suffices to show that the natural transformation induces a pointwise equivalence 
    $$\mathpzc{BiInv}_{\cC_{\infty}/x}(x \times a,x \times b) \simeq \mathpzc{Eq}_{\cC_{\infty}/x}(x \times a,x \times b)$$
    of spaces. For readability, we will drop from now on the subscripts. Observe first, that the projection $\pi \colon  \myuline{\mathpzc{BiInv}(a,b)} \to \myuline{\mathpzc{Map}(a,b)}$ induces by the $\infty$-Yoneda lemma indeed a natural transformation $\mathpzc{BiInv}(- \times a,- \times b) \to \mathpzc{Map}(- \times a,- \times b)$ which gives us on points a natural map $\mathpzc{BiInv}(x \times a,x \times b) \to \mathpzc{Map}(x \times a,x \times b)$. This map is in fact (-1)-truncated, since $\pi$ is a monomorphism by Proposition \ref{proposition: projection from equivalence (biinv) to hom is monomorphism  infinity category}. Moreover, any morphism in the image of $\pi$ is an equivalence. Thus $\pi$ induces a (-1)-truncated map
    $$\mathpzc{BiInv}(x \times a,x \times b) \to \mathpzc{Eq}(x \times a,x \times b).$$
    To show this is an equivalence, it suffices by (-1)-truncatedness to show that it induces a surjection on components
    $$\pi_0\mathpzc{BiInv}(x \times a,x \times b) \to \pi_0 \mathpzc{Eq}(x \times a,x \times b).$$
    This is however clear since the map is itself surjective on objects. For, if $f \colon  x \times a \to x \times b$ is any equivalence, then choosing left and right inverses $g$ and $h$, we get that $(f,g,h) \in \mathpzc{BiInv}(x \times a,x \times b)$ is in the preimage of $f$.
\end{proof}

\subsection{Univalent Morphisms}

Now, we will define univalence in $\infty$-categories. Given a morphism $p \colon e \to b$ in a locally cartesian closed $\infty$-category $\cC$, we consider the right fibration 
$$Eq_{b\times b}(\pi_1^*p, \pi_2^*p) \to \cCi/b \times b$$
where $\pi_i \colon b \times b \to b$ are the product projections. Its fiber over $(f,g) \colon x \to b \times b$ is the space $\mathpzc{Eq}_{\cCi/x}(f^*p, g^*p) \subseteq \mathpzc{Map}_{\cCi/x}(f^*p, g^*p)$ of equivalences $f^*p \simeq g^*p$ over $x$.

Consider the diagonal $(id,id) \colon b \to b \times b$. The map $\cCi/(b\times b)/(id,id)\cong \cCi/(id,id) \to \cCi/b$ is in fact a trivial Kan fibration by the pushout-join/pullback-slice adjunction and the fact that the inclusion $\{1\} \to \Delta^1$ is right anodyne. Choosing a section, we obtain an equivalence $s \colon  \cCi/b \xrightarrow{\simeq } (\cCi/(b\times b))/(id,id)$. 

The right fibration $(\cCi/(b\times b))/(id,id) \to \cCi/b\times b$ is always representable since $(id,id)$ is a terminal object in its domain. A point in the fiber of $Eq_{b\times b}(\pi_1^*p, \pi_2^*p) \to \cCi/b \times b$ over $(id,id)$ is an equivalence $e \to e$ over $b$. Thus, by the $\infty$-Yoneda lemma, the identity $e \to e$ corresponds to a map $\gamma \colon  (\cCi/(b\times b))/(id,id) \to Eq_{b\times b}(\pi_1^*p, \pi_2^*p)$. Composing this map with the section $s$ we obtain a map $\delta_e$ as well as a map $\cCi/b \to \cCi/b \times b$ as in 
\[\begin{tikzcd}
	{\cCi/b} & {(\cCi/b\times b)/(id,id)} \\
	\\
	{\cCi/b\times b} & {Eq_{b\times b}(\pi_1^*p,\pi_2^*p)}
	\arrow["s"', shift right, from=1-1, to=1-2]
	\arrow[dashed, from=1-1, to=3-1]
	\arrow["{\delta_e}", dashed, from=1-1, to=3-2]
	\arrow["\simeq"', shift right, two heads, from=1-2, to=1-1]
	\arrow["\gamma", from=1-2, to=3-2]
	\arrow[from=3-2, to=3-1]
\end{tikzcd}\]
On objects, $\delta_e$ sends a morphism $f \colon  x \to b$ to the identity $id_{f^*p} \in \mathpzc{Eq}_{\cCi/x}(f^*p, f^*p)$ in the fiber over $(f,f) \colon x \to b \times b$. By 2-out-of-3, $\delta_e$ is a categorical equivalence if and only if $\gamma$ is. This makes the following definition of univalence well-defined:

\begin{definition}[{Cf. \autocite[Prop. 3.8]{gepner_univalence_2017}}]
    \label{def: univalence in infinity category1}
    A morphism $p \colon e \to b$ in an $\infty$-category $\cCi$ is \textit{univalent} if the map $\delta_e$ in the commutative digram
    \[\begin{tikzcd}
    	{\cCi/b} && {Eq_{b\times b}(\pi_1^*p,\pi_2^*p)} \\
    	& {\cCi/b\times b}
    	\arrow["{\delta_{e}}", from=1-1, to=1-3]
    	\arrow[from=1-1, to=2-2]
    	\arrow[from=1-3, to=2-2]
    \end{tikzcd}\]
    is a categorical equivalence.
\end{definition}

This definition is a variation of the definition by Gepner and Kock \cite[Def. 3.2]{gepner_univalence_2017}, but it does not require the fibration of equivalences to be representable nor $\cCi$ to be presentable, which are two further assumptions made in \cite{gepner_univalence_2017}.

However, by Proposition \ref{prop: right fibration of equivalences is always representable}, if $\cCi$ is finitely complete and locally cartesian closed, then, the right fibration $Eq_{b\times b}(\pi_1^*p,\pi_2^*p) \to \cCi/b\times b$ is indeed representable and represented by some object $\myuline{\mathpzc{Eq_b}}(e) \to b \times b$ in $\cCi/b \times b$. Then, $p \colon  e \to b$ is univalent if and only if the map $\gamma$ from above is an equivalence between representable right fibrations. This in turn is by the $\infty$-Yoneda lemma the case if and only if the corresponding map between representing objects 
\[\begin{tikzcd}
	b && {\myuline{\mathpzc{Eq_b}}(e)} \\
	& {b \times b}
	\arrow[from=1-1, to=1-3]
	\arrow["{(id,id)}"', from=1-1, to=2-2]
	\arrow[from=1-3, to=2-2]
\end{tikzcd}\]
is an equivalence. This last statement is the definition of univalence used in \cite{gepner_univalence_2017}.

\begin{definition}[Cf. Def. \ref{def: local class tribe}]
\label{def: local class infty cat}
    Let $\cCi$ be an $\infty$-category. A \textit{local class} is a class of morphisms $S$ in $\cCi$ that is closed under pullback. A local class in a finitely complete and locally cartesian closed $\infty$-category $\cCi$ is \textit{closed}, if $S$ is closed under composition and if for all morphisms $f,g \in S$, $\Pi_{f}(g) \in S$. A local class is $p$-\textit{bounded} if all of its morphisms are obtained as pullbacks of some morphism $p$.
\end{definition}

\begin{definition}[Cf. Def. \ref{def: local class of pullbacks tribe}]
    Given a morphism $p \colon e \to b$ in an $\infty$-category, denote as $S_p$ the class of all morphisms that are pullbacks of $p$. This is a $p$-bounded local class.
\end{definition}

\begin{definition}[Cf. Def. \ref{def: enough univalent fibs tribe}]
\label{def: enough univalent fibs infty cat}
    
    A finitely complete and locally cartesian closed $\infty$-category has \textit{enough univalent morphisms} if for every morphism $f \colon a \to x$, there is a univalent morphism $p$ such that $f \in S_p$ and $S_p$ is closed.
\end{definition}

Next, we collect some equivalent characterisations of univalence for which we need some further definitions.

\begin{definition}[{\cite[§3.3]{gepner_univalence_2017}}]

Let $S$ be a local class of  morphisms in an $\infty$-category $\cCi$. Let $\cO_{\cCi}^{(S)}$ denote the full subcategory of the arrow category $\mathpzc{Fun}(\Delta^1,\cCi)$ containing only the maps in a class $S$ and consisting of those morphisms which are pullback squares. For any object $c$ we denote as of $(\cCi/c)^{\mathsf{gpd}}$ the maximal subgroupoid of $\cCi/c$ containing all equivalences. We denote as $(\cCi/c)^{\mathsf{gpd}, S}$ the full subgroupoid of $(\cCi/c)^{\mathsf{gpd}}$ spanned by the morphisms in $S$ with codomain $c$.
\end{definition}

The codomain fibration $\cO_{\cCi}^{(S)} \to \cCi$ induced by $\{1\} \subseteq \Delta^1$ is a right fibration by \cite[\;6.1.3.4]{lurie_higher_2009} and its fiber over an object $c$ is the $\infty$-groupoid $(\cCi/c)^{\mathsf{gpd}, S}$. For any object $c$, the fiber of the representable right fibration $\cCi/c \to \cCi$ over some object $x$ is equivalent to the space $\mathpzc{Map}(x,c)$. Given a morphism $p \colon e \to b$ and a local class $S$ with $p \in S$, by the $
\infty$-Yoneda lemma there is an induced morphism of right fibrations 
$$\kappa_p  \colon \cCi/b \to \cO_{\cCi}^{(S)} $$
over $\cCi$ which acts on objects by sending $f$ to the pullback $f^*p$. The fiber of this morphism over an object $x \in \cCi$ is the map $\mathpzc{Map}(x,b) \to (\cCi/x)^{\mathsf{gpd}, S}$ acting via pullback.

\begin{proposition}[Cf. {\autocite[Prop. 3.8]{gepner_univalence_2017}}]

\label{prop: different characterisations of univalence in infty-cat}

Let $p \colon e \to b$ be a morphism in an $\infty$-category $\cCi$ and let $S_p$ be the local class of maps obtained as pullbacks of $p$. The following are equivalent:

\begin{enumerate}
    \item $p$ is univalent.
    \item $\kappa_p  \colon \cCi/b \to \cO_{\cCi}^{(S_p)} $ is an equivalence of right fibrations, i.e. on fibers $\mathpzc{Map}(x,b) \to (\cCi/x)^{\mathsf{gpd}, S_p}$ is an equivalence for all $x \in \cCi$ inducing a representation of  $\cO_{\cCi}^{(S_p)} \to \cCi$.
    \item $p$ is terminal in $\cO_{\cCi}^{(S_p)}.$\qed
\end{enumerate}
\end{proposition}

The third condition in Proposition \ref{prop: different characterisations of univalence in infty-cat} encodes that every morphism in $S_p$ is a homotopy-unique pullback of $p$.

\begin{proposition}[{\autocite[Cor. 3.10]{gepner_univalence_2017}, \autocite[Prop. 2.5]{rasekh_univalence_2021}}, Cf. \ref{prop: univalence is preserved and reflected by pullback along homotopy monomorphisms}]

    \label{prop: pullback of univalent along monomorphism in infty cat}
    Let $p \colon e \to b$ a univalent morphism and let $f \colon x \to b$ be a morphism such that
    \[\begin{tikzcd}
    	x & e \\
    	y & b
    	\arrow["h", from=1-1, to=1-2]
    	\arrow["q"', from=1-1, to=2-1]
    	\arrow["\lrcorner"{anchor=center, pos=0.125}, draw=none, from=1-1, to=2-2]
    	\arrow["p", from=1-2, to=2-2]
    	\arrow["f"', from=2-1, to=2-2]
    \end{tikzcd}\]
    is a pullback. Then, $q$ is univalent if and only if $f$ is a monomorphism.
\qed \end{proposition}

\subsection{The Localisation of Univalent Tribes has Univalent Morphisms}

We will now turn to the proof that the $\infty$-localisation sends univalent fibrations to univalent morphisms.

\begin{proposition}

    \label{prop: object of equivalences in loclaisation}
    Let $\cC$ be $\pi$-tribe and $A$ and $B$ objects in $\cC$. Then the localisation of the object of equivalences $\gamma(\myuline{\mathsf{Eq}}(A,B))$ is an object of equivalences in $\cC_{\infty}$, i.e. $\gamma(\myuline{\mathsf{Eq}}(A,B)) \simeq \myuline{\mathpzc{BiInv}}(A,B)$, such that for any object $x \in \cC_{\infty}$ we have a representation

    $$\mathpzc{Map}_{\cC_{\infty}}(x, \gamma(\myuline{\mathsf{Eq}}(A,B))) \simeq \mathpzc{Eq}_{\cC_{\infty}}(x \times A, x \times B)$$
    of the right fibration $Eq(A,B) \to \cC_{\infty}$.
\end{proposition}

\begin{proof}
    We just need to show from what we established so far that the localisation preserves the required properties such that when applied to the object of equivalences in a tribe from Construction \ref{constr: object of equiv}, we get an object equivalent to the object of equivalences in an $\infty$-category from Construction \ref{construction: object of equivalences infinity cat}. 
    
    By \autocite[Prop. 7.5.6]{cisinski_higher_2019}, the localisation preserves pullbacks along fibrations. By Proposition \ref{prop:adjoint functors derive adjoint functors unit/counit}, $\gamma(\uHom_{\cC}(A,B)) \simeq \myuline{\mathpzc{Map}}_{\cC_{\infty}}(A,B)$. If we have a path object factorisation $B \to P_B \to B \times B$ of the diagonal into a homotopy equivalence followed by a fibration in the tribe $\cC$, the morphism $B \to P_B$ will become an equivalence in $\cC_{\infty}$ which induces a canonical equivalence $ \myuline{\mathpzc{Map}}_{\cC_{\infty}}(A,B) \simeq \myuline{\mathpzc{Map}}_{\cC_{\infty}}(A,P_B) $ of internal mapping space objects in $\cC_{\infty}$.

    By Proposition \ref{prop:adjoint functors derive adjoint functors unit/counit} the localisation commutes up to homotopy with adjoint transposition, thus $\gamma(c_{A,B,A}) \simeq c_{A,B, A,\infty}$ and $\gamma(c_{B,A,B}) \simeq c_{B,A, B,\infty}$. Moreover, the localisation of the constant identity morphism $\gamma(const_{A})$ is homotopic to the constant identity morphism in the localisation. Thus, by the definitions, we get equivalences $\gamma(\myuline{\mathsf{LInv}}(A,B)) \simeq \myuline{\mathpzc{LInv}}(A,B)$ and $\gamma(\myuline{\mathsf{RInv}}(A,B)) \simeq \myuline{\mathpzc{RInv}}(A,B)$. Hence, we get $\gamma(\myuline{\mathsf{Eq}}(A,B)) \simeq \myuline{\mathpzc{BiInv}}(A,B)$ as desired.
\end{proof}

\begin{corollary}
\label{corollary: localisation preserves object of equivalences local version}

    For any $(f,g) \colon x \to B \times B$ in $\cC_{\infty}$ and any fibration $p \colon E \to B$ in $\cC$, we have a representation
    $$\mathpzc{Map}_{\cC_{\infty}/B \times B}(x, \gamma(\myuline{\mathsf{Eq}}_B(E))) \simeq \mathpzc{Eq}_{\cC_{\infty}/x}(f^*\gamma(p),g^*\gamma(p)).$$
\end{corollary}
\begin{proof}
    Use the fact that $(\cC/\!\!/B \times B)_{\infty} \simeq \cC_{\infty}/B \times B$.
\end{proof}

\begin{theorem}
\label{theorem: univalence in a tribe implies univalence in infinity-cat}

    Let $\cC$ be $\pi$-tribe and let $p \colon E \to B$ be a univalent fibration. Let $\gamma  \colon \cC \to \cC_{\infty}$ be the localisation. Then $\gamma(f)$ is a univalent morphism in $ \cC_{\infty}$.
\end{theorem}

\begin{proof}
    By univalence of $p$, we have a factorisation in $\cC$
    
\[\begin{tikzcd}
	B && {\myuline{\mathsf{Eq}}_B(E)} \\
	& {B \times B}
	\arrow["{\delta_E \sim}", from=1-1, to=1-3]
	\arrow["{(id,id)}"', from=1-1, to=2-2]
	\arrow[two heads, from=1-3, to=2-2]
\end{tikzcd}\]

    The localisation $\gamma$ sends $\delta_E$ to an equivalence in $\cC_{\infty}$ and preserves the diagonal. By Corollary \ref{corollary: localisation preserves object of equivalences local version}, it sends $\myuline{\mathsf{Eq}}_B(E)$  to an object of equivalences of $\gamma(p)$ in $\cC_{\infty}$.
    
    The morphism $\delta_E$ was defined to be the adjoint transpose of the identity on $E$ over $B$ via $p$ and by Proposition \ref{prop:adjoint functors derive adjoint functors unit/counit} the localiation commutes up to homotopy with adjoint transposition such that $\gamma(\delta_E)$ is equivalent transpose of the identity on $E$ over $B$. Thus $ \gamma(\delta_E) \simeq \delta_{\gamma(E)}$ and $\delta_{\gamma(E)}$ is an equivalence. By the $\infty$-Yoneda lemma this implies that the map that is represented by $\delta_{\gamma(E)}$ is an equivalence and hence that
    \[\begin{tikzcd}
    	{\cC_{\infty}/B} && {\mathpzc{Eq}_{B\times B}(\pi_1^*\gamma(p),\pi_2^*\gamma(p))} \\
    	& {\cC_{\infty}/B\times B}
    	\arrow[from=1-1, to=1-3]
    	\arrow[from=1-1, to=2-2]
    	\arrow[from=1-3, to=2-2]
    \end{tikzcd}\]
is an equivalence and so $\gamma(p)$ is univalent.
\end{proof}

\begin{corollary}
\label{thm: tribe enough universal families then localisation enough univalent morphisms}
    Let $\cC$ be a $\pi$-tribe that has enough univalent fibrations. Then, $\cC_{\infty}$ has enough univalent morphisms.
\end{corollary}

\begin{proof}
     $\cC_{\infty}$ has finite limits by \autocite[Prop. 7.5.6]{cisinski_higher_2019} and is locally cartesian closed by \autocite[Prop. 7.6.16]{cisinski_higher_2019}.

    By \autocite[Cor. 7.2.18 and 7.2.10.4]{cisinski_higher_2019}, using the right calculus of fractions on trivial fibrations any map in $\cC_{\infty}$ is (homotopic to one) of the form $\gamma(g)\gamma(s)^{-1}$ where $s$ is a trivial fibration in $\cC$ and $g$ is any map. Thus, by factoring $g$ into an anodyne map followed by a fibration, a morphism $f$ in $\cC_{\infty}$ is equivalent to the image of the localisation $\gamma$ of some fibration $f'$ in the tribe $\cC$. But since $\cC$ has enough univalent fibrations there is a univalent fibration $p$ in $\cC$ that classifies $f'$ and a closed local class $S_p$. By Theorem \ref{theorem: univalence in a tribe implies univalence in infinity-cat}, $\gamma(p)$ is univalent. Since the localisation preserves pullbacks along fibrations by \autocite[Prop. 7.5.6]{cisinski_higher_2019}, $\gamma(p)$ classifies $\gamma(f')$, hence it classifies the equivalent morphism $f$ which has the same codomain. Moreover, $\gamma(p)$ defines a local class containing $f$.

    It remains to show that $S_{\gamma(p)}$ is closed. Let $g,f \in S_{\gamma(p)}$ be composable. We need to show that any chosen composition $gf \in S_{\gamma(p)}$. By the proof of \autocite[Prop. 7.5.6]{cisinski_higher_2019}, we may assume that $f$ and $g$ are equivalent to morphisms $\gamma(f')$ and $\gamma(g')$ which are images of fibrations that are strict pullbacks of $p$ in $\cC$. In particular, we may assume that $f = \gamma(f')\gamma(s)^{-1}$ and $g = \gamma(f')\gamma(t)^{-1}$ for $s$ and $t$ being homotopy equivalences. Then taking the pullback in $\cC$ of $f'$ along $t$ we obtain a fibration $f''$ that is by pullback pasting a pullback of $p$. Since $S_p$ is closed under composition, $g'f'' \in S_p$ hence $g'f''$ is a pullback of $p$. Then, $\gamma$ sends this pullback to a pullback in of $\gamma(p)$ in $\cC_{\infty}$ and moreover $\gamma(g'f'')$ is equivalent to the composite $gf$, hence $gf \in S_{\gamma(p)}$.

    For the closure under the right adjoint to pullback, we only need to show by factorisation and the proof of \autocite[Prop. 7.6.16]{cisinski_higher_2019}, that for any fibration $g$ and any $f \in S_{\gamma(p)}$, $\Pi_{\gamma(g)}(\gamma(f)) \in S_{\gamma(p)}$. This, however, follows then directly using the diagrams in Proposition \ref{prop:adjoint functors derive adjoint functors unit/counit} together with the fact that $(\cC/\!\!/X)_{\infty} \simeq \cC_{\infty}/X$ which imply that $\Pi_{\gamma(g)}(\gamma(f)) \sim \gamma(\Pi_g(f))$ and using the fact that $S_{p}$ is closed.
\end{proof}

\subsection{Subobject Classifiers in $\infty$-Categories}

We saw that in a $\pi$-tribe with univalent fibrations we can always construct homotopy subobject classifiers. We will now show that the analogous result holds for locally cartesian closed $\infty$-categories with finite limits and univalent morphisms. That is, there is a subobject classifier for every any univalent morphism $p$ classifying the monomorphisms in the local class $S_p$.

This means in particular that we can derive the existence of subobject classifiers from the existence of these univalent morphisms. We will moreover show that the localisation sends a homotopy subobject classifier in a $\pi$-tribe to a subobject classifier in the induced $\infty$-category. This allows us to work with subobject classifiers both at the level of tribes and at the level of $\infty$-categories.

\begin{definition}

    Let $\cCi$ be an $\infty$-category with finite limits. Let $S$ be a local class. There is a functor 
    $$\mathsf{Sub}_S(-)  \colon \cCi^{op} \to \Set$$ 
    that takes each object $x$ to the set $\mathsf{Sub}_S(x)$ of equivalence classes of those monomorphisms with target $x$ that are in $S$, with the equivalence relation of being equivalent over $x$ in $\cCi/x$. Notice that the pullback of a monomorphism in an $\infty$-category is also a monomorphism and so this gives us an actual functor.
\end{definition}

The fact that this functor is valued in $\Set$ rather than $\mathbf{Spc}$ comes from the fact the space of subobjects of an object $x$ in an $\infty$-category is by Lemma \ref{lemma: different characterisations of monomorphism in infty cat} homotopy equivalent to a set with a preorder, so taking equivalence classes does not truncate away any higher information.

\begin{definition}
\label{def: subobject classifier infinity-categories}
    Let $\cCi$ be an $\infty$-category with finite limits.  A monomorphism  $\mathsf{el}(\omega_S) \to \omega_S$ is a \textit{subobject classifier} for a local class $S$ if it represents the functor $\mathsf{Sub}_S(-)$ such that
    $$\mathpzc{Map}_{\cCi_{\infty}}(x, \omega_S) \cong \mathsf{Sub}_S(x)$$
    In particular, $\mathpzc{Map}_{\cCi_{\infty}}(x, \omega_p)$ is a 0-truncated space.
\end{definition}

\begin{lemma}[Cf. Lem. \ref{lemma: characterisation homotopy monomorphisms classified by univalent fibration}]
    \label{lemma: infty cat characterisation of monomorphisms classified by univalent morphism}
    A morphism $f \colon x \to y$ classified by some univalent morphism $p \colon e \to b$ in an $\infty$-category $\cCi$ via a pullback diagram
\[\begin{tikzcd}
	x & e \\
	y & b
	\arrow["{q_f}", from=1-1, to=1-2]
	\arrow["f"', from=1-1, to=2-1]
	\arrow["\lrcorner"{anchor=center, pos=0.125}, draw=none, from=1-1, to=2-2]
	\arrow["p", from=1-2, to=2-2]
	\arrow["{\chi_f}"', from=2-1, to=2-2]
\end{tikzcd}\]
    is a monomorphism if and only if there is a morphism $H$
    \[\begin{tikzcd}
    	{x \times_y x} && e \\
    	& {e\times_b e}
    	\arrow["H", from=1-1, to=1-3]
    	\arrow["{q_f\times_{\chi_f}q_f}"', from=1-1, to=2-2]
    	\arrow["{(id,id)}", from=1-3, to=2-2]
    \end{tikzcd}\]
where $(id,id)$ is the fibered diagonal of $p$ in the slice $\cCi/b$.
    
\end{lemma}

\begin{proof}
    Consider the pullback diagram
    \[\begin{tikzcd}
    	x & e \\
    	{x \times_y x} & {e\times_b e}
    	\arrow["{q_f}", from=1-1, to=1-2]
    	\arrow["{(id_x,id_x)}"', from=1-1, to=2-1]
    	\arrow["\lrcorner"{anchor=center, pos=0.125}, draw=none, from=1-1, to=2-2]
    	\arrow["{(id_e,id_e)}", from=1-2, to=2-2]
    	\arrow["{q_f \times_{\chi_f} q_f}"', from=2-1, to=2-2]
    \end{tikzcd}\]
    Since the pullback functor $(q_f \times_{\chi_f} q_f)^*$ always has a left adjoint given by composition with $q_f \times_{\chi_f} q_f$, there is an equivalence of spaces
    $$\mathpzc{Map}_{\cCi/e\times_b e}(q_f \times_{\chi_f} q_f,(id_e,id_e)) \simeq \mathpzc{Map}_{\cCi/x\times_x x}(id_{x \times_y x},(id_x,id_x))$$
    between the space of morphisms from $q_f \times_{\chi_f} q_f$ to $(id_e,id_e)$ in the slice $\cCi/e\times_b e$ and section to the pullback $(id_x,id_x) = (q_f \times_{\chi_f} q_f)^*(id_e,id_e)$. Thus, such a morphism $H$ exists if and only if $(id_x,id_x) \colon x \to x \times_y x$ has a section which is the case if and only if $(id_x,id_x)$ is an equivalence, since the projections are always retractions. This in turn holds if and only if
    \[\begin{tikzcd}
        x & x \\
        x & y
        \arrow["id", from=1-1, to=1-2]
        \arrow["id"', from=1-1, to=2-1]
        \arrow["f", from=1-2, to=2-2]
        \arrow["f"', from=2-1, to=2-2]
    \end{tikzcd}\]
    is a pullback which is the case if and only if $f$ is a monomorphism.
    \end{proof}

\begin{construction}[Cf. Constr. \ref{constr: homotopy subobject classifier from univalent fibration}]

    Let $p \colon  e \to b$ be a univalent fibration in a locally cartesian closed and finitely complete $\infty$-category $\cCi$. Let $p \times p \colon e \times_b e \to b$ be the product $p \times p$ in $\cCi/b$. Local cartesian closedness induces a functor $\Pi_{p \times p} \colon \cCi/e\times_b e \to \cCi/b$ that is right adjoint to pullback $(p \times p)^*\colon \cCi/b \to \cCi/e\times_b e$.
    
    Let $(id,id) \colon e \to e \times_b e$ be the diagonal. Define $pr_p \colon \omega_p \to b$ as the morphism
    $$\Pi_{p \times p}((id,id)) \colon \Pi_{p \times p}(e \times_b e) \to b$$
    and let $\top \colon \mathsf{el}(\omega_p) \to \omega_p$ be defined by the pullback
    \[\begin{tikzcd}
    	{\mathsf{el}(\omega_p)} & e \\
    	{\omega_p} & b
    	\arrow[from=1-1, to=1-2]
    	\arrow["\top"', from=1-1, to=2-1]
    	\arrow["\lrcorner"{anchor=center, pos=0.125}, draw=none, from=1-1, to=2-2]
    	\arrow["p", from=1-2, to=2-2]
    	\arrow["{pr_p}"', from=2-1, to=2-2]
    \end{tikzcd}\]
    The morphism $pr_p \colon \omega_p \to b$ has by adjointness the property that there is a an equivalence
    $$\mathpzc{Map}_{\cCi/b}(\chi_f,pr_p) \simeq \mathpzc{Map}_{\cCi/e\times_b e}(q_f \times_{\chi_f} q_f,(id,id))$$
    for any pullback
    \[\begin{tikzcd}
    	x & e \\
    	y & b
    	\arrow["{q_f}", from=1-1, to=1-2]
    	\arrow["f"', from=1-1, to=2-1]
    	\arrow["\lrcorner"{anchor=center, pos=0.125}, draw=none, from=1-1, to=2-2]
    	\arrow["p", from=1-2, to=2-2]
    	\arrow["{\chi_f}"', from=2-1, to=2-2]
    \end{tikzcd}\]
    
\end{construction}

\begin{lemma}[Cf. Lem. \ref{lemma: top is homotopy monomorphism}]
\label{lemma: top is a monomorphism infty cat}

    $\top \colon \mathsf{el}(\omega_p) \to \omega_p$ is a monomorphism.
\end{lemma}

\begin{proof}

This follows directly using Lemma \ref{lemma: infty cat characterisation of monomorphisms classified by univalent morphism}: The identity fits into the following diagram on the left, inducing the desired morphism $K$ by adjoint transposition on the right which witnesses that the map $\top$ is a  monomorphism.
\[\begin{tikzcd}
	{\omega_p} && {\omega_p} && {\mathsf{el}(\omega_p)\times_{\omega_p} \mathsf{el}(\omega_p)} && e \\
	& b &&&& {e \times_b e}
	\arrow["id", from=1-1, to=1-3]
	\arrow["{pr_p}"', from=1-1, to=2-2]
	\arrow["{pr_p}", from=1-3, to=2-2]
	\arrow["K", from=1-5, to=1-7]
	\arrow[from=1-5, to=2-6]
	\arrow["{(id,id)}", from=1-7, to=2-6]
\end{tikzcd}\]
\end{proof}

Our next lemma is the $\infty$-categorical version of the type theoretic proof of $\mathtt{isProp(isProp(f))}$.

\begin{proposition}[Cf. Prop. \ref{prop: homotopy subobject classifier is a proposition}]

\label{lemma: pr_p is monomorphism}
    The morphism $pr_p \colon  \omega_p \to b$ is a monomorphism.
\end{proposition}

\begin{proof}
   By Lemma \ref{lemma: different characterisations of monomorphism in infty cat}, it suffices to show that for any $\chi_f \colon y \to b$
    $$\mathpzc{Map}_{\cCi/B}(\chi_f,pr_p)$$
    is either empty or contractible. Suppose we are given a pullback
        \[\begin{tikzcd}
        	x & e \\
        	y & b
        	\arrow["{q_f}", from=1-1, to=1-2]
        	\arrow["f"', from=1-1, to=2-1]
        	\arrow["\lrcorner"{anchor=center, pos=0.125}, draw=none, from=1-1, to=2-2]
        	\arrow["p", from=1-2, to=2-2]
        	\arrow["{\chi_f}"', from=2-1, to=2-2]
        \end{tikzcd}\]
    By adjointness
    $$\mathpzc{Map}_{\cCi/b}(\chi_f,pr_p) \simeq \mathpzc{Map}_{\cCi/e\times_b e}(q_f \times_{\chi_f} q_f,(id_e,id_e)).$$
    But as in the proof of Lemma \ref{lemma: infty cat characterisation of monomorphisms classified by univalent morphism} we also have
    $$\mathpzc{Map}_{\cCi/e\times_b e}(q_f \times_{\chi_f} q_f,(id_e,id_e))) \simeq \mathpzc{Map}_{\cCi/x\times_x x}(id_{x \times_y x},(id_x,id_x))) $$
    where the latter space is the space of sections of  $(id_x,id_x) \colon x \to x \times_y x$. Since this map always has retractions given by the projections, if $\mathpzc{Map}_{\cCi/x\times_x x}(id_{x \times_y x},(id_x,id_x))$ is non-empty, then $(id_x,id_x)$ is an equivalence, hence by Lemma \ref{lemma: different characterisations of monomorphism in infty cat}, this space of sections is contractible. Thus, $\mathpzc{Map}_{\cCi/b}(\chi_f,pr_p)$ is either empty or contractible.
    \end{proof}
\begin{proposition}
\label{proposition: top is univalent infty cat}
    
    $\top \colon \mathsf{el}(\omega_p) \to \omega_p$ is univalent.
\end{proposition}

\begin{proof}
    By Lemma \ref{lemma: pr_p is monomorphism}, $\top$ is the pullback of a univalent morphism along a monomorphism. Thus $\top$ is univalent by Proposition \ref{prop: pullback of univalent along monomorphism in infty cat}.
\end{proof}

\begin{definition}
    Let $p$ be a univalent morphism in an $\infty$-category $\cCi$. Denote as $S_{p}^{\mathsf{mon}}$ the local subclass of the local class $S_p$ consisting only of monomorphisms.
\end{definition}

\begin{lemma}
\label{lemma: groupoid of monos is 0-truncated}
    Let $\cCi$ be an $\infty$-category and $x$ be an object in $\cCi$. The space $(\cCi/x)^{\mathsf{gpd}, S_{p}^{\mathsf{mon}}}$ is 0-truncated. In particular, $(\cCi/x)^{\mathsf{gpd}, S_{p}^{\mathsf{mon}}} \cong  \mathsf{Sub}_{S_p}(x)$ as sets.
\end{lemma}

\begin{proof}
    Let $f \colon  a \to x$ and $g \colon  b \to x$ be objects in $(\cCi/x)^{\mathsf{gpd}, S_{p}^{\mathsf{mon}}}$. We need to show that $\mathpzc{Map}_{(\cCi/x)^{\mathsf{gpd}}}(f,g)$ is either empty or contractible. Since $g$ is a monomorphism, this follows from Lemma \ref{lemma: different characterisations of monomorphism in infty cat}.
\end{proof}

\begin{lemma}
    
\label{lemma: monomorphisms of univalent morphisms are those of subobject classifier}
    The class $S_{p}^{\mathsf{mon}}$ of monomorphisms obtained as pullbacks of $p$ is precisely the local bounded class $S_{\top}$ of morphisms obtained as pullbacks of $\top$.
\end{lemma}

\begin{proof}
    It is clear that any pullback of $\top$ is a pullback of $p$ by the pullback pasting lemma for $\infty$-categories and it is a monomorphism since $\top$ is a mono by Lemma \ref{lemma: top is a monomorphism infty cat} and monomorphisms are stable under pullback. Conversely, let $f\colon x \to y$ be any monomorphism that is a pullback of $p$ given by
        \[\begin{tikzcd}
        	x & e \\
        	y & b
        	\arrow["{q_f}", from=1-1, to=1-2]
        	\arrow["f"', from=1-1, to=2-1]
        	\arrow["\lrcorner"{anchor=center, pos=0.125}, draw=none, from=1-1, to=2-2]
        	\arrow["p", from=1-2, to=2-2]
        	\arrow["{\chi_f}"', from=2-1, to=2-2]
        \end{tikzcd}\]
    Since $f$ is a monomorphism classified by $p$, by Lemma \ref{lemma: infty cat characterisation of monomorphisms classified by univalent morphism}, there is $H$ as in 
        \[\begin{tikzcd}
        	{x \times_y x} && e \\
        	& {e\times_b e}
        	\arrow["H", from=1-1, to=1-3]
        	\arrow["{q_f\times_{\chi_f}q_f}"', from=1-1, to=2-2]
        	\arrow["{(id,id)}", from=1-3, to=2-2]
        \end{tikzcd}\]
which by adjoint transposition corresponds to $\overline{H}$ fitting into a commutative diagram inducing a cone as in 
\[\begin{tikzcd}
	x & {\mathsf{el}(\omega_p)} & e \\
	y & {\omega_p} & b
	\arrow[dashed, from=1-1, to=1-2]
	\arrow["{q_f}", curve={height=-18pt}, from=1-1, to=1-3]
	\arrow["f"', from=1-1, to=2-1]
	\arrow["\lrcorner"{anchor=center, pos=0.125}, draw=none, from=1-1, to=2-2]
	\arrow[from=1-2, to=1-3]
	\arrow["\top"', from=1-2, to=2-2]
	\arrow["\lrcorner"{anchor=center, pos=0.125}, draw=none, from=1-2, to=2-3]
	\arrow["p", from=1-3, to=2-3]
	\arrow["{\overline{H}}", from=2-1, to=2-2]
	\arrow["{\chi_f}"', curve={height=18pt}, from=2-1, to=2-3]
	\arrow["{pr_p}", from=2-2, to=2-3]
\end{tikzcd}\]
then using our definition of limit and choosing a section to the trivial Kan fibration from the definition of limit, we get the induced dashed arrow between the cones. Then, the left square is a pullback by the pullback pasting lemma for $\infty$-categories.
\end{proof}

\begin{proposition}
\label{prop: top is a subobject classifier for infty cat}
    $\top \colon \mathsf{el}(\omega_p) \to \omega_p$ is a subobject classifier for the local class $S_p$, i.e.
    it represents the functor $\mathsf{Sub}_{S_p}(-)$ such that for any object $x$
    $$\mathpzc{Map}_{\cCi_{\infty}}(x, \omega_p) \cong \mathsf{Sub}_{S_p}(x).$$
\end{proposition}

\begin{proof}
    
    By Lemma \ref{lemma: monomorphisms of univalent morphisms are those of subobject classifier},  $S_{p}^{\mathsf{mon}}$ is precisely the local bounded class $S_{\top}$ of morphisms obtained as pullbacks of $\top$. Thus, $(\cCi/x)^{\mathsf{gpd}, S_{\top}} \cong \mathsf{Sub}_{S_p}(x)$ by Lemma \ref{lemma: groupoid of monos is 0-truncated}. By Proposition \ref{proposition: top is univalent infty cat}, $\top$ is univalent.  Thus, by Proposition \ref{prop: different characterisations of univalence in infty-cat} (2) we get 
    \[\pushQED{\qed} 
    \mathpzc{Map}(x,\omega_p) \simeq (\cCi/x)^{\mathsf{gpd}, S_{\top}} \cong \mathsf{Sub}_{S_p}(x).\qedhere
    \popQED\]
    \renewcommand{\qedsymbol}{}
\end{proof}

\begin{corollary}
\label{cor: every infty cat finitely complete lcc and univalent has subobject classifier}

    Every finitely complete, locally cartesian closed $\infty$-category has a subobject classifier $$\top \colon \mathsf{el}(\omega_p) \to \omega_p$$ for every univalent morphism $p$.
\qed\end{corollary}

Finally, as one would expect, a homotopy subobject classifier in a tribe becomes a subobject classifier in the $\infty$-localisation.

\begin{proposition}

    \label{prop: infty localisation of tribe with subobject classifier shas subobject classifiers}
    Let $\cC$ be a tribe and let $\top \colon \mathsf{El}(\Omega_p) \to \Omega_p$  be a homotopy subobject classifier for a local class $S_p$ induced by a fibration $p \colon E \to B$ in $\cC$ as in Theorem \ref{prop: constructed top is homotopy subobject classifier}. Let $\gamma \colon \cC \to \cC_{\infty}$ be the localisation functor. Then $\gamma(\top)$ is a subobject classifier for $S_{\gamma(p)}$ in $\cC_{\infty}$.
    
\end{proposition}
\begin{proof}
    Since the localisation sends a path object $P_p \to E \times_B E $ to a morphism equivalent to the diagonal $E \to E \times_B E$ and the localisation commutes up to homotopy with the right adjoint to pullback by Proposition \ref{prop:adjoint functors derive adjoint functors unit/counit}, we get that $\gamma(\Omega_p) \simeq \omega_{\gamma(p)}$. Since it also preserves pullbacks along fibrations $\gamma(\top)$ is a subobject classifier for $\gamma(p)$ in the $\infty$-category $\cC_{\infty}$.
\end{proof}

\subsection{Finite Colimits}
\label{sect: finite colimits}

\begin{proposition}[{\cite[Prop. 5.1.3.2 and Cor. 5.1.2.3]{lurie_higher_2009}}]
\label{proposition: colimit in infty-category}
    Let $d\colon \mathcal{I} \to\cCi$ be a diagram in an $\infty$-category $\cCi$. An object $x$ together with a natual transformation $d \Rightarrow \operatorname{const}_x$ is a \textit{colimit} for $d$ if and only if for all objects $x$, the map $$\Map(z,x) \to \lim_i\Map(d_i, x)$$ is an equivalence of spaces.
\qed\end{proposition}

\begin{lemma}\label{lemma:relation internal colimit and colimit in infty-cat}
    Let $d\colon \mathcal{I} \to\cCi$ be a diagram in a cartesian closed $\infty$-category and let $d \Rightarrow \operatorname{const}_z$ be a natural transformation. Suppose that $\cCi$ has limits of shape $\mathcal{I}$. Then, the canonical map $$\uMap(z,x) \to \lim_i\uMap(d_i,x)$$ is an equivalence in $\cCi$ if and only if $z$ is a colimit of $d$.
\end{lemma}

\begin{proof}
    This follows directly from $\Map(1,\uMap(-,x))\simeq\Map(-,x)$ and the fact that $\Map(1,-)$ preserves limits and that $\Map(-,x)$ sends colimits to limits.
\end{proof}

\begin{proposition}
\label{prop: localisation of internal homotopy pushout is homotopy pushout}
    Let $\cC$ be a $\pi$-tribe. Let $Q$ be an internal homotopy pushout of a diagram $A \xleftarrow{f} C \xrightarrow{g} B$ with the data $i_A, i_B$ and $H \colon i_A f \simeq i_Bg$. Then $Q$ is a pushout of $A \xleftarrow{\gamma (f)} C \xrightarrow{\gamma (g)} B$ in the localisation $\gamma\colon \cC \to \cC_{\infty}$.
\end{proposition}

\begin{proof}
    The localisation commutes with internal homs and preserves pullbacks along fibrations. Moreover, it sends path objects to morphisms equivalent to the diagonal. Thus, $$\uMap_{\cC_{\infty}}(Q,X) \simeq \gamma(\Cocone(f,g,X)) \simeq \uMap_{\cC_{\infty}}(A,X) \times_{\uMap_{\cC_{\infty}}(C,X)} \uMap_{\cC_{\infty}}(B,X)$$ for all objects $X$. So the claim follows from Lemma \ref{lemma:relation internal colimit and colimit in infty-cat}.
\end{proof}

\begin{theorem}
\label{thm: localisation of pi-tribe with internal homotopy initial object and internal homotopy pushouts is cocomplete}
    Let $\cC$ be a $\pi$-tribe with an internal homotopy initial object and internal homotopy pushouts. Then $\cC_{\infty}$ is finitely cocomplete.
\end{theorem}

\begin{proof}
    By \cite[Cor. 4.4.2.4]{lurie_higher_2009}, it suffices to show that $\cC_{\infty}$ has an initial object and pushouts.
    
    Let $I$ be the homotopy initial object of $\cC$. Since $\uHom(I,Z)$ is contractible for any object $Z$ and since the localisation commutes with the internal hom-functors, we get that $\uMap(I,Z) \simeq 1$. Since the limit of the empty diagram is the terminal object, we get by Lemma \ref{lemma:relation internal colimit and colimit in infty-cat} that $I$ is an initial object.

    Next, let $A \xleftarrow{f} C \xrightarrow{g} B$  be a span in $\cC_{\infty}$. By formula \cite[\;7.2.10.4]{cisinski_higher_2019}, we may assume that this span is the image of the localistaion of a zig-zag $$A \xleftarrow{} C' \xrightarrow{\simeq}  C \xleftarrow{\simeq} C'' \xrightarrow{} B$$
    where the two homotopy equivalences are in fact trivial fibrations. Choosing sections we obtain a span $A \xleftarrow{f'} C \xrightarrow{g'} B$ in $\cC$ whose image in $\cC_{\infty}$ is equivalent to the original span. Thus, this span has an internal homotopy pushout consisting of an object $Q$, morphisms $i_A\colon A\to Q$, $i_B\colon B\to Q$, and a homotopy $H$ witnessing that $i_Af'\simeq i_Bg'$ such that $\uHom(Q,D) \to\Cocone(f',g',D)$ is a homotopy equivalence for any object $D$. By Proposition \ref{prop: localisation of internal homotopy pushout is homotopy pushout}, $Q$ is a pushout of $A \xleftarrow{\gamma (f')} C \xrightarrow{\gamma (g')} B$ in $\cC_{\infty}$ which is therefore a pushout of the original span.

\end{proof}

%% file: sect5-elementary-infty-toposes.tex
\section{Elementary $\infty$-Toposes}
\label{sect: ele infty toposes}

\begin{definition}
\label{def: elementary infinity topos}

An \textit{elementary $\infty$-topos} is an $\infty$-category $\cCi$ such that

\begin{enumerate}
    \item $\cCi$ has all finite limits.
    \item $\cCi$ is locally cartesian closed.
    \item $\cCi$ has enough universal morphisms, i.e. for every morphism $f$, there is a closed local class $S_p$ that includes $f$ and a univalent morphism $p \colon \tilde{U}_{S_p} \to U_{S_p}$ classifying $S_p$.
\end{enumerate}
\end{definition}

\begin{theorem}
\label{theorem: every elementary infinity topos has subobject classifiers}
    
    Every elementary $\infty$-topos $\cCi$ has a subobject classifier $\omega_p$ for every univalent morphism $p$ in $\cCi$, classifying the monomorphisms of $S_p$.
\end{theorem}

\begin{proof}
    Since $\cC$ has finite limits, is locally cartesian closed and has enough universal morphisms, this follows from Corollary \ref{cor: every infty cat finitely complete lcc and univalent has subobject classifier}.
\end{proof}

\begin{theorem}
    \label{theorem: nice tribes present elementary infinity topoi}
    Let $\cC$ be a $\pi$-tribe with enough univalent fibrations. Let $\gamma \colon \cC \to \cC_{\infty}$ be the localisation functor. Then, $\cC_{\infty}$ is an elementary $\infty$-topos.

\end{theorem}

\begin{proof}
    $\cC_{\infty}$ has finite limits by \autocite[Prop. 7.6.16]{cisinski_higher_2019} and is locally cartesian closed by \autocite[Prop. 7.5.6]{cisinski_higher_2019}. By Corollary \ref{thm: tribe enough universal families then localisation enough univalent morphisms}, $\cC_{\infty}$ has enough univalent morphisms.
\end{proof}
    
\begin{theorem}
\label{thm: categorical models present infty toposes}
    Let $\cC$ be a categorical model of a dependent type theory $\T$ with $\Sigma, \Pi$ and $\mathtt{Id}$-structures, satisfying the $\Pi$-$\eta$ rule and function extensionality and with enough univalent universes, closed under all constructors. Then, $\cC_{\infty}$ is an elementary $\infty$-topos. In particular, the localisation of the syntactic category $\cC(\T)_{\infty}$ of such a type theory is an elementary $\infty$-topos. If moreover $\T$ satisfies propositional resizing, then in $\cC_{\infty}$ every monomorphism is the pullback of a single subobject classifier.
\end{theorem}

\begin{proof}
    We know that $\cC$ is a $\pi$-tribe and by Corollary \ref{cor: categorical model has enough univalent universes}, $\cC$ has enough univalent fibrations, so the claim follows from Theorem \ref{theorem: nice tribes present elementary infinity topoi}.

    Suppose now that $\T$ satisfies propositional resizing. For any subobject classifier $\mathsf{El}(\Omega_n) \to \Omega_n$, the homotopy commutative square of Lemma \ref{lemma: resizing in model gives almost homotopy unique pullbacks} will become a commutative square in $\cC_{\infty}$ in which the top and bottom morphisms are invertible. Thus, it will become in particular a pullback square. Then, by pullback pasting any monomorphism classified by some univalent morphism of the form $\gamma(p_n)$ with the corresponding subobject classifier $\Omega_n$ is also classified by $\Omega_0$. 
\end{proof}

\begin{theorem}
\label{thm: pushout types give finite colimits}
    If $\cC$ is a categorical model of dependent type theory with homotopy pushouts and 0-types, then $\cC_{\infty}$ has all finite colimits.
\end{theorem}

\begin{proof}
    This follows from Lemma \ref{lemma: categorical model has all internal homotopy pushouts} and Theorem \ref{thm: localisation of pi-tribe with internal homotopy initial object and internal homotopy pushouts is cocomplete}.
\end{proof}